\documentclass[msom,sglanonrev]{informs4}
\RequirePackage{tgtermes}
\RequirePackage{newtxtext}
\RequirePackage{newtxmath}
\RequirePackage{bm}
\RequirePackage{endnotes}
\OneAndAHalfSpacedXII % Current default line spacing
\usepackage{algorithm}
\usepackage{multirow}
\usepackage{booktabs}
\usepackage{algpseudocode}
\usepackage{tikz}
\usetikzlibrary{patterns}
\usetikzlibrary{decorations.pathreplacing}
\usepackage{amsmath}
\usepackage{etoolbox}
\AtBeginEnvironment{proposition}{\vspace{-5pt}} 
\AtEndEnvironment{proposition}{\vspace{-5pt}}
\AtBeginEnvironment{figure}{\vspace{-5pt}} 
\AtEndEnvironment{figure}{\vspace{-15pt}}
\AtBeginEnvironment{remark}{\vspace{-5pt}} 
\usepackage{hyperref}
\usepackage{ulem}

\newcommand{\citeapp}[1]{\cite{#1}}%for arxiv
\newcommand{\citepapp}[1]{\citep{#1}}%for arxiv

\newcommand{\ignore}[1]{}
\newcommand{\zzhdeltext}[1]{}%
\newcommand{\deltext}[1]{}%

\usepackage{natbib}
 \bibpunct[, ]{(}{)}{,}{a}{}{,}%
 \def\bibfont{\small}%
\EquationsNumberedThrough    % Default: (1), (2), ...
\TheoremsNumberedThrough     % Preferred (Theorem 1, Lemma 1, Theorem 2)
\ECRepeatTheorems  %  

\MANUSCRIPTNO{MSOM-0001-2024.00}

\begin{document}
%%%%%%%%%%%%%%%%

%\setcounter{page}{1}
%\DOI{10.1287/xxxx.0000.0000}%

\RUNAUTHOR{Xiangting Liu et al.}

% \RUNTITLE{Predictive Maintenance in Manufacturing}
\RUNTITLE{A non-parametric framework to tackle contextual dd-ccp}

% \TITLE{Optimal Resource Allocation in Humanitarian Logistics: A Stochastic Programming Approach}
\TITLE{Solving contextual chance-constrained programming under decision-dependent uncertainty}

\ARTICLEAUTHORS{%
%\AUTHOR{John Doe,\textsuperscript{a} Jane Smith,\textsuperscript{b}}
%\AFF{\textsuperscript{a}Department of Industrial Engineering, University of XYZ, \EMAIL{john.doe@xyz.edu; \textsuperscript{b}Department of Computer Science, University of ABC, \EMAIL{jane.smith@abc.edu}} 
\AUTHOR{Xiangting Liu, Shengran Wang, Kaile Yan, Zhi-Hai Zhang*}
\AFF{Department of Industrial Engineering, Tsinghua University,
\EMAIL{liu-xt22@mails.tsinghua.edu.cn}, 
\EMAIL{wang-sr22@mails.tsinghua.edu.cn},
\EMAIL{ykl22@mails.tsinghua.edu.cn}, 
\EMAIL{zhzhang@tsinghua.edu.cn}\URL{}}
} % end of the block

\ABSTRACT{%
% Enter your abstract
\textbf{Problem definition}: We study contextual chance-constrained programming under decision-dependent uncertainty. In this setting, a decision not only needs to satisfy constraints but also alters the distribution of uncertain outcomes. This dependency makes the problem particularly difficult: because feasibility probabilities vary with decisions, it creates both statistical endogeneity and computational intractability. 
\textbf{Methodology/results}: To address this, we propose a nonparametric approximation method based on Contextual Cluster Weights (CCW). For any given decision and context, CCW constructs a local neighborhood (cluster) of ``similar" historical observations and assigns them equal weight. This approach successfully renders both the objective and chance constraints tractable, while providing uniform-in-decision consistency guarantees. Furthermore, we develop reformulations that use pre-calculated clusters. We show that under a specific nestedness condition, these reformulations yield a convex feasible region, which allows for efficient solving. Experiments, including a case study with JD.com, demonstrate that our method outperforms benchmarks in solution quality, feasibility reliability, and runtime. 
\textbf{Managerial implications}: This framework offers a scalable and data-driven approach for firms to make reliable operational decisions when their actions influence uncertainty. It effectively balances performance, risk, and robustness, while remaining interpretable and implementable in practice.
}%

\FUNDING{}%This research is partially supported by the National Natural Science Foundation of China (Grant no. 72171129, 72188101, 92467302, 72250710683)}

%Supplemental Material:
%Data Ethics & Reproducibility Note:

%\KEYWORDS{Stochastic programming, Decision support,Uncertainty, Disaster response, Optimization}


\begin{center}
\Large {Solving contextual chance-constrained programming under decision-dependent uncertainty}
\end{center}
\vspace{0.2cm}
\begin{center}
\begin{small}
    Xiangting Liu, \\
    Department of Industrial Engineering, Tsinghua University, 100084, China\\
    Shengran Wang, \\
    Department of Industrial Engineering, Tsinghua University, 100084, China\\
    Kaile Yan, \\
    Department of Industrial Engineering, Tsinghua University, 100084, China\\
    Zhi-Hai Zhang*\\
    Department of Industrial Engineering, Tsinghua University, 100084, China, zhzhang@tsinghua.edu.cn
\end{small}
\end{center} 

\vspace{0.2cm}
\textbf{Abstract:} We study contextual chance-constrained programming under decision-dependent uncertainty. In this setting, a decision not only needs to satisfy constraints but also alters the distribution of uncertain outcomes. This dependency makes the problem particularly difficult: because feasibility probabilities vary with decisions, it creates both statistical endogeneity and computational intractability. To address this, we propose a nonparametric approximation method based on Contextual Cluster Weights (CCW). For any given decision and context, CCW constructs a local neighborhood (cluster) of ``similar" historical observations and assigns them equal weight. This approach successfully renders both the objective and chance constraints tractable, while providing uniform-in-decision consistency guarantees. Furthermore, we develop reformulations that use pre-calculated clusters. We show that under a specific nestedness condition, these reformulations yield a convex feasible region, which allows for efficient solving. Experiments, including a case study with JD.com, demonstrate that our method outperforms benchmarks in solution quality, feasibility reliability, and runtime. This framework offers a scalable and data-driven approach for firms to make reliable operational decisions when their actions influence uncertainty. It effectively balances performance, risk, and robustness, while remaining interpretable and implementable in practice.

\textbf{Key words: }{Contextual Optimization, Decision-dependent Uncertainty, Chance Constraint Programming, Pricing and Revenue Management} 
%\HISTORY{Received: Month DD, YYYY; Accepted: Month DD, YYYY; Published Online: Month DD, YYYY}

%\maketitle% for arxiv
%%%%%%%%%%%%%%%%%%%%%%%%%%%%%%%%%%%%%%%%%%%%%%%%%%%%%%%%%%%%%%%%%%%%%%

% Text of your paper here
\vspace{0.3cm}
%\hline
\vspace{0.5cm}

\section{Introduction}

Contextual Chance-Constrained Programming (CCCP) provides a principled, data-driven way to make operational decisions that are reliable with respect to observable contexts. In CCCP, one models the conditional distribution of uncertain outcomes given covariates and seeks prescriptions that optimize expected performance while ensuring a probabilistic feasibility guarantee \citep{shapiro2003monte, vinod2023sample}. Contextual variants integrate learning into risk control and provide conditional feasibility in applications such as supply chains \citep{YeChengHijaziVanHentenryck2025}, safe control \citep{Shirai2022ChanceConstrainedManipulation},facility location \citep{SaldanhaDaGama2024FacilityLocationChance},  among others. 

Formally, a CCCP problem can be expressed as
\begin{subequations}\label{eq-originalCSO}
    \begin{align}
    \min_{z \in \mathcal{Z}} \quad & \mathbb{E}_Y [l(z; Y) \mid X = x], \label{cccp-obj}\\
    \text{s.t.} \quad & \mathbb{P}\big(\psi_{\xi}(z, Y) \le 0 \,\big|\, X = x\big) \ge 1 - \alpha_{\xi}, \forall \xi \in \{1, \ldots, \Xi\}.\label{cccp-chanceconstraint}
    \end{align}
\end{subequations}
Here, \(z\in\mathcal{Z}\subseteq \mathbb{R}^n\) denotes the decision variables, where \(\mathcal{Z}\) is a compact and bounded value space. \(Y\) represents the random variables, \(x \in \mathcal{X}\) is the context and \(l(\cdot \mid \cdot)\) is the loss function, whose expectation is the objective function (\ref{cccp-obj}) to be optimized. Constraints (\ref{cccp-chanceconstraint}) are the chance constraints and $\Xi$ is the number of individual chance constraints. It ensures that the probability of the events defined by the constraint function \(\psi_{\xi}(\cdot)\) being satisfied is at least \(1-\alpha_{\xi}\), where \(\alpha_{\xi} \in (0, 1)\) is the pre-specified risk level for the constraint, representing the maximum allowable probability of constraint violation. In this paper, unless otherwise specified, Chance-Constrained Programming specifically refers to the programming problem with Individual Chance Constraint(s).

This field has received extensive attention in recent years. A central challenge in CCCP lies in the intractability of evaluating conditional feasibility probabilities, which has motivated a rich line of research on approximation and reformulation methods for chance constraints, including approaches with uniform convergence guarantees where the uncertainty is exogenous.

However, in many real applications, the underlying uncertainty is not exogenous but endogenous to the decision itself. For instance, demand depends on prices, arrival patterns hinge on service promises, and congestion levels rely on capacity choices. This decision-dependent uncertainty (DDU) fundamentally changes the nature of CCCP: the decision influences both the objective distribution and the feasibility probability, introducing endogeneity and non-convexity \citep{GoelGrossmann2004, Gallego1994, shapiro2003monte}. The resulting “1+1” setting—CCCP with DDU—captures a richer and more realistic class of problems central to data-driven operations \citep{blackhurst2007network, goda2018stochastic, soroudi2013decision}, but is substantially more challenging to approximate and solve.

Bringing decision-dependent uncertainty into CCCP introduces three intertwined challenges that fundamentally complicate both approximation and optimization.

First, approximation itself becomes decision-dependent. In contrast to exogenous uncertainty, the conditional distribution of outcomes must now be learned jointly over decisions and contexts, while chance constraints involve discontinuous indicator functions whose probabilities depend on the decision. As a result, standard predictive-to-prescriptive decision pipelines—where a predictive model of uncertainty is learned first and then embedded into an optimization problem—no longer apply directly, and must be refined to account for endogeneity between decisions and uncertainty \citep{ban2019big, huber2019data, oroojlooyjadid2020applying, bertsimas2020predictive}.

Second, theoretical guarantees must hold uniformly over the decision space. In decision-dependent settings, data-driven approximations of objectives and feasibility probabilities are evaluated at decision points that are selected by the optimization procedure itself. Pointwise or average-case guarantees are therefore insufficient, as the optimizer may select decisions where approximation errors are largest. Ensuring uniform-in-decision consistency is thus essential for the reliability of both learned objectives and chance constraints \citep{lin2022data, CutlerDiazDrusvyatskiy2024, HeuserKesselheim2025}.

Third, the resulting optimization problems pose significant computational challenges. Decision-dependent chance constraints can induce feasibility regions that are non-convex or even discontinuous, and naive sample-average or scenario-based formulations typically lead to large-scale mixed-integer programs that are intractable in practice. This motivates the need for tractable reformulations and decomposition-based solution methods \citep{pagnoncelli2009sample, VanParys2021, ZhuYuBayraksan2024}.

To address these challenges, we develop a non-parametric approximation and solution framework for contextual chance-constrained programming under decision-dependent uncertainty (CCCP-DDU). The main contributions are threefold.

First, we identify and formalize Contextual Cluster Weights (CCW) as a pivotal class of non-parametric estimators. While general wSAA methods struggle with the discontinuity of chance constraints under DDU, CCW leverages a ``set-based'' weighting logic. This structure allows us to bridge the gap between machine learning clusters and optimization tractability, enabling the first derivation of uniform-in-decision consistency and sample efficiency for contextual chance constraints where decisions actively shift the uncertainty distribution.

Second, we develop a scalable solution framework for the resulting CCCP-DDU models. We propose a reformulation that significantly reduces the computational burden of CCW-based chance constraint approximations. Furthermore, under an additional structural condition that induces convexity of the approximated feasible region in downstream decisions, we show that the problem admits an efficient solution via Benders decomposition.

Third, we demonstrate the practical effectiveness of the proposed framework through extensive numerical experiments and a real-world case study based on transaction-level data from JD.com. The results show that our approach consistently outperforms parametric benchmarks and existing non-parametric methods in terms of solution quality, feasibility reliability, and computational efficiency in settings with endogenous uncertainty.

Together, these contributions advance both the theory and practice of contextual chance-constrained optimization by providing a principled, scalable, and data-driven approach to problems where decisions influence uncertainty.

The rest of the paper is structured as follows. Section \ref{sec::LR} reviews relevant Literature. Section \ref{sec::preliminaries} introduces the problem of CCCP-DDU and outlines the approximation with CCW. Section \ref{sec::Asymptotic} presents the assumptions and provides the asymptotic results. Section \ref{sec::Solution} provides a solution framework to tackle the resulting model. Section \ref{sec::PSNP} applies this framework to the price-setting newsvendor problem as an example. Section \ref{sec::NumericalExperiment} validates our method through extensive numerical experiments. Section \ref{sec::case} illustrates the practical relevance of our approach with a real-world case study. Finally, Section \ref{sec::conclusion} concludes the paper and outlines avenues for future research.

\section{Relevant Literature}\label{sec::LR}
We first review the literature on basic contextual stochastic optimization (CSO), which is the most relevant research stream to the present work. Then we provide a review about its two extensions with chance constraint and decision-dependent uncertainty. CSO is a framework that integrates observable context into decision-making under uncertainty. CSO adapts to heterogeneity in both contexts and uncertainty by conditioning decisions on features or covariates, making it a powerful tool for real-world applications.

Recent work in CSO has developed along three major paradigms \citep{Sadana2025ContextualOptimizationSurvey}: decision rule optimization, sequential learning and optimization (SLO), and integrated learning and optimization (ILO). Decision rule optimization focuses on deriving direct decision-making policies, using methods such as linear policies \citep{ban2019big}, decision tree ensembles \citep{huber2019data}, and neural networks \citep{oroojlooyjadid2020applying}. As an extension, Operational Data Analysis framework proposed by \cite{feng2023framework} selects a best statistic as the decision, which is applied to solve the contextual newsvendor problem \citep{feng2025contextual}. SLO aims to minimize the discrepancy between the approximated distribution of uncertainty and its true distribution \citep{bertsimas2020predictive}, whereas ILO aims to optimize the integrated objective function, combining learning and optimization tasks.

These advances have paved the way for two major extensions of CSO: CSO with chance constraints and CSO with decision-dependent uncertainty.

\subsection{CSO with Chance Constraints (CCCP)}

% A contextual stochastic optimization problem with chance constraints (CSO–CC) can be formulated as follows:
% \setlength{\abovedisplayskip}{-2pt}
% \setlength{\belowdisplayskip}{-1pt}
% \begin{equation}
%     \begin{aligned}
%     \min_{z \in \mathcal{Z}} \quad & \mathbb{E}_Y [c(z; Y) \mid X = x], \\
%     \text{s.t.} \quad & \mathbb{P}\big(\psi_{\xi}(z, Y) \le 0 \,\big|\, X = x\big) \ge 1 - \alpha_{\xi},\\
%     &\quad\quad\quad \forall \xi \in \{1, \ldots, \Xi\}.
%     \end{aligned}
% \end{equation}

% The key feature of this problem is the chance constraints: \(\mathbb{P}(\psi_{\xi}(z, Y) \le 0) \ge 1-\alpha_{\xi}\). It ensures that the probability of the events defined by the constraint function \(\psi_{\xi}(\cdot)\) being satisfied is at least \(1-\alpha_{\xi}\), where \(\alpha_{\xi} \in (0, 1)\) is the pre-specified risk level for the constraint, representing the maximum allowable probability of constraint violation.
CCCP frameworks have been widely applied in diverse domains, including supply chain design \citep{Liu2021FacilityLocationChance, SaldanhaDaGama2024FacilityLocationChance}, robotic control under uncertainty \citep{Shirai2022ChanceConstrainedManipulation}, and safe reinforcement learning \citep{Peng2021CCAC}.

% This framework enables the decision-making process to adapt not only to the uncertainty in the system but also to the contextual factors that affect uncertainty.

Traditional chance-constrained optimization, as introduced by \cite{shapiro2003monte}, provides a framework for enforcing probabilistic feasibility under uncertainty. Classical methods include the sample-average approximation \citep{pagnoncelli2009sample}, the scenario approach \citep{calafiore2006scenario}, and the quantile-based formulation \citep{vinod2023sample}. While traditional chance-constrained methods have been successfully applied in many practical settings, most of them do not possess uniform convergence, which is solved by the inclusion of CSO.

CCCP research has advanced by integrating contextual information into probabilistic feasibility guarantees, focusing on two complementary directions. The first direction emphasizes contextual distribution approximation. For instance, \cite{lin2022data} verified uniform convergence of the chance constraint with wSAA under exogenous uncertainty distributions. \cite{TianJiangPangGuoJinWang2024} approximated conditional distributions using standard machine learning regressors, while \cite{YeChengHijaziVanHentenryck2025} incorporated context-dependent forecasts into large-scale operational decision-making. More recently, \cite{QinZouLiuLiu2025} proposed piecewise affine decision rules that map contextual features to decision variables within chance-constrained optimization.

The second direction addresses learning-for-optimization with theoretical guaranties.  \cite{HeuserKesselheim2025} derived polynomial sample complexity results for learning contextual value distributions to achieve near-optimal policies. Additionally, \cite{PacchianoGhavamzadehBartlett2025} investigated contextual bandits with stage-wise high-probability constraints, offering algorithms balancing exploration and per-round constraint satisfaction. In batch contextual settings, \cite{rahimian2023data} provided finite-sample feasibility guarantees for data-driven contextual chance-constrained programs, while \cite{vinod2023sample} derived explicit sample-complexity and tightening bounds for a sample quantile-based approximation of non-convex separable chance constraints.

\subsection{CSO with Decision-dependent Uncertainty}

A contextual stochastic optimization problem with decision-dependent uncertainty (CSO–DDU) can be formulated as follows:
\begin{equation}\label{eq-CSODDU}
    \min_{z \in \mathcal{Z}} \mathbb{E}_Y[l(z; Y(z)) \mid X = x],
\end{equation}
with the key difference to \eqref{cccp-obj} being that the uncertainty $Y$ is now influenced by decision $z$. Specifically, we focus on the setting where the decision alters the underlying probability distribution itself (i.e., $Y \sim f(\cdot|z, x)$), often referred to as 'Type I' decision-dependent uncertainty \citep{jonsbraten1998class, GoelGrossmann2004}.

This kind of DDU exists in many real-world applications, including supply chain management \citep{blackhurst2007network}, energy systems \citep{soroudi2013decision}, finance \citep{goda2018stochastic}, and price-setting newsvendor problem \citep{liu2023solving}.  It complicates both the modeling and the optimization processes.

To address CSO-DDU, researchers have developed both parametric and non-parametric approaches. Parametric models assume a pre-specified functional form between decisions, covariates and uncertain parameters \citep{young1978price, Gallego1994, GoelGrossmann2004}. While computationally tractable, these methods rely heavily on correct model specification, and misspecification can lead to poor decisions.

% However, incorporating DDU into wSAA formulations introduces substantial computational challenges. The feasible region often becomes non-convex, making optimization NP-hard \citep{Shapiro2021}. \cite{bertsimas2020predictive} provided convergence guarantees for decision-dependent prescriptions, whereas \cite{VanParys2021} highlighted robustness mechanisms through distributionally robust reformulations. While consistency can be established under mild regularity conditions, examples exist where uniform convergence fails. When the loss function $c$ (or inner function) is discontinuous or non-smooth, conventional stochastic approximation or sample-average-based methods cannot guarantee convergence. Each decision may correspond to a distinct, irregular conditional law, making expectations ill-defined or unstable. Computational cost is another obstacle, as decomposition techniques like Benders decomposition or cutting-plane methods scale poorly with problem size \citep{Luedtke2008}.

In contrast, non-parametric approaches aim to relax structural assumptions. \cite{bertsimas2020predictive} proposed the wSAA, arguably the most well-known non-parametric method, which conditioned on observed features or decision variables to approximate the conditional distribution of uncertain outcomes. \cite{WangTian2023} extended this approach by analyzing finite-sample performance gaps of wSAA under limited data, illustrating practical limitations in high-dimensional settings. Beyond wSAA, distributionally robust optimization \citep{ZhuYuBayraksan2024} and machine learning-based methods \citep{HikimaTakeda2025} are also commonly employed.

However, these existing contributions primarily focus on optimizing expected objectives (e.g., minimizing expected cost). Extending these non-parametric frameworks to enforce chance constraints under DDU introduces fundamental challenges. The feasible region often becomes non-convex, making optimization NP-hard. While consistency can be established for continuous objectives under mild regularity conditions \citep{bertsimas2020predictive, VanParys2021}, these guarantees typically fail for chance constraints.

The core conflict is twofold. First, the indicator function $\mathbb{I}(\psi_{\xi}(z, Y) \le 0)$ inherent in CCCP is discontinuous. This violates the smoothness assumptions, rendering uniform convergence proofs inapplicable. Second, computationally, approximating these discontinuous functions under DDU requires coupling binary variables with decision-dependent uncertainty. This transforms the problem into a large-scale Mixed-Integer Non-Linear Program (MINLP), which renders standard decomposition techniques like Benders or cutting-plane methods ineffective or poorly scalable.

Consequently, CSO-DDU remains an underdeveloped area that requires (i) novel techniques to handle discontinuity, and (ii) tractable reformulations to bypass combinatorial intractability.

We aim to bridge these gaps by developing a general, non-parametric solution framework for CCCP-DDU. The proposed approach targets the core difficulties of discontinuity, convergence reliability, and computational efficiency, providing both theoretical justification and practical tractability for decision-dependent stochastic systems, which generally belongs to the ILO framework.

% Our work contributes to filling this gap by proposing a novel, non-parametric, data-driven framework for solving decision-dependent optimization problems. We extend the work of \cite{bertsimas2020predictive}, focusing on both the theoretical guarantees of convergence and practical solution methods that can scale efficiently in large, complex decision spaces.

% \subsection{Contribution}

\section{Preliminaries}\label{sec::preliminaries}

In this section, we detail the problem setup, provide the mathematical formulation of the problem, and state how we approximate the objective and the chance constraint.

The CCCP-DDU considered is formulated as follows.
\begin{subequations}\label{eq-True-ddccp}
    \begin{align}
    \min_{z \in \mathcal{Z}}&\quad L(z\mid X=x)= E_{Y\sim f(\cdot \mid z, x)}[l(z, Y)],\label{eq-obj}\\
    s.t.&\quad g_{\xi}(z \mid X=x) = \mathbb{P}_{Y\sim f(\cdot \mid z, x)}(\psi_{\xi}(z, Y) \le 0) \ge 1-\alpha_{\xi}, \forall \xi \in \{1,\cdots,\Xi\},\label{eq-cc}
\end{align}
\end{subequations}
where \(\mathcal{Z}\) is the value space for the decision variable \(z\), \(f(\cdot \mid z, x)\) is the probability density function for \(Y\) given decision variable \(z\) and covariates \(x\), and \(l(\cdot)\) is the loss function.

Following wSAA proposed by \cite{bertsimas2020predictive}, we construct a data-driven approximation of the objective function $L$ based on a historical dataset $S_N={(z_1,x_1,y_1), (z_2,x_2,y_2), \cdots, (z_N,x_N,y_N)}$, which is formulated as follows.
%In order to approximate the objective function based on a historical dataset \(S_N=\{(z_1, x_1, y_1),\) \((z_2, x_2, y_2), \cdots, (z_N, x_N, y_N)\}\), we construct a data-driven approximation of $L$ by wSAA, which is,
\begin{equation} \label{eq-Appr-Obj}
    \hat{L}(z \mid X = x) = \sum_{i = 1}^N w_i(z, x) l(z; y_i),
\end{equation}
where $w_i(\cdot, \cdot): \mathcal{Z}\times\mathcal{X}\rightarrow \mathbb{R}$ is a weight function that describes the similarity between the context of historical data $i$ and the input context. $\sum_{i=1}^N w_i(z,x) = 1$ and $w_i(z,x) \ge 0$.

Similarly, we approximate the feasible region of \eqref{eq-cc} with \textbf{Appr-CC}, namely 
%$\hat{P}(x) = $
\begin{equation}\label{Est-cc}
\begin{aligned}
    \hat{P}(x) = \biggl\{z\in \mathcal{Z}: \hat{g}_{\xi}(z\mid X=x) = \sum_{i=1}^N w_i(z,x) \mathbb{I}\left\{\psi_{\xi}(z,y_i) \le 0\right\} \ge 1-\alpha_{\xi},
    \forall \xi \in \{1,\cdots,\Xi\}\biggr\},
\end{aligned}
\end{equation}
where $\hat{g}_{\xi}(z\mid X=x)$ is the approximation of probability, and $\mathbb{I}(\cdot)$ is an indicator function. We refer to the model with objective \eqref{eq-Appr-Obj} and constraint \eqref{Est-cc} as \textbf{Appr-dd-ccp}. 

A fundamental challenge in CCCP-DDU is that general weight functions often fail to provide the uniform convergence required when the feasible region itself is decision-dependent. To address this, we restrict our focus to a specialized category of weight functions, which we term Contextual Cluster Weights (CCW). It corresponds to a specific category of weight functions and can be collectively written in the following manner.
\begin{equation}
    w_i(z,x) = \frac{\mathbb{I}\{(z_i, x_i) \in \mathcal{C}(z, x)\}}{|\mathcal{C}(z, x)|},
\end{equation}
where $\mathcal{C}(z, x)$ denotes the local neighborhood set constructed from $S_N$ for a specific query point $(z, x)$, and $|\mathcal{C}(z, x)|$ denotes the cardinality of $\mathcal{C}(z, x)$. This set contains the data points that are close to or share characteristics similar to the current context $(z, x)$ according to predefined rules. Data points in $\mathcal{C}(z, x)$ are assigned the same positive weight of $1/|\mathcal{C}(z, x)|$, while the remaining points receive a weight of zero.

Unlike generic weighting schemes, CCW maps the joint decision-context space into a finite collection of local neighborhoods (clusters). This `all-or-nothing' weighting logic is not merely a simplification; it is a strategic choice that enables us to:
\begin{itemize}
    \item Uniformize Convergence: By leveraging the finite VC-dimension of cluster-based indicators, we obtain uniform-in-decision consistency guarantees that are otherwise unattainable for discontinuous probability operators. We detail this convergence in Section \ref{sec::Asymptotic}.
    \item Structural Decoupling: It allows the separation of cluster identification from the continuous optimization of uncertainty-independent decisions, transforming a non-convex MINLP into a tractable structured program. We detail this decoupling in Section \ref{sec::Solution}.
\end{itemize}

A subset of the well-established weight functions introduced in \cite{bertsimas2020predictive} can be categorized as CCW. 

\begin{definition}[kNN weight]
    The k-Nearest Neighbor (kNN) weight function is given by,
    \begin{equation}\label{eq-kNN}
    w_i^{kNN}(z, x) = \frac{1}{k} \mathbb{I}\left\{(z_i, x_i) \text{ is a kNN of } (z, x)\right\},
    \end{equation}
    where \((z_i, x_i)\) is a kNN of \((z, x)\) if \(dist((z_i, x_i),(z, x))\) is one of the smallest $k$ numbers in \(\{dist((z_1, x_1),(z, x)), \cdots, dist((z_N, x_N),(z, x))\}\), and \(dist((z_i, x_i),(z, x))\) is the distance between \((z_i, x_i)\) and \((z, x)\). By this function, the \(k\)-nearest historical data points of \((z, x)\) have the same positive weight $1/k$, while other data points have the weight of 0.
\end{definition}

\begin{definition}[CART weight]
    The Classification and Regression Tree (CART) weight function is given by,
    \begin{equation}\label{eq-CART}
    \begin{aligned}
    &w_i^{CART}(z, x)= \frac{\mathbb{I}\!\left\{\mathcal{R}_{CART}(z, x)
    =\mathcal{R}_{CART}(z_i, x_i)\right\}}
    {\sum_{j=1}^N
    \mathbb{I}\left\{\mathcal{R}_{CART}(z, x)
    =\mathcal{R}_{CART}(z_j, x_j)\right\}},
    \end{aligned}
    \end{equation}
    where \(\mathcal{R}_{CART}\) is a map from \(\mathcal{Z}\times\mathcal{X}\) to \(\{1,\cdots,r\}\), the \(r\) different leaf nodes of a CART trained by $S_N$. \(\mathcal{R}_{CART}(z, x)\) returns the corresponding leaf node index given \((z, x)\). By this function, the historical data points that fall into the same leaf node as \((z, x)\) will have the same positive weight, while other data points will have the weight of 0.
\end{definition}

\begin{definition}[LSA weight]
    The Local Sample Average (LSA) weight function is given by,
    \begin{equation}\label{eq-LSA}
        w_i^{LSA}(z,x) = \frac{\mathbb{I}\left\{dist((z_i, x_i), (z,x)) \le h\right\}}{\sum_{j=1}^N \mathbb{I}\left\{dist((z_j, x_j), (z,x)) \le h\right\}},
    \end{equation}
    where \(h\) is a pre-determined bandwidth. By this function, the historical data points within a distance of \(h\) from \((z, x)\) will have the same weight, while the other will have a weight of 0. 
    % This weight is a special case of Kernel weight stated in \cite{bertsimas2020predictive}.
\end{definition}

Considering the specific weight functions, the weighted sample average approximation with CCW-hereafter referred to as the CCW approximation-exhibits excellent properties, especially when approximating expectations of discontinuous, bounded functions like probability. We explain this in the following section.

\section{Asymptotic Results}\label{sec::Asymptotic}

We first present the necessary assumptions and then establish the asymptotic properties of the CCW approximation.

\subsection{Assumptions}

In this section, we introduce the necessary assumptions used in the asymptotic analysis.

\begin{assumption}[Decomposition and Ignorability]
\label{ass::decomp&Ign}
Decision variables can be decomposed into $z=(z_1, z_2)\in \mathcal{Z}=\mathcal{Z}_1\times\mathcal{Z}_2$ where only $z_1\in \mathcal{Z}_1$ affects uncertainty, namely $Y(z_1, z_2)=Y(z_1, z_2')$ for all $(z_1, z_2), (z_1, z_2')\in \mathcal{Z}$. Here $\mathcal{Z}_1$ and $\mathcal{Z}_2$ are the value spaces for $z_1$ and $z_2$, respectively. For every $z \in \mathcal{Z}$, $Y(z)$ is independent of $Z$ conditioned on $X$. $\mathcal{Z}_1$, $\mathcal{Z}_2$ and $\mathcal{X}$ have a dimension of $d^{z_1}$, $d^{z_2}$ and $d^x$,  respectively. For brevity, we write $\mathcal{C}(z_1, z_2, x)=\mathcal{C}(z_1, x)$, $w(z_1, z_2, x)=w(z_1, x)$.
\end{assumption}

Assumption \ref{ass::decomp&Ign} follows the setting in \cite{bertsimas2020predictive}, ensuring that only part of the decisions affects uncertainty and that contextual information $X$ captures all historical dependence. $z_1$ and $z_2$ are referred to as the uncertainty-affecting decision and the uncertainty-independent decision, respectively.

\begin{assumption}[Equi-continuous]
\label{ass::continuous}
Both $l(z;y)$ and $\psi_{\xi}(z;y)$ are equicontinuous in $z$ for all $\xi\in \{1,\cdots,\Xi\}$. Namely, for any $z \in \mathcal{Z}$, $y\in\mathcal{Y}$, and $\epsilon>0$, there exists $\delta>0$ such that $|l(z;y) - l(z';y)|\le \epsilon$ for all $z'$ satisfying $\Vert z-z'\Vert\le\delta$. Likewise for $\psi_{\xi}(z;y)$.
\end{assumption}

\begin{assumption}[Regularity]
\label{ass::regularity}
The following conditions hold:
\begin{enumerate}
    \item $\mathcal{X}$, $\mathcal{Y}$, and $\mathcal{Z}$ are compact and bounded;
    \item $(z_i, x_i, y_i)$ are i.i.d. samples from a joint distribution $f(z,x,y)$;
    \item The joint density $f(z,x,y)$ is continuous and bounded away from zero, namely $0 < c < f(z,x,y) < C$, and $f(z, x) = \int_{\mathcal{Y}} f(z,x,y) dy$ is twice continuously differentiable.
\end{enumerate}
\end{assumption}

Assumptions \ref{ass::continuous} and \ref{ass::regularity} ensure that small perturbations in decisions or contexts do not induce abrupt changes in the conditional distribution of uncertainty, the loss or the probability. These conditions are satisfied in many parametric and non-parametric demand and arrival models commonly used in operations management \citep{bertsimas2020predictive, lin2022data}.

\begin{assumption}[Existence]
\label{ass::existence}
$\hat{P}(x)$ is non-empty and compact for all $x$.
\end{assumption}

This guarantees that the approximated feasible region is well-defined. In practice,  $\hat{P}(x)$ can be explicitly verified and maintained by calibrating the risk level $\alpha_\xi$.

\begin{assumption}[Bounded tails]
\label{ass::boundedtail}
There exists $\lambda>0$ such that
\begin{align}
    &\sup_{z \in \mathcal{Z},x\in\mathcal{X}} E\biggl[exp(\lambda |l(z;Y) -E[l(z;Y)\mid X=x, Z_1=z_1]|) \mid X=x, Z_1=z_1\biggr] < \infty.
\end{align}
\end{assumption}

This light-tail condition holds for many practical distributions such as Gaussian, ensuring finite exponential moments.

\begin{assumption}[Slater condition]
\label{ass::slater}
For any given $x$, there exists an optimal $z^*$ of \eqref{eq-True-ddccp} such that $\forall \epsilon>0$, there exists $z\in \mathcal{Z}$ with $\Vert z-z^* \Vert \le \epsilon$ and $g_{\xi}(z \mid X=x) > 1-\alpha_{\xi}$ for all $\xi=1,\cdots,\Xi$.
\end{assumption}

This assumption ensures the feasibility of the chance constraint. It holds unless there is only one feasible solution, which can be avoided by adjusting $\psi_{\xi}(z;y)$ or $\alpha$. With Assumptions \ref{ass::decomp&Ign} to \ref{ass::slater}, we proceed to analyze the consistency and optimality of the proposed CCW approximation method.

\subsection{Asymptotic consistency}

This section respectively gives the asymptotic results of the CCW approximation for the objective function \eqref{eq-obj} and the chance constraint \eqref{eq-cc}. We first state the definition of uniform consistency.

\begin{definition}[Uniform-in-decision Consistency]
\label{def::uniform consistency}
    An approximation for $l(\cdot;Y)$ as in (\ref{eq-Appr-Obj}) has strongly uniform-in-decision consistency, hereinafter referred to as strongly uniform consistency, if for any $x$,
    \begin{equation}\label{eq-consistency}
    \sup_{z\in\mathcal{Z}}\left|L(z\mid X=x) - \hat{L}(z\mid X=x) \right|\xrightarrow[]{N\to\infty} 0
    \end{equation}
    almost surely. And $\hat{L}(\cdot\mid \cdot)$ has weakly uniform consistency if (\ref{eq-consistency}) holds with probability. Definition for uniform consistency between $\hat{g}(z\mid X=x)$ and $g(z\mid X=x)$ is likewise.
\end{definition}

Based on Definition \ref{def::uniform consistency}, we provide the uniform consistency of $\hat{L}(\cdot\mid \cdot)$ in Lemmas \ref{lem::kNN}, \ref{lem::CART} and \ref{lem::LSA}.
\begin{lemma}[kNN consistency]
%(Lemma 5 in \cite{bertsimas2019predictions})
\label{lem::kNN}
    Suppose Assumptions \ref{ass::continuous}, \ref{ass::regularity} and \ref{ass::boundedtail} hold. For some $0<C<1, 0.5<\delta_{kNN}<1$, let $k=\lceil CN^{\delta_{kNN}}\rceil$, then $\hat{L}(z\mid X=x)$ based on \eqref{eq-kNN} weight has strongly uniform consistency.
\end{lemma}

\begin{lemma}[CART consistency]
\label{lem::CART}
    Suppose Assumptions \ref{ass::continuous}, \ref{ass::regularity} and \ref{ass::boundedtail} hold. Suppose an honest, regular, random-split tree is constructed, with the minimum number of training examples in each leaf being $\min \{\lceil CN^{\delta_{cart}}\rceil,N-1 \}$ for some $C>0, 0<\delta_{cart}<1$. Suppose $L(z \mid X=x)$ is $L_1$-Lipschitz continuous, and the covariate $x$ is uniformly distributed random variable on $[0,1]^{d^x}$. Then $\hat{L}(z\mid X=x)$ based on \eqref{eq-CART} weight has weakly uniform consistency.
\end{lemma}

\begin{lemma}[LSA consistency]
\label{lem::LSA}
    Suppose Assumptions \ref{ass::continuous} and \ref{ass::regularity} hold. And $E[|l(z;Y)\max\{log|(l(z;Y))|, 0\}|] <\infty$. For some $C>0, 0<\delta_{LSA}<1/(2d+1)$ where $d=d^{z_1}+d^{x}$, let $h_N=CN^{-\delta_{LSA}}$, then $\hat{L}(z\mid X=x)$ based on \eqref{eq-LSA} weight has strongly uniform consistency.
\end{lemma}

% Lemma~\ref{lem::kNN} and Lemma~\ref{lem::CART} are derived from \cite{bertsimas2019predictions} with some modifications to account for decision-dependent uncertainty, while Lemma~\ref{lem::LSA} follows directly from \cite{bertsimas2020predictive}.
These lemmas demonstrate asymptotic consistency of $\hat{L}(\cdot\mid\cdot)$ for continuous loss functions that meets the requirement of Assumptions \ref{ass::continuous} and \ref{ass::regularity}, but their direct application to the chance constraints is hindered by the discontinuity of the indicator function $\mathbb{I}\{\psi_{\xi}(z;y)\le0\}$.

% With the above Lemmas, decision-makers can use these three weights to construct the wSAA approximation of the objective function. However, since $\mathbb{I}(\psi_{\xi}(z;y) \le 0)$ is not equicontinuous in $z$, Assumption \ref{ass::continuous} fails when the above weights are directly applied to probability approximation, namely $g_{\xi}(z\mid X=x) = E_{Y\sim f(\cdot \mid z, x)} \left[\mathbb{I}(\psi_{\xi}(z;y) \le 0)\right]$.

To state the consistency of $\hat{g}(z|X=x)$, which equals
\[ \sum_{i=1}^N w_i(z,x) \mathbb{I}\left\{\psi_{\xi}(z,y_i) \le 0\right\}=\frac{\sum_{i=1}^N \mathbb{I}\left\{(z_i, x_i) \in \mathcal{C}(z, x)\right\} \cdot\mathbb{I}\{\psi_{\xi}(z, y_i)\le 0\}}{\sum_{i=1}^N \mathbb{I}\left\{(z_i, x_i) \in \mathcal{C}(z, x)\right\}},\] 
we firstly analyze the consistency of the numerator and denominator. The denominator function and its approximation are defined as follows, respectively.
\begin{equation}
    \begin{aligned}
        U(z,x)&=E[\mathbb{I}\left\{(Z, X) \in \mathcal{C}(z, x)\right\}]=P\left((Z, X) \in \mathcal{C}(z, x)\right),\\
        \hat{U}(z,x) &= \frac{1}{N}\sum_{i=1}^N \mathbb{I}\left\{(z_i, x_i) \in \mathcal{C}(z, x)\right\}.   
    \end{aligned}
\end{equation}

And the numerator function and its approximation are respectively defined below, for all $\xi \in \{1,\cdots, \Xi\}$,
\begin{equation}
    \begin{aligned}
        V_{\xi}(z,x)&=P\!\left[((Z, X)\! \in\! \mathcal{C}(z, x))\&(\psi_{\xi}(z,Y) \!\le\! 0)\right],\\
        \hat{V}_{\xi}(z,x) &= \frac{1}{N}\sum_{i=1}^N \mathbb{I}\left\{(z_i, x_i) \in \mathcal{C}(z, x)\right\} \cdot\mathbb{I}\{\psi_{\xi}(z, y_i)\le 0\}.
    \end{aligned}
\end{equation}
% When adding a specific superscript to the function name, we are referring to this specific type of approximation.

The parts within the summation symbol in functions $\hat{V}_{\xi}$ and $\hat{U}$ can be regarded as indicators determined by $(z, x)$. This enables us to discuss the consistency of the numerator and denominator separately using the Vapnik-Chervonenkis (VC) Theorem. We recall this theorem in E-Companion \ref{app::Proof}.

\begin{proposition}[Consistency of the denominator and numerator]
\label{prop::DNConsistency}
    When the approximation is defined as in Lemmas \ref{lem::kNN} and \ref{lem::LSA}, we have for kNN and LSA, the indicator function class $\mathcal{I} = \left\{ \mathbb{I}_{(z,x)} \mid \mathbb{I}_{(z,x)}(Z,X) = \mathbb{I}\{(Z, X) \in \mathcal{C}(z, x)\}, (z,x) \in \mathcal{Z} \times \mathcal{X} \right\}$ has finite VC-dimension, and for all $\xi \in \{1,\cdots, \Xi\}$,
\begin{align}
\sup_{z\in\mathcal{Z}, x\in\mathcal{X}}\left|\hat{U}^{kNN}(z,x) -U^{kNN}(z,x)\right| &= O_p\left(\sqrt{\tfrac{\log N}{N}}\right),\\
\sup_{z\in\mathcal{Z}, x\in\mathcal{X}}\left|\hat{V}_{\xi}^{kNN}(z,x) -V_{\xi}^{kNN}(z,x)\right|&= O_p\left(\sqrt{\tfrac{\log N}{N}}\right),\\
\sup_{z\in\mathcal{Z}, x\in\mathcal{X}}\left|\hat{U}^{LSA}(z,x) -U^{LSA}(z,x)\right| &= O_p\left(\sqrt{\tfrac{\log N}{Nh_N}}\right),\\
\sup_{z\in\mathcal{Z}, x\in\mathcal{X}}\left|\hat{V}_{\xi}^{LSA}(z,x) -V_{\xi}^{LSA}(z,x)\right|&= O_p\left(\sqrt{\tfrac{\log N}{Nh_N}}\right).
\end{align}
\end{proposition}
% We recall the definition of VC theorem in Appendix E.

Proposition \ref{prop::DNConsistency} states that when using kNN and LSA in $\hat{g}_{\xi}$, both  $\hat{V}_{\xi}$ and $\hat{U}$ converge to their actual counterparts $V_{\xi}$ and $U$, respectively. Next, we present the uniform consistency property of the fraction $V_{\xi}/U$ in Proposition \ref{prop::Frac Consistency}.

\begin{proposition}[Consistency of the fraction]
\label{prop::Frac Consistency}
    When the approximation is defined as the way Lemmas \ref{lem::kNN} and \ref{lem::LSA} mention, we have for kNN and LSA, for all $\xi \in \{1,\cdots, \Xi\}$,
    \begin{equation}
    \sup_{z\in\mathcal{Z},x\in\mathcal{X}}\left|\frac{\hat{V}_{\xi}(z,x)}{\hat{U}(z,x)} - \frac{V_{\xi}(z,x)}{U(z,x)}\right| \xrightarrow[N\to\infty]{a.s.} 0.
    \end{equation}
\end{proposition}

With proposition \ref{prop::Frac Consistency}, Theorem \ref{theo::chance-consistency} is obtained.

\begin{theorem}[Consistency of estimating chance constraint with CCW]
\label{theo::chance-consistency}
    When the approximation is defined as the way Lemmas \ref{lem::kNN} to \ref{lem::LSA} mention, we have that $\forall$$\xi \in \{1,\cdots, \Xi\}$, $\hat{g}_{\xi}(z\mid X=x)$ with kNN and LSA weight has strongly uniform consistency, while $\hat{g}^{CART}_{\xi}(z\mid X=x)$ has weakly uniform consistency. And,
%    \begin{equation}
    \begin{align}
        \sup_{z\in\mathcal{Z}, x\in\mathcal{X}} \left|g_{\xi}(z\mid X=x) - \hat{g}_{\xi}^{kNN}(z\mid X=x)\right| 
        &= O_P\left(\sqrt{\frac{\log N}{N^{2\delta_{kNN}-1}}}\right),
        \quad0<2\delta_{kNN}-1<1,\\
        \sup_{z\in\mathcal{Z}, x\in\mathcal{X}} \left|g_{\xi}(z\mid X=x) - \hat{g}_{\xi}^{LSA}(z\mid X=x)\right|
        &= O_P\left(\sqrt{\frac{\log N}{N^{1-\delta_{LSA}(2d+1)}}}\right),\,0<\delta_{LSA}(2d+1)<1.
    \end{align}
 %   \end{equation}
\end{theorem}

The significance of Theorem \ref{theo::chance-consistency} lies in its resolution of the endogeneity-discontinuity conflict. CCW provides a piecewise-constant approximation of the probability measure. This structure ``freezes'' the local sample set within specific decision-dependent neighborhoods, ensuring that the approximation error remains controlled even when the optimizer actively seeks the boundary of the chance constraint. Consequently, the $O_P(\cdot)$ rates established for kNN and LSA provide the first quantitative characterization of sample efficiency for contextual chance constraints in the presence of decision-dependent feedback loops.
 
Meanwhile, although we are unable to provide the specific sample efficiency under the CART weight approximation, it still performs well in the numerical experiments (Section \ref{AccCompare}). Theorem \ref{theo::chance-consistency} further leads to the asymptotic optimality of the optimal solution obtained through this method as mentioned in Theorem \ref{theo::asym}.

\begin{theorem}[\textbf{Asymptotic Results}] \label{theo::asym}
     Suppose Assumptions \ref{ass::decomp&Ign} to \ref{ass::slater} hold, the mapping $(z,x) \to f(z,x)$ is $L_f-$Lipschitz continuous under the 1-Wasserstein distance, and for any $(z, x)$, for all $\xi$ the probability density function of $\psi_{\xi}(z, Y)$ where $Y\sim f(z,x)$ has a consistent upper bound, then for weight function that has strongly (weakly, resp.) uniform consistency,
    \begin{enumerate}
        \item (Accuracy) $\forall x\in\mathcal{X}$, for some optimal solution $z^*$, there exists a sequence $z^n=(z_1^n, z_2^n)\in \hat{P}(x)$ that $z^n\rightarrow z^*$ almost surely (or with probability). $\hat{L}(z^n \mid X = x)$ converges to $L(z^* \mid X=x)$ almost surely (or with probability);
        \item (Optimality) $L(z^n \mid X = x)$ converges to $L(z^* \mid X=x)$ almost surely (or with probability).
    \end{enumerate}
\end{theorem}

% \textit{Proof sketch: Compared with the content in \cite{lin2022data}, the differences of this theorem lie in the introduction of a decision-dependent effect, which affects the uniform consistency of the approximation and may introduce confounding bias. Therefore, we need to first prove the uniform consistency of this approximation under the influence of opportunity constraints. Based on this, we will re-prove the feasibility, accuracy, and optimality.}

% \begin{remark}
%     Theorem \ref{theo::asym} is not only applicable to the context considered in this paper but also extends to other optimization problems where the objective function is the expectation of an equicontinuous cost function when the aforementioned assumptions hold. Even in the presence of endogenous uncertainty, by employing kNN to approximate the objective function and the corresponding chance constraints, and then searching for the optimal solution, the properties mentioned in Theorem \ref{theo::asym} can still hold. Moreover, the chance constraint can be replaced with specific manifestations whose performance functions are equicontinuous with respect to the decision variables.
% \end{remark}

With Theorem \ref{theo::asym}, the minimum value of $L(z \mid X = x)$ can be obtained by finding the minimum of $\hat{L}(z \mid X = x)$ in $\hat{P}^N(x)$. At this point, the chance constraint requirement has also been transformed into the requirement that in the cluster, there should be at most $\alpha_{\xi}$ of points satisfying $\psi_{\xi}(z, y_i) \le 0$ for all $\xi$. Problem (\textbf{Appr-dd-ccp}) can be mathematically modeled as a MINLP in \eqref{eq-original}.
\begin{subequations}\label{eq-original}
    \begin{align}
        \min\quad& \frac 1 k\sum_{i=1}^N \mathcal{K}_i\cdot l(z, y_i),\label{eq-objective} \\
        s.t.\quad& \sum_{i=1}^N \mathcal{K}_i = k,\label{eq-ksum} \\
        &\text{Determine } \mathcal{K}_i \text{ for a given weight function}, \deltext{\nonumber\\
        &\qquad\qquad\qquad }i=1,\cdots,N, \label{eq-generalCluster}\\
        &\sum_{i=1}^N \mathcal{K}_i \gamma_{i, \xi} \le \lfloor k\alpha_{\xi}\rfloor,\xi \in \{1,\cdots, \Xi\},\label{eq-varcc1}\\
        &\psi_{\xi}(z, y_i) - M\gamma_{i, \xi} \le 0,\deltext{\nonumber\\
        &\qquad} i=1,\cdots,N, \xi \in \{1,\cdots, \Xi\},\label{eq-varcc2}\\
        &z=(z_1, z_2)\in \mathcal{Z}, k\in \mathbb{Z}^+,\deltext{\nonumber\\
        &} \mathcal{K}_i\in \{0,1 \},\gamma_i \in \{0,1\}, i=1,\cdots,N, \nonumber
    \end{align}
\end{subequations}
in which \eqref{eq-objective} and \eqref{eq-varcc1} correspond respectively to \eqref{eq-Appr-Obj} and \eqref{Est-cc}. $\mathcal{K}_i \in \{0, 1\}$ denotes the membership indicator for the $i$-th data point, i.e., $\mathcal{K}_i = \mathbb{I}\{(z_i, x_i) \in \mathcal{C}(z,x)\}$. $k=\sum_{i=1}^N \mathcal{K}_i=|\mathcal{C}(z,x)|$ is the variable that represents the cardinality of the cluster. And hence $w_i(z,x)=\mathcal{K}_i/k$. $M$ is a sufficiently large positive number. $\gamma_{i,\xi}$ is an indicator for whether $\psi_{\xi}(z, y_i)$ is larger than $0$.  \eqref{eq-ksum} and \eqref{eq-generalCluster} are cluster identification constraints that select $k$ points into the cluster. (\ref{eq-varcc1}) and (\ref{eq-varcc2}) stipulate that there are less than $\lfloor k\alpha_{\xi}\rfloor$ points in the cluster that $\psi_{\xi}(z,y_i) > 0$, as (\textbf{Appr-CC}) requires.

Specifically, in kNN, \eqref{eq-generalCluster} is
\begin{equation}\label{eq-ranking}
    \begin{aligned}
        \|z_1-z_{1i}\|^2 &- \|z_1-z_{1j}\|^2+\Vert x-x_i \Vert^2 - \Vert x-x_j \Vert^2\le M(\mathcal{K}_j - \mathcal{K}_i + 1), j=1,\cdots,N, i\neq j,
    \end{aligned}
\end{equation}
which rank the distances between historical contexts and current context. In CART, \eqref{eq-generalCluster} is $\mathcal{K}_i = \mathbb{I}((z_{1i}, x_i) \text{ and }(z_1,x)\text{ are in the same leaf node})$, which can be transformed into mixed-integer linear programming (MILP) constraints \citep{Ethan2025Hyper}. In LSA, \eqref{eq-generalCluster} is
\begin{equation}\label{eq-LSACluster}
    \begin{aligned}
        & \|z_1-z_{1i}\|^2+\Vert x-x_i \Vert^2 - h \ge M(1-\mathcal{K}_i),\\
        &h - \|z_1-z_{1i}\|^2-\Vert x-x_i \Vert^2 + \epsilon \le M\mathcal{K}_i, i=1,...,N,
    \end{aligned}
\end{equation}
which realize the logic of ``when $\|z_1-z_{1i}\|^2+\Vert x-x_i \Vert^2 \le h, \mathcal{K}_i=1$", where $\epsilon$ is an extremely small positive number.

% Due to the nonlinear and discontinuous nature of the CCW, it is difficult to solve \eqref{eq-original} directly. The quadratic objective \eqref{eq-objective} and the quadratic constraints \eqref{eq-varcc1} make the model a MINLP, which is NP-hard. The difficulty primarily arises from the endogenous uncertainty. If $z_1$ does not affects uncertainty, then using Constraints \eqref{eq-ksum} and \eqref{eq-generalCluster}, $\mathcal{K}_i$ can be uniquely determined. Then variable $\mathcal{K}_i$ and $k$ would degenerate into constants when solving the remaining problem. In other words, the endogenous uncertainty is the core reason for the significant increase in the difficulty of solving this problem.

\section{Solution Framework}\label{sec::Solution}

Formulation \eqref{eq-original} presents significant computational challenges as a MINLP. Two primary sources of complexity are: (i) Decision-dependent weights $\mathcal{K}_i$, which introduce complex combinatorial logic (e.g., sorting or tree traversal) into the optimization; and (ii) Constraints \eqref{eq-varcc1} and \eqref{eq-varcc2}, which generally define non-convex feasible regions. Direct optimization using off-the-shelf solvers is therefore computational intractable.

We develop a specialized solution framework to address these  challenges. We focus on the setting where the uncertainty-affecting decision space $\mathcal{Z}_1=\{z_1^1, \dots, z_1^{N_1}\}$ is discrete. This assumption is practically well-founded, as many strategic decisions in operations—such as selecting facility locations, pricing tiers, or technology options—are inherently discrete. Under this setting, we can effectively decouple the complex decision-dependent interactions.

Building on this assumption, the proposed framework provides a linearization reformulation method to constraints \eqref{eq-ksum}-\eqref{eq-generalCluster}, followed by a specialized strategy tailored when convexity conditions are met.

% The approximated problem \textbf{Appr-dd-ccp} presents significant computational challenges as a Mixed-Integer Non-Linear Program (MINLP). Two primary sources of complexity are: (i) the decision-dependent weights $\mathcal{K}_i(z,x)$, which introduce complex combinatorial logic (e.g., sorting or tree traversal) into the optimization; and (ii) the approximated chance constraints, which generally define non-convex feasible regions. Direct optimization using standard solvers is therefore intractable.

% To address these challenges, we develop a specialized solution framework. In this section, we focus on the setting where the uncertainty-affecting decision space $\mathcal{Z}_1=\{z_1^1, \dots, z_1^{N_1}\}$ is discrete. This assumption is practically well-founded, as many strategic decisions in operations—such as selecting facility locations, pricing tiers, or technology options—are inherently discrete. Under this setting, we can effectively decouple the complex decision-dependent interactions.

% Building on this structure, our method proceeds in three stages. In Section \ref{sec::reformulation of CF}, we reformulate the cluster identification by pre-calculating weights, thereby removing complex machine learning logic from the constraints. Section \ref{RoCC} then analyzes the structural properties of the chance constraints, establishing convexity conditions that eliminate the need for binary variables in the recourse problem.

\subsection{Linearization Reformulation of the Model} \label{sec::reformulation of CF}

The first computational hurdle in \eqref{eq-original} arises from constraints \eqref{eq-ksum}-\eqref{eq-generalCluster}. These constraints involve the endogenous variables $\mathcal{K}_i$ (or equivalently, weights $w_i(z,x)$) that depend on the decision variables in a complex, non-linear manner. Directly optimizing these constraints requires modeling the specific machine learning logic (e.g., sorting for kNN, tree traversal for CART) using Big-M formulations like \eqref{eq-ranking} and \eqref{eq-LSACluster}, resulting in a large-scale MINLP.

To address this, we exploit the discreteness of $\mathcal{Z}_1$. Instead of dynamically determining the cluster structure during optimization, we decouple the cluster identification process. Specifically, for all $z_1^t \in \mathcal{Z}_1$, $t=1,\cdots,N_1$, the corresponding set of weights are fixed. Let $\mathscr{K}^t = \{\mathcal{K}_1^t, \cdots, \mathcal{K}_N^t\}$ denote the weight vector associated with candidate $z_1^t$, which can be pre-calculated before the initiation of the optimization processes. Constraints \eqref{eq-ksum}-\eqref{eq-generalCluster} are then reformulated equivalently as a linear selection mechanism:
\begin{subequations}\label{eq-cluster reformulation}
    \begin{align}
    &\sum_{t=1}^{N_1} g_t=1, z_1 = \sum_{t=1}^{N_1} g_t z_1^t, \label{eq-select-g}\\
    &\mathcal{K}_i = \sum_{t=1}^{N_1} g_t \mathcal{K}_i^t,\quad i=1,\cdots,N, \label{eq-link-k}
    \end{align}
\end{subequations}
where $g_t \in \{0, 1\}$ is a binary selection variable and represents the choice of candidate $z_1^t$. Through this reformulation, the non-linear dependency of $\mathcal{K}_i$ on $z_1$ is replaced by the linear relationship \eqref{eq-cluster reformulation}. After the linearization of the bilinear constraints \eqref{eq-varcc1} \citep{glover1975improved}, the resulting model can be directly solved by off-the-shelf solvers. 

To efficiently compute $\mathscr{K}^t$ in \eqref{eq-cluster reformulation} before optimization, we introduce acceleration strategies for each weight type, with the final complexity being $O((N + N_1) \log N)$, which works best for low-to-medium dimensions:

\begin{itemize}
    \item \textbf{For kNN (Top-$k$ Search):} A brute-force approach requires calculating distances from a candidate $z_1^t$ to all $N$ historical points and then sorting them, yielding a overall complexity of $O(N_1\cdot N \log N)$. To accelerate this, k-d trees can be applied. As established by \cite{friedman1977algorithm}, searching for the $k$ nearest neighbors in such structures achieves an expected time complexity of $O(\log N)$. Thus, the total complexity for constructing the tree and querying all $N_1$ candidates is reduced to $O((N + N_1) \log N)$.
    
    \item \textbf{For CART (Tree Traversal):} Since the CART partition is static once the tree is trained, identifying the cluster for all $z_1^t$ reduces to passing the candidates down the tree. This can be batched efficiently, with a complexity of $O((N + N_1) \log N)$
    
    \item \textbf{For LSA (Spatial Partitioning):} The brute-force calculation of $\mathscr{K}^t$ involves computing pairwise distances between every candidate and every data point, resulting in a complexity of $O(N_1 \cdot N)$. We employ a k-d tree structure to index the historical data (taking $O(N \log N)$ construction time), which allows us to query the neighborhood set in $O(\log N)$ average time. Consequently, the total complexity is reduced to $O((N + N_1) \log N)$ (See  \ref{app::LSA}).
\end{itemize}

By applying the reformulation \eqref{eq-cluster reformulation} and these acceleration strategies, we transform the original MINLP into a standard MILP, which can be solved directly using off-the-shelf solvers.

\subsection{Special Case: Convexity Certification via Nested Structure} \label{RoCC}

Although \eqref{eq-cluster reformulation} in Section \ref{sec::reformulation of CF} linearizes constraints \eqref{eq-ksum}-\eqref{eq-generalCluster}, the approximated chance constraints \eqref{Est-cc}, MIP reformulated as \eqref{eq-varcc1}-\eqref{eq-varcc2}, remain a computational bottleneck. Specifically, requiring the total weight of satisfied data points exceeds $1-\alpha$ is inherently combinatorial. Modeling this requirement directly necessitates introducing binary variables for each scenario and Big-M constraints, which would render the remaining problem an intractable MILP.

To bypass this complexity, we further explore under what conditions constraints \eqref{eq-varcc1} and \eqref{eq-varcc2} become convex.
%, thereby allowing for efficient optimization without additional integer variables\hl{???}.
%
%\paragraph{Scenario-Induced Sub-feasible regions.}
For a fixed $z_1$ and a specific historical data point $y_i$, we define the \textit{sub-feasible region} $S(z_1, y_i)$ as the set of $z_2$ that satisfies the $\xi$-th chance constraint for this data point:
\begin{equation}
    S(z_1, y_i) = \{z_2 \in \mathcal{Z}_2 : \psi_{\xi}(z_1, z_2, y_i) \le 0\}.
\end{equation}

Formally, constraint \eqref{eq-varcc1} and \eqref{eq-varcc2} are satisfied if there exists a subset of indices $\mathcal{I} \subseteq \{1, \cdots, N\}$ such that the total weight $\sum_{i \in \mathcal{I}} w_i(z_1, x) \ge 1-\alpha_{\xi}$, and $z_2$ lies in the common intersection of the corresponding sets, namely $z_2 \in \bigcap_{i \in \mathcal{I}} S(z_1, y_i)$. %The feasible region generated from the $\xi$-th chance constraint is denoted as $F_{\xi}(z_1)$.
%
%\paragraph{The Nested Property.}

Note that these sub-feasible regions exhibit a \textit{nested structure} in many operations management contexts. Consequently, satisfying the constraint for a ``stricter" scenario implies satisfaction for ``easier" scenarios.
This structural property is formally captured in Lemma \ref{lem::cc}.

\begin{lemma}[Convex Approximated Chance Constraint] \label{lem::cc}
    Fix \(z_1\) and \(\xi\). Suppose that for any \(y_1,y_2\in Y\), \(S(z_1,y_1)\) and \(S(z_1, y_2)\) are nested and convex, i.e., \(S(z_1, y_1)\subseteq S(z_1, y_2)\) or \(S(z_1, y_2)\subseteq S(z_1, y_1)\). Then for any \(0<\alpha_{\xi}\le 1\), the approximated feasible region generated from the $\xi$-th chance constraint, denoted as \(F_\xi(z_1)\), is convex.
\end{lemma}

%\paragraph{Geometric Interpretation and Convexity.}
Lemma \ref{lem::cc} reveals that the nested structure simplifies the combinatorial nature of constraints \eqref{Est-cc}. Figure \ref{fig:nested_sets} presents a visual intuition of this property.
Each horizontal bar represents $S(z_1, y_{(i)})$ for a specific $y_{(i)}$. We sort the scenarios by difficulty such that they form a descending chain $S(z_1, y_{(1)}) \subseteq S(z_1, y_{(2)}) \subseteq \dots \subseteq S(z_1, y_{(N)})$.

Under this sorting, the problem of selecting a subset of scenarios with cumulative weight $1-\alpha_{\xi}$ collapses into identifying a single \textit{critical scenario} (the quantile). $F_{\xi}(z_1)$ is then exactly equivalent to $S(z_1, y_{(k^*)})$, where $k^*$ is the smallest index that the cumulative weight of scenarios $\{y_{(k^*)}, \dots, y_{(N)}\}$ reaches $1-\alpha_{\xi}$.

\begin{remark}
    This nestedness arises in many contexts where constraints scale with uncertainty. For instance, in inventory management, an order quantity satisfying high demand trivially satisfies low demand; in service systems, meeting a strict deadline implies meeting any looser deadline; in capacity planning, a system built for peak loads inherently accommodates off-peak volumes.
\end{remark}

Consequently, since the individual set $S(z_1, y_{(k^*)})$ is convex (as shown by the solid line interval in Figure \ref{fig:nested_sets}), the entire feasible region $F_{\xi}(z_1)$ is convex. This critical result implies that for any fixed $z_1$, the optimal $z_2$ can be found using standard convex optimization techniques, eliminating the need for complex MILP formulations for the remaining problem.

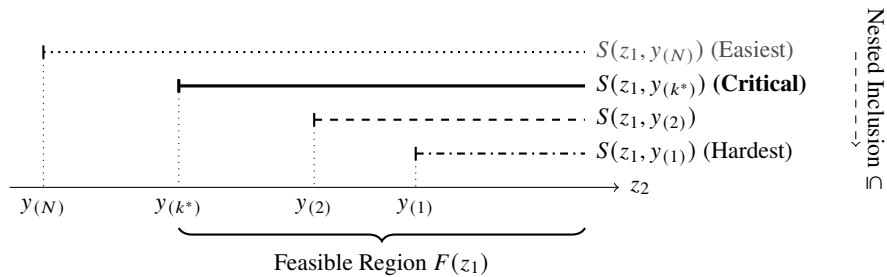
\begin{figure}[htbp]
\centering
\begin{tikzpicture}[scale=0.9]
    % Draw Axis
    \draw[->] (0,0) -- (9,0) node[right] {\footnotesize $z_2$};
    
    % --- Intervals (Feasible Regions) using Patterns for B&W ---
    
    % 1. S_(1) - Most difficult -> Dash-dotted line
    % Represents the strictest constraint (subset of all others)
    \draw[thick, dashdotted] (6, 0.5) -- (8.5, 0.5);
    \draw[dotted] (6, 0.5) -- (6, 0) node[below] {\footnotesize $y_{(1)}$};
    \node[right] at (8.5, 0.5) {\footnotesize $S(z_1, y_{(1)})$ (Hardest)};
    % Add a tick at the start
    \draw[thick] (6, 0.4) -- (6, 0.6); 

    % 2. S_(2) -> Dashed line
    \draw[thick, dashed] (4.5, 1.0) -- (8.5, 1.0);
    \draw[dotted] (4.5, 1.0) -- (4.5, 0) node[below] {\footnotesize $y_{(2)}$};
    \node[right] at (8.5, 1.0) {\footnotesize $S(z_1, y_{(2)})$};
    \draw[thick] (4.5, 0.9) -- (4.5, 1.1);

    % 3. S_(k) - Critical -> Very Thick Solid Line
    % This is the most important one, so we use a solid, heavy line
    \draw[very thick] (2.5, 1.5) -- (8.5, 1.5);
    \draw[dotted] (2.5, 1.5) -- (2.5, 0) node[below] {\footnotesize $y_{(k^*)}$};
    \node[right] at (8.5, 1.5) {\footnotesize \textbf{$S(z_1, y_{(k^*)})$ (Critical)}};
    \draw[very thick] (2.5, 1.4) -- (2.5, 1.6);

    % 4. S_(N) - Easiest -> Dotted line
    \draw[thick, dotted] (0.5, 2.0) -- (8.5, 2.0);
    \draw[dotted] (0.5, 2.0) -- (0.5, 0) node[below] {\footnotesize $y_{(N)}$};
    \node[black!70, right] at (8.5, 2.0) {\footnotesize $S(z_1, y_{(N)})$ (Easiest)};
    \draw[thick] (0.5, 1.9) -- (0.5, 2.1);

    % Nested Arrow (Keep dashed)
    \draw[->, dashed] (12.5, 2.0) -- (12.5, 0.6);
    \node[rotate=-90, above] at (12.5, 1.3) {\footnotesize Nested Inclusion $\subseteq$};

    % Highlight Feasible Region (Brace)
    % Changed color to black
    \draw[decorate,decoration={brace,amplitude=5pt,mirror},thick] (2.5,-0.6) -- (8.5,-0.6) node[midway,below=5pt] {\footnotesize Feasible Region $F(z_1)$};

\end{tikzpicture}
\caption{\centering{Visual illustration of the nested property.}}
\label{fig:nested_sets}
\end{figure}

%\paragraph{Efficient Solution Strategies.}
The convexity established in Lemma \ref{lem::cc} fundamentally simplifies the problem structure. Unlike the general case that requires the use of mixed-integer programming, the problem associated with a fixed $z_1$ reduces to a standard convex optimization problem. It motivates us to develop a Benders Decomposition approach for solving the original problem under the convexity conditions. The corresponding Master Problem (MP) iteratively selects a candidate $z_1^t$, and the convex Subproblem (SP) with fixed $z_1^t$ generates optimality or feasibility cuts to refine the search. Since the SP is convex, valid cuts can be efficiently derived from dual information or combinatorial logic. We relegate the detailed algorithm and cut formulations to  \ref{BDSF}. %This allows us to employ decomposition-based strategies:

%\begin{enumerate}
 %   \item \textbf{Naive Enumeration:} If the number of candidates $N_1$ is small and the convex subproblems are computationally cheap, one can simply  solve the corresponding convex subproblem for $z_2$ while enumerate all $z_1^t$ in the cluster identification process, and select the optimal pair.
    
%    \item \textbf{Benders Decomposition:} For cases where $N_1$ is large, we can adopt a Benders Decomposition approach. The Master Problem (MP) iteratively selects a candidate $z_1^t$, and the convex Subproblem (SP) generates optimality or feasibility cuts to refine the search. Since the subproblem is convex, valid cuts can be efficiently derived from dual information or combinatorial logic. We relegate the detailed algorithm and cut formulations to  \ref{BDSF}.
%\end{enumerate}

\section{Illustrative Example: Price-setting Newsvendor Problem} \label{sec::PSNP}
We adopt the Price-Setting Newsvendor Problem (PSNP), which has garnered considerable attention in the relevant literature \citep{deyong2020price}, to demonstrate the applicability of the reformulation method and the proposed solution framework.

We consider a retailer minimizing loss by determining a discrete price $p \in P = \{p^1, \cdots, p^{N_P}\}$ (the uncertainty-affecting decision) and an order quantity $0 \le q \le \bar{q}$ (the uncertainty-independent decision). The demand $d$ is uncertain and depends on the price $p$ and context $x$. Let $c$ be the unit ordering cost and $s$ be the unit salvage value. The loss function is defined as:
\begin{equation}\label{eq-cost}
    l(p, q; d) = -(p - c)q + (p - s)\max\{q-d, 0\}.
\end{equation}

We focus on the PSNP with a profit target constraint (also known as the Value-at-Risk (VaR) constraint), which requires the probability of the profit falling below a target $v$ is capped at $\alpha_v$. (The discussion regarding the alternative service level constraint is provided in  \ref{app::PSNP SL}).
The original optimization problem, which minimizes the expected loss subject to the profit target, is formulated as:
\begin{subequations} \label{eq-PSNP-original}
    \begin{align}
        \min\limits_{p\in P,\ 0 \le q \le \bar{q}}& \quad L(p,q\mid X=x)=E_{D\sim F(p, x)}[l(p, q; D) \mid X=x], \label{eq-PSNP}\\
        \text{s.t.} & \quad \mathbb{P}_D[l(p, q; D) \le -v \mid X=x] \ge 1 - \alpha_v. \label{eq-orig-cc}
    \end{align}
\end{subequations}

\subsection{Reformulation of the Approximated Chance Constraint} \label{subsec::PSNP RoCC}
Given a historical dataset $S_N = \{(p_i, q_i, x_i, d_i)\}_{i=1}^N$, we apply the CCW approximation framework to approximate \eqref{eq-PSNP} and \eqref{eq-orig-cc}.
The approximated model is formulated as:
\begin{subequations} \label{eq-PSNP-approx}
    \begin{align}
        \min_{p\in P,\ 0 \le q \le \bar{q}}\quad & \hat{L}(p, q) = \sum_{i = 1}^N w_i(p, x) l(p, q; d_i), \label{Est-nvObj} \\
        \text{s.t.} \quad & \sum_{i=1}^N w_i(p,x) \mathbb{I}\{l(p,q;d_i) \le -v\} \ge 1-\alpha_v. \label{Est-var}
    \end{align}
\end{subequations}
Equation \eqref{Est-nvObj} approximates the expected loss (\ref{eq-PSNP}), while constraint\ \eqref{Est-var} approximates the feasible region of the profit target constraint (\ref{eq-orig-cc}).

%Following the framework in Section \ref{sec::Solution}, for any fixed uncertainty-affecting decision (price) $p \in P$, the cluster structure is determined. Specifically, the weights $\mathcal{K}_i$ are pre-calculated.
%We first address the approximated profit target constraint \eqref{Est-var}. Specifically, for a fixed price $p$, the constraint requires:
%\begin{equation} \label{eq-PSNP-constraint-static}
%    \sum_{i=1}^N w_i(p, x) \cdot \mathbb{I}\{ l(p, q; d) \le -v \} \ge 1-\alpha_v.
%\end{equation}

To understand the structure of constraint (\ref{eq-orig-cc}), we examine the geometry of the loss function $l(p, q; d)$. As shown in Figure \ref{fig:VaRDemo}, when $p$ and $d$ are fixed, $l(p, q; d)$ is V-shaped with respect to $q$. The function value starts at 0 when $q=0$, decreases to a minimum of $-(p-c)d$ when $q=d$, and then continuously increases. The gradients during the descending and ascending phases are constant values $(-(p-c))$ and $(p-s)$, respectively.

\begin{figure}[htbp]
    \centering
    \includegraphics[trim=0 0cm 0 2cm, clip, width=0.65\linewidth]{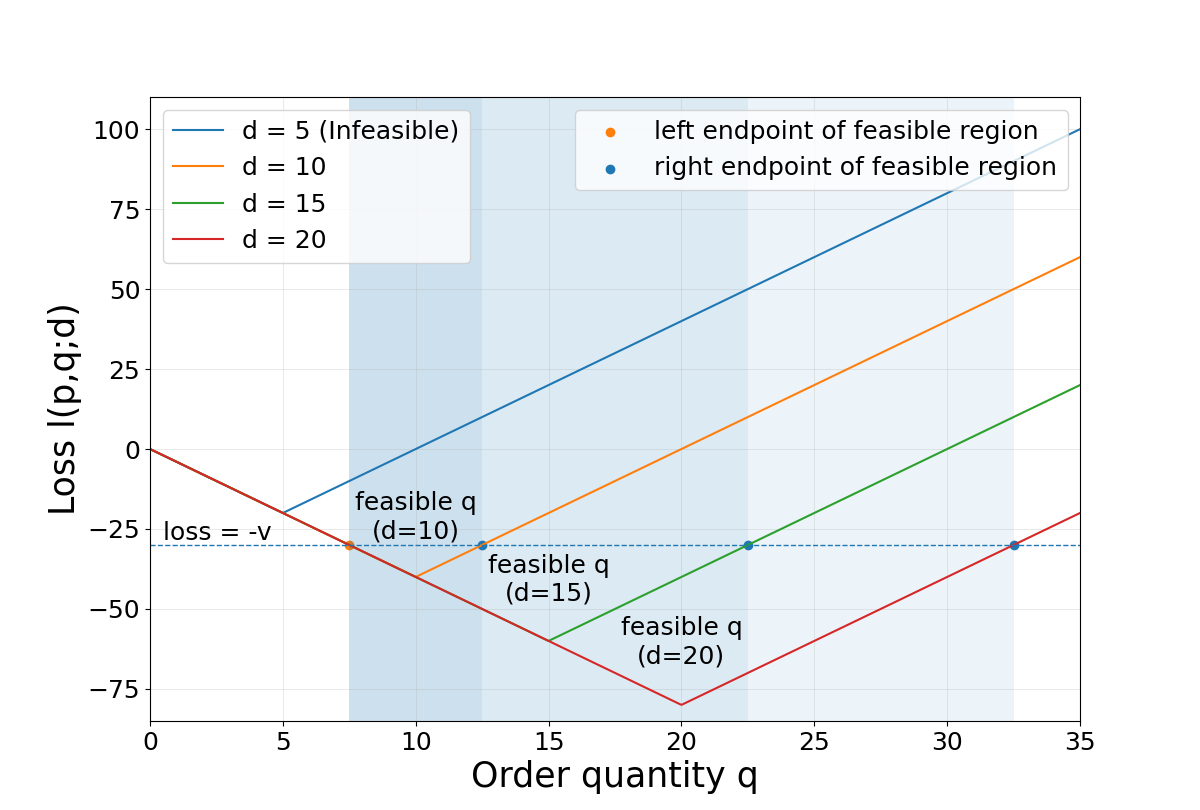}
    \caption{\centering A demonstration of the loss function and the VaR constraint with different $d$. \label{fig:VaRDemo}}
\end{figure}

Therefore, to satisfy the profit target condition $l(p,q;d) \leq -v$, the demand $d$ must be sufficiently high (specifically $(p-c)d \geq v$), and the order quantity $q$ must fall within a closed interval. Crucially, as $d$ increases, the ``valley" of the loss function deepens and widens to the right. Consequently, the left endpoint of the feasible interval remains constant, while the right endpoint increases.
This behavior confirms the nested property discussed in Section \ref{RoCC}: the feasible interval for a smaller demand scenario is strictly contained within the interval for a larger demand scenario. This allows us to reformulate constraint \eqref{Est-var} as a single deterministic interval, as stated in Proposition \ref{prop::var_ss}.

\begin{proposition}[\textbf{Est-Profit reformulation}]\label{prop::var_ss}
     Assume that $\min_i\{|(p-c)d_i - v|\} > 0$ for any possible $p$. For all $p$, denote $\upsilon_i$ as a binary indicator for whether $d_i > v/(p-c)$, and denote $D^v_i = \mathcal{K}_i\upsilon_id_i$ for all $i$. Sort $D^v_i$ as $(D^v_{(1)}, D^v_{(2)}, \cdots, D^v_{(N)})$, where $D^v_{(1)} \ge D^v_{(2)} \ge \cdots \ge D^v_{(N)}$. Constraint (\ref{Est-var}) is equivalent to $v/(p-c) \le q \le ((p-s)D^v_{(\lceil k\alpha_v \rceil)}-v)/(c-s)$, where $k = \sum_{i=1}^N \mathcal{K}_i$ is the number of historical points in the cluster.
\end{proposition}

\begin{remark}
    The requirement of $\min\{|(p-c)d_i - v|\} > 0$ is not difficult to achieve in real situations. Since both $p$ and $d_i$ are discrete values, the possible values of $(p-c)d_i$ are finite. By making adjustment to the value of $v$, we can ensure that none of the values of $(p-c)d_i$ is exactly equal to $v$. Such an adjustment does not affect real-world scenarios.
\end{remark}

Proposition \ref{prop::var_ss} implies that for a fixed $p$, the feasible region for $q$ is a simple closed interval. Denote the lower bound and upper bound derived in Proposition \ref{prop::var_ss} as $q_{min} = v/(p-c)$ and $q_{max} = ((p-s)D^v_{(\lceil K\alpha_v \rceil)}-v)/(c-s)$, respectively.

\subsection{Structural Analysis and Efficient Solution} \label{subsec::PSNP subAcc}

Having established the feasible interval $[q_{min}, q_{max}]$ for $q$, we now derive the optimal $q$ for a fixed $p$. The objective \eqref{Est-nvObj} represents the weighted expected Newsvendor loss. Since the loss function is convex with respect to $q$, the optimal solution for a fixed $p$ can be derived analytically by finding the unconstrained optimum and projecting it onto the feasible interval.

\begin{proposition} \label{prop::optimal_solution}
    (\textbf{Optimal Solution}) For all price $p$, sort the current cluster based on the daily demand to obtain the sorted cluster demand data, denoted as $(d_{(1)}, d_{(2)}, \cdots, d_{(k)})$, where $d_{(1)} \le d_{(2)} \le \cdots \le d_{(k)}$. Let $n^*=\min\{n \mid n(p-s)-k(p-c) > 0\}$, then, $q^* = \min\{\max\{q_{min}, d_{(n^*)}\}, q_{max}\}$, $y_i^* = \max\{q^*-d_i, 0\}$ is the optimal solution for the subproblem $SP(p)$.
\end{proposition}

Proposition \ref{prop::optimal_solution} reveals that for any price, the optimal order quantity admits a closed-form representation. This structural insight provides a significant computational advantage when implementing the Benders Decomposition framework proposed in Section \ref{RoCC}.

Specifically, the existence of an analytical solution allows us to solve the subproblem instantaneously without relying on numerical optimization solvers. In the context of the Benders decomposition, this means that for every iteration of the master problem, the optimal dual variables (and consequently the Benders cuts) can be derived directly and efficiently. This seamless integration of the specific problem structure into our general decomposition framework substantially reduces the computational burden and accelerates the overall convergence of the solution algorithm. We provide the detailed MP and SP formulations for the BD approach in  \ref{app::benders_newsvendor} and utilize it for performance comparison in Section \ref{subsec::SpeedTest}.

\section{Numerical Experiment} \label{sec::NumericalExperiment}

In this section, we conduct simulation-based numerical experiments to evaluate the proposed method using the price-setting newsvendor problem with VaR constraint. We set the demand-price relationship as an explicit function in the form of the commonly studied location-and-scale demand, serving as the ground truth. Section \ref{DataGeneration} introduces the data generation approach used in the experiments. Section \ref{subsec::nume sample efficiency} verifies the sample efficiency of kNN and LSA proposed in Theorem \ref{theo::chance-consistency}. Section \ref{AccCompare} compares our method with parametric models to demonstrate its effectiveness. Finally Section \ref{subsec::SpeedTest} tests the effectiveness of the proposed solution framework in Section \ref{sec::Solution}. 

The numerical experiment is conducted on a server equipped with a FusionServer G5500 V6, 128 cores, and 2TB of RAM, running on a Linux x64 system. The code is implemented in Python 3.10.12, with Gurobi version 9.5.2, using its default solver settings.

\subsection{Data generation method} \label{DataGeneration}
In the simulation experiments, we adopt two location-scale demand functions proposed in \cite{harsha2021prescriptive} as the true demand-price relationships:
\begin{enumerate}
    \item $D = \beta_0 + \beta_1 p + \sum_{i=2}^{11} \beta_iX_{i-2} +\left(\gamma_0+\gamma_1p+\gamma_2p^2+\sum_{i=3}^{12}\gamma_iX_{i-3}\right)U,$
    \item $D = \beta_0 + \beta_1 p + \beta_1'p^2 +\sum_{i=2}^{11} \beta_iX_{i-2} + (\gamma_0+\gamma_1p)\min(U,0) + \gamma_2p^2\max(U,0),$
\end{enumerate}
which are denoted as relationship modes 1 and 2. $\beta_0=200$, $\beta_1=-10$, $\beta_1'=-0.2$, $(\beta_2,\beta_3, \cdots, \beta_{11})=(-2, -1, 0, 1, 2, 0, 0, 0, 0, 0)/\sqrt{10}$, $\gamma_0=20$, $\gamma_1=-2.1$, $\gamma_2=0.2$, $(\gamma_3, \gamma_4, \cdots, \gamma_{12}) = (1, 1, 1, 1, 1, 0, 0, 0, 0, 0)/\sqrt{5}$. And $P = \{10, 10.1, 10.2,\cdots, 29.8, 29.9\}$, $c=5, s=2$, which follows examples in \cite{harsha2021prescriptive} and \cite{serrano2024bilevel}.

For random variable $U$, we consider the following three uncertainty modes:
\begin{enumerate}
    \item Normal(0, 1): Normal distribution with mean 0 and standard deviation 1.

    \item Log-normal(0, 1): Log-normal distribution with the underlying normal distribution having mean 0 and standard deviation 1.

    \item Student's t(3): Student's t distribution with 3 degrees of freedom,
\end{enumerate}
which are denoted as uncertainty modes 1, 2 and 3.

The normal distribution is a common random distribution, the log-normal distribution exhibits asymmetry, and the Student's t-distribution is another common symmetrical distribution. The properties of these three types of distributions effectively represent some frequently observed characteristics in real-world distributions.

To generate historical data, we employ the method proposed in \cite{serrano2024bilevel} to generate $X$ vector, by which $X\in R^{10}$ are drawn from a multivariate normal distribution with mean zero and covariance with entries $\sigma_{ij}=0.5^{|i-j|}$, for $i, j=0,\cdots,9$, while $p$ is obtained by uniformly sampling from $P$. Negative demands generated during the process are set as 0.

\subsection{Verification of the sample efficiency} \label{subsec::nume sample efficiency}
In this section, we verify the sample efficiency of kNN and LSA proposed in Theorem \ref{theo::chance-consistency}. The data size $N$ is increasing from 1000 to 100000. In the case of kNN, $\delta_{kNN}$ is set as 0.7. And in the case of LSA, since the influence of price is much larger than other context, the dimension of context can be approximately treated as 1. Hence $\delta_{LSA}$ is set as 0.2 to ensure that $(2d^{z} + 1)\delta_{LSA} = 0.6 < 1$. As a result, according to Theorem \ref{theo::chance-consistency}, in both cases, the approximation of the chance constraint should follow:
\begin{equation}\label{eq-supremegap_num}
    \sup_{z\in\mathcal{Z}, x\in\mathcal{X}}\! \left|g(z\mid X=x)\! -\! \hat{g}(z\mid X=x)\right|\! =\! O_P\left(\sqrt{\tfrac{\log N}{N^{0.4}}}\right).
\end{equation}

\begin{figure}[t]
    %\FIGURE
    {\includegraphics[scale=0.8]{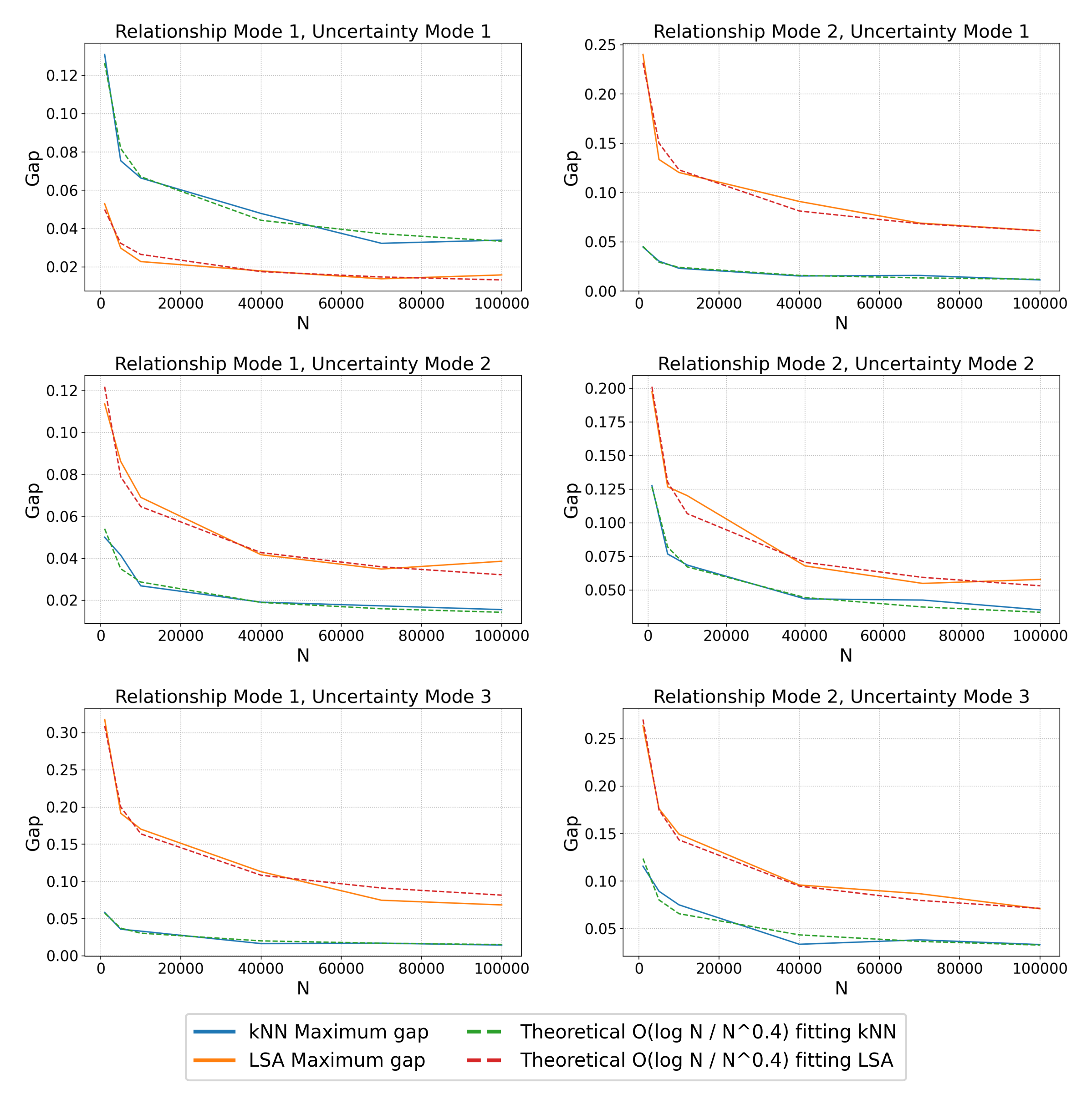}} % 图片内容
    \caption{\centering Sample Efficiency under Different Relationship and Uncertainty Modes.}\label{fig:Efficiency} % 图题
    {} % 备注说明
\end{figure}

In order to verify this conclusion, for each data size, we randomly select 20 different context-price pairs. For each context-price pair, we generate 10 sets of data to calculate the difference between the approximated probability and the true probability, and then take the average. We use the largest difference among these 20 context-price pairs as the approximation for the supreme gap in (\ref{eq-supremegap_num}).

The result is shown in Figure \ref{fig:Efficiency}, which shows the difference between the approximated probability and the true probability under different relationship and uncertainty modes. In the figure, the x-axis represents the data size, while the y-axis denotes the magnitude of the difference with different CCWs, namely kNN and LSA. The solid lines in the figure represent the actual sample efficiency curves, while the dashed lines are the curves obtained after fitting them with $\frac{\log N}{N^{0.4}}$.

The plotted curves illustrate the consistency between the experimental results and the theoretical predictions. The gap between the CCW approximation of the probability and the true probability descends roughly in $O_P\left(\frac{\log N}{N^{0.4}}\right) \in O_P\left(\sqrt{\frac{\log N}{N^{0.4}}}\right)$. This not only proves the correctness of Theorem \ref{theo::chance-consistency}, but also indicates that there might be other proof methods that can demonstrate this better sample efficiency.

To further verify Theorem \ref{theo::chance-consistency}, we choose both kNN and LSA as the weights, with relationship mode and uncertainty mode set to 1. We vary $\delta_{kNN}$ and $\delta_{LSA}$ to examine whether the sample efficiency aligns with the theoretical predictions. Figure \ref{fig:convergence_verification} presents the results. The top panel displays the convergence for the kNN approximation, while the bottom panel corresponds to the LSA approximation. The plots on the left shows the real sample efficiency plot, and the figure on the right shows the corresponding curves obtained after fitting them with the theoretical predictions. In both plots, the theoretical curves closely match the empirical errors, validating the rates predicted in Theorem \ref{theo::chance-consistency}.

\begin{figure}[htbp]
    \centering
    % Top Image: kNN
    \includegraphics[width=0.8\linewidth]{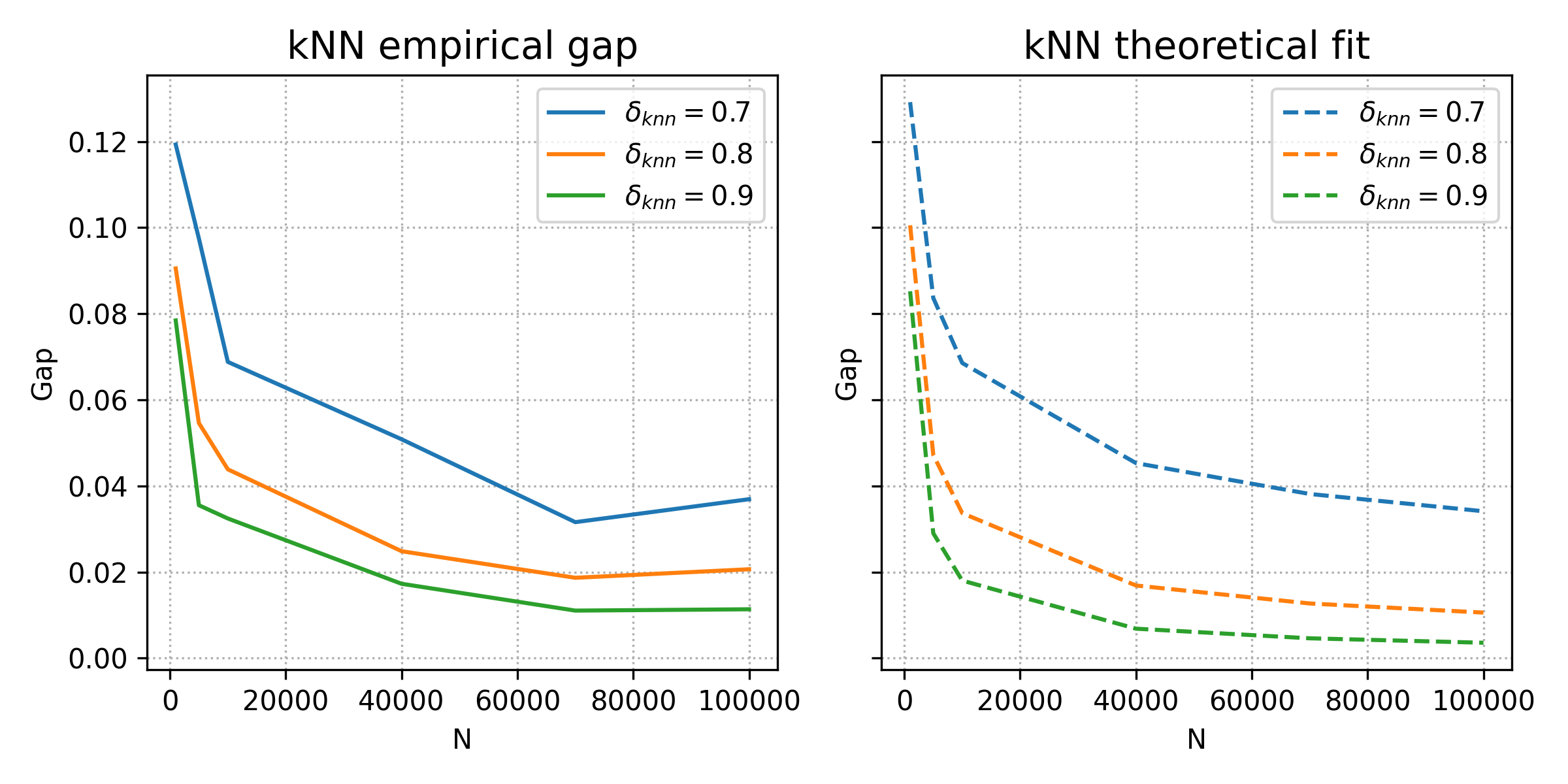}
    \vspace{0.2cm} % Small gap between images
    
    % Bottom Image: LSA
    \includegraphics[width=0.8\linewidth]{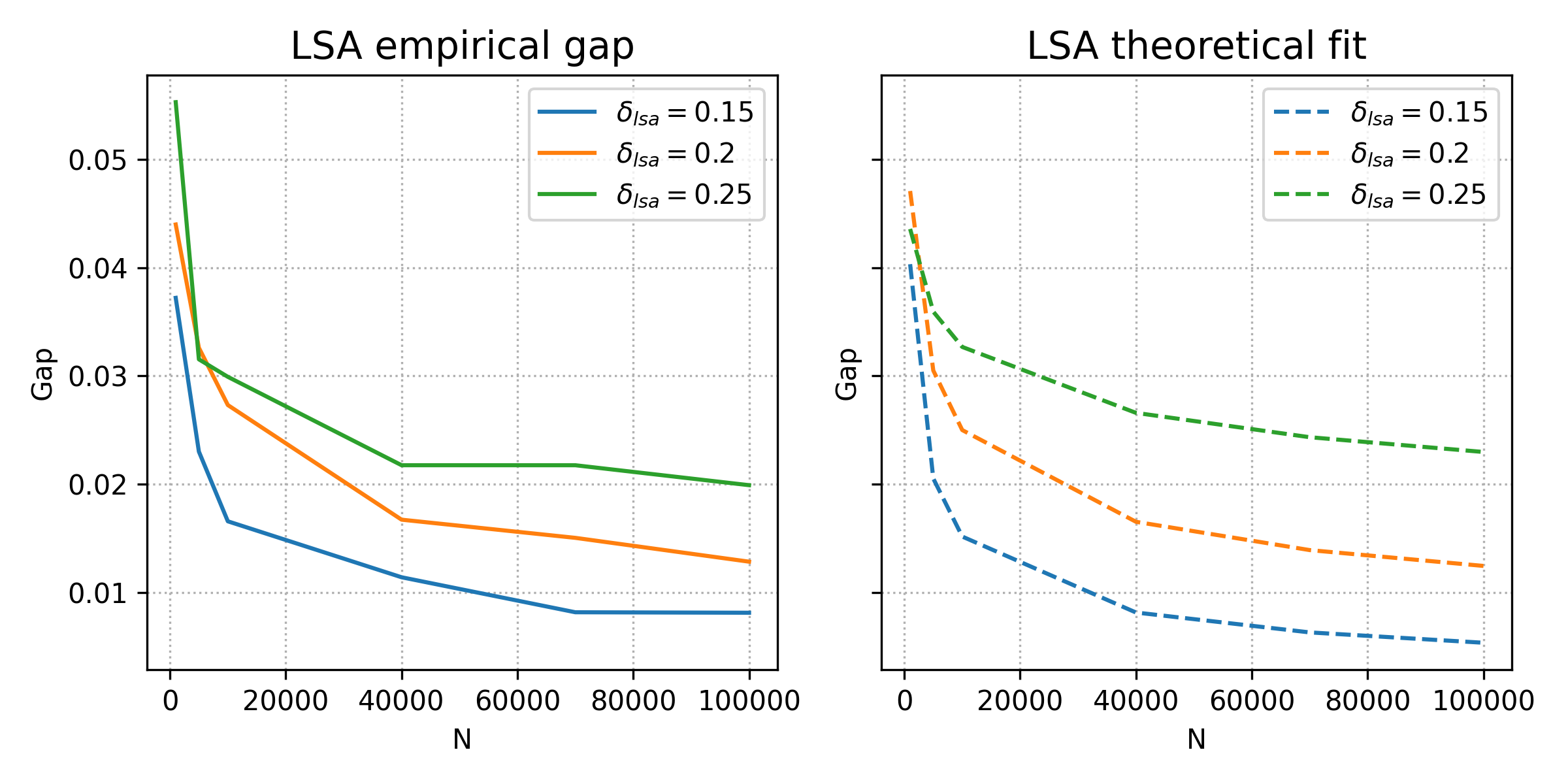}
    
    \caption{\centering Verification of sample efficiency}
    \label{fig:convergence_verification}
\end{figure}

\subsection{Comparison to parametric methods}\label{AccCompare}
In this section, we evaluate and compare the solution precision of parametric models with our proposed non-parametric model. This is done by assessing the discrepancy between the approximated optimal objectives provided by each model and the true objectives derived from the optimal solutions obtained, using the actual distribution of demand. Moreover, we measure and compare the quality of the optimal solutions delivered by each model using the true objective values as criteria. In addition, we compute and compare both the approximated and true probability for each model to evaluate how well they enforce the chance constraints. It should be noted that the same dataset, which was initially generated randomly, is utilized for this experiment.

Denote $S$ as the set of alternative regression models that maps from $P\times \mathcal{X}$ to $\mathbb{R}$. We first need to determine the best-performing model:
$$s^* = \arg\min_{s \in S} E_{\hat{P}_N}[\mathcal{D}(s(p, x), d)],$$
in which $(z, x, y) \sim \hat{P}_N$ is the empirical distribution constructed by $S_N$, $\mathcal{D}$ is a divergence function. We choose 20 regression models of three types, namely machine learning methods like linear regression, ensemble learning methods like light gradient boosting, and deep learning methods like CNN, as the set of alternative regression models. Models with smaller RMSE values are selected, which are lasso regression (las), lasso least angle regression (llar), and orthogonal matching persuit (omp) respectively. More detailed numerical information is reported in  \ref{app::parametric}. 

After completing the regression, the empirical distribution of the residuals exhibited by the regression model on the historical dataset is used as the approximated distribution of demand uncertainty, which is called residual-based distribution, commonly employed in research and practical work \citep{deng2022predictive, Sadana2025ContextualOptimizationSurvey}. Namely, for new context $x$, a decision is obtained through:
\begin{equation}\label{eq18}
    p^*(x) = \arg\min_{p\in \mathcal{P}} E_{y \sim P_\epsilon}[l(p, q; d)],
\end{equation}
in which $P_\epsilon$ is measured by the residual error on historical data $\{\epsilon_i\}$ where $\epsilon_i = d_i - s^*(p_i, x_i)$. Denote $\delta_{y}(x) = 1$ if $x=y$ and $0$ otherwise, then $P_\epsilon$ is given by:
\begin{equation}
    P_\epsilon(p, x) = \frac{1}{N} \sum_{i=1}^N \delta_{s^*(p, x) + \epsilon_i}.
\end{equation}
In this way, we build a conditional distribution of $d$ given $p$ and $x$.

Specifically, in this set of experiments, we select mode 1 (normal distribution) as the true distribution of the random variable, while keeping the other parameter combinations consistent with those in the previous section with $N=10000$. For each profit target $v$ and required chance level $\alpha$, we repeat it 10 times and take average on the results.% Comparisons between the actual realized loss and actual feasible rate of the optimal solution given by each model are reported in Table \ref{comparisontable}.

\begin{table}[t]
\TABLE
{Performance comparison between parametric methods and CCW approximation\label{comparisontable}}
{%
\begin{tabular}{ccllllllllllll}
\hline\up
\multirow{2}{*}{$v$} & \multirow{2}{*}{$\alpha$} & \multicolumn{6}{c}{Actual Realized Loss (CVaR)} & \multicolumn{6}{c}{Actual Feasible Rate} \\ \cline{3-14}
 & & kNN & LSA & CART & las & llar & omp & kNN & LSA & CART & las & llar & omp \\ \hline\up
\multirow{4}{*}{0} & 0.1 & -641.9 & -694.2 & -749.8 & -506.9 & -499.3 & -510.8 & 0.983 & 0.969 & 0.967 & 0.279 & 0.276 & 0.278 \\
 & 0.2 & -675.3 & -679.4 & -759.7 & -548.8 & -549.4 & -547.4 & 0.973 & 0.972 & 0.961 & 0.282 & 0.286 & 0.284 \\
 & 0.5 & -709.0 & -691.8 & -761.4 & -547.1 & -549.3 & -545.6 & 0.958 & 0.971 & 0.971 & 0.284 & 0.285 & 0.283 \\
 & 0.9 & -693.0 & -701.7 & -756.7 & -548.7 & -549.3 & -548.2 & 0.971 & 0.971 & 0.970 & 0.285 & 0.283 & 0.284 \\ \hline
\multirow{4}{*}{25} & 0.1 & -656.2 & -685.4 & -748.3 & -477.2 & -477.5 & -474.7 & 0.960 & 0.954 & 0.950 & 0.272 & 0.270 & 0.272 \\
 & 0.2 & -675.7 & -692.5 & -758.0 & -551.4 & -549.8 & -549.8 & 0.956 & 0.966 & 0.963 & 0.282 & 0.283 & 0.281 \\
 & 0.5 & -693.3 & -684.4 & -757.7 & -547.0 & -552.9 & -544.6 & 0.971 & 0.976 & 0.966 & 0.284 & 0.283 & 0.283 \\
 & 0.9 & -665.2 & -704.3 & -755.0 & -548.2 & -548.7 & -548.6 & 0.961 & 0.951 & 0.953 & 0.284 & 0.286 & 0.287 \\ \hline
\multirow{4}{*}{50} & 0.1 & -649.2 & -695.7 & -759.0 & -439.5 & -440.2 & -442.6 & 0.976 & 0.955 & 0.949 & 0.264 & 0.263 & 0.264 \\
 & 0.2 & -671.2 & -696.4 & -756.5 & -548.3 & -547.1 & -546.6 & 0.969 & 0.955 & 0.951 & 0.283 & 0.284 & 0.285 \\
 & 0.5 & -671.6 & -686.3 & -748.5 & -555.2 & -555.3 & -553.4 & 0.977 & 0.974 & 0.969 & 0.282 & 0.286 & 0.283 \\
 & 0.9 & -692.8 & -685.0 & -752.0 & -549.6 & -547.8 & -546.8 & 0.959 & 0.952 & 0.941 & 0.283 & 0.282 & 0.282 \\ \hline
\multirow{4}{*}{100} & 0.1 & -685.6 & -679.5 & -750.9 & -372.6 & -368.1 & -373.5 & 0.940 & 0.938 & 0.923 & 0.246 & 0.244 & 0.249 \\
 & 0.2 & -686.5 & -689.9 & -757.0 & -539.6 & -536.1 & -534.8 & 0.954 & 0.943 & 0.934 & 0.282 & 0.281 & 0.279 \\
 & 0.5 & -685.5 & -686.7 & -755.5 & -550.7 & -550.0 & -553.9 & 0.927 & 0.943 & 0.930 & 0.283 & 0.283 & 0.284 \\
 & 0.9 & -710.4 & -712.0 & -747.6 & -548.8 & -550.4 & -549.2 & 0.941 & 0.936 & 0.948 & 0.285 & 0.285 & 0.284 \\ \hline
\end{tabular}%
}
{} % Notes 
\end{table}

Table \ref{comparisontable} compares the actual realized loss and actual feasible rate of the optimal solution given by each model. 
The actual feasible rate is the actual probability of $l(p, q; d) \le -v$. Hence the actual feasible rate should be greater than $1-\alpha$. And since the parametric methods sometimes give infeasible solutions, we use the actual realized loss, which is the expectation of $l(p, q; d)$ when it is lower than the profit target $-v$ as a comparison metric (Namely, conditional value at risk (CVaR) when profit is greater than $v$). From the results in Table \ref{comparisontable}, it can be observed that, first, regardless of the combination of \(v\) and \(\alpha\), the optimal solution obtained by our proposed method provides decision-makers with more stable actual returns, with an average improvement of 31\%, 33.6\% and 45.8\% respectively for kNN, LSA and CART. Second, the solution we provide demonstrates greater stability in satisfying the constraints compared to parametric methods. These results indicate that the proposed method ensures both the quality and robustness of the solutions.

\subsection{Performance of Algorithmic Efficiency}\label{subsec::SpeedTest}

When using kNN weight, the wSAA of PSNP can be transformed into a MILP that can be solved by off-the-shelf solvers. Therefore, we use this situation further illustrate the effectiveness our solution framework.

In this experiment, the VaR parameters are set as $\alpha=0.2, v=100$, with relationship mode and uncertainty mode both chosen as 1. All experiments are performed using only one core of the server to ensure a fair comparison of algorithmic efficiency.

\begin{figure}[htbp]
    \includegraphics[width=1\linewidth]{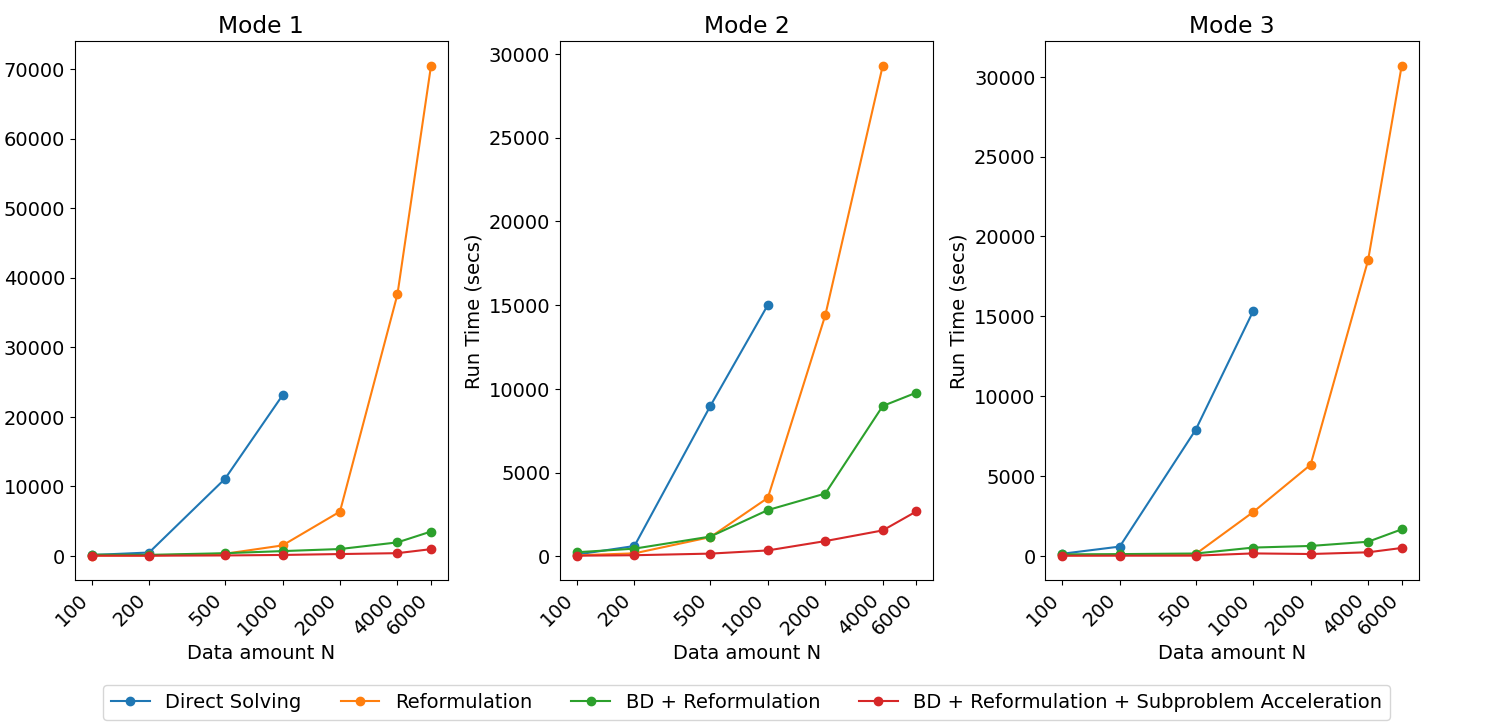}
    \caption{\centering Comparison of the computation time of four solution strategies.}
    \label{fig:speed_test}
\end{figure}

We evaluate and compare the performance of four distinct solution strategies:
\begin{enumerate}
    \item \textbf{Direct Solving:} This serves as the baseline approach. It involves directly solving the original mixed-integer linear program using an off-the-shelf solver without any structural exploitation.
    
    \item \textbf{Reformulation:} This method is based on the reformulation introduced in Sections \ref{sec::reformulation of CF} and \ref{RoCC}. It involves the pre-calculation of the decision-dependent weights, and the chance constraints are reformulated into deterministic intervals. This simplifies the problem structure significantly compared to the original formulation but is still solved as a single MILP.
    
    \item \textbf{BD+Reformulation:} This method builds upon the second method by applying the Benders Decomposition framework directly. It decomposes the reformulated problem into a master problem and convex subproblems, iteratively generating cuts to converge to the optimal solution.

    \item \textbf{BD + Reformulation + Subproblem Acceleration} This method builds upon the third method by utilizing complementary slackness to accelerate the Benders subproblem with the analytical solution proposed in Section \ref{RoCC}, detailed in  \ref{app::benders_newsvendor}.
\end{enumerate}

Figure \ref{fig:speed_test} compares the computation time of these four methods as the sample size $N$ increases.
The results reveal a clear hierarchy in computational efficiency. The first approach exhibits the steepest growth in computation time, quickly becoming inefficient as the data scale expands. The second method provides a substantial improvement over the baseline by simplifying the constraints and fixing the weights, thus reducing the solver's workload.
The third method works significantly better when the sample size grows large. By decoupling the uncertainty-affecting decisions from the continuous uncertainty-independent subproblems, it effectively mitigates the complexity growth, maintaining a low and stable computation time. And the fourth method demonstrates the superior performance among the four. This confirms that combining structural reformulation with algorithmic decomposition and acceleration yields the most significant speedup for CCW approximation.

\section{Case Study}\label{sec::case}
We demonstrate the performance of the proposed framework under a real case. A transaction-level dataset sourced from JD.com, one of China's premier retailers, for the 2020 MSOM Data-Driven Research Challenge \citep{2024JD} is employed as the real-world scenario. We model this case using price-setting newsvendor problem with profit target constraint. The effectiveness of the framework is validated through a comparative analysis involving the proposed framework and the company's original decision-making approach. 

% \subsection{Dataset preparation} \label{dsbg}
After transition, the data ultimately used as covariates and uncertain demand for the case study are described in Table~\ref{tab::real-data-description}. The detailed content of data transition can be found in  \ref{app::caseData}.

\begin{table}[t]
\TABLE
{Description of JD sales data\label{tab::real-data-description}}
{%
\begin{tabular}{ccp{20em}p{14em}}
\toprule\up
Variable & Type & Description & Statistics \\ 
\midrule
skuID  & string & The IDs corresponding to different types of SKUs & 90 types in total, which are then further divided into dummy variables. \\
Original\_unit\_price & float & The original price marked on the SKU & min 19.8, median 89, max 336. \\
Final\_unit\_price & float & The average unit price on that day & min 8.9, median 63.2, max 224.7. \\
Attribute1 & int & The first key attribute of the SKU & min 1, median 2, max 4. \\
Attribute2 & int & The second key attribute of the SKU & min 30, median 60, max 100. \\
Weekend & int & 1 means the sales date is a weekend, 0 otherwise & - \\
Quantity & int & Daily sales volume of this type of SKU & min 0, median 53, max 2765. \\ 
\bottomrule
\end{tabular}%
}
{}
\end{table}

Since the final prices of the goods range from 20 RMB to 120 RMB, through our visits and investigations of merchants, it is extremely unlikely that JD self-run items would be out of stock. Therefore, we make the assumption that once the price is acceptable for the customer, the market can respond to customer demands at any time. That is, the daily sales volume for each type of SKU in the dataset equals to the customer demand for that SKU in the corresponding day at current prices. This enables us to treat sales volume as demand in later process.

We conduct a regression analysis using machine learning methods to examine the relationship between demand and these covariates. Some SKUs demonstrate excellent performance in the best regression models. This part is elaborated in detail in  \ref{app::caseData}. This section focuses on 14 SKUs with favorable prediction performance (R² of 0.88 or above) to ensure that the true relationship between demand, prices and covariates is well established. Meanwhile, the determination of SKU cost and salvage value can be found in  \ref{app::caseParam}.

To illustrate the improvement that our method can bring about in practical applications, we divide the sales records of these 14 SKUs into a training set and a test set at a ratio of 0.8:0.2. Subsequently, we compare the expected revenue and the average feasible rate of the solutions offered by our method on the test set with those of the original solutions in the sales records. Since the size of the training set at this time is not large ($<400$), to ensure the usability of the method, we use kNN weight to guarantee the effectiveness of the method (that is, to avoid the error of division by zero). 

The result is shown in Table \ref{tab::comp_real_case}. The ``Loss Ratio" column in the table represents the ratio of the CVaR corresponding to $p_{CCW}$, $q_{CCW}$ to that of $p_o$ and $q_o$, where ($p_o$, $q_o$) is the original price and quantity in the dataset, and ($p_{CCW}$, $q_{CCW}$) is the optimal solution of the approximation with CCW. Since the objective value is negative, if ratio is greater than 1, it indicates that the new solution is better than the original one. The ``Original Feasible Rate" column shows the average rate of $P_D[l(p_o, q_o, D) \le -v \mid X=x]$ in the test set. And the ``CCW Feasible Rate" column shows the average rate of $P_D[l(p_{CCW}, q_{CCW}, D) \le -v \mid X=x]$ in the test set.

\begin{table}[t]
\TABLE
{Comparison of the CCW solution and the original solution\label{tab::comp_real_case}}
{%
\begin{tabular}{c|ccc|ccc|ccc}
\toprule
\multirow{2}{*}{$\alpha$}&\multicolumn{3}{c|}{$v$=0}&\multicolumn{3}{c|}{$v$=5}&\multicolumn{3}{c}{$v$=10}\\
\cline{2-10}\up
 & LR & OFR & CFR & LR & OFR & CFR & LR & OFR & CFR \\
\midrule
0.1 & 6.242 & \multirow{4}{*}{0.889} & 0.876 & 2.825 & \multirow{4}{*}{0.879} & 0.862 & 1.995 & \multirow{4}{*}{0.880} & 0.857 \\
0.2 & 6.188 & & 0.876 & 2.825 & & 0.862 & 1.996 & & 0.858\\
0.5 & 6.031 & & 0.883 & 2.832 & & 0.868 & 2.073 & & 0.864\\
0.9 & 4.549 & & 0.912 & 2.292 & & 0.903 & 1.732 & & 0.902 \\
\bottomrule
\end{tabular}%
}
{\centering LR=Loss Ratio, OFR = Original Feasible Rate, CFR = CCW Feasible Rate}
\end{table}

% \textcolor{red}{Note that it is abnormal that the feasible rate corresponding to the original solution actually increases when $v$ varies from 5 to 10. }This is because when $v$ increases and $\alpha$ is high, in some cases within the test set, there is actually no feasible solution. Therefore, in order to maintain the validity and consistency of the comparison, these cases will be excluded from this set of experiments. This further led to an increase in the Feasible Rate.

From the experimental results, it can be seen that: 
\begin{enumerate}
    \item Regardless of the value of $v$, the solution obtained by the CCW approximation achieves a significant average improvement in the profit, with the minimum average loss ratio being $1.73$.
    
    \item The CCW approximation method can flexibly adjust the pricing and order quantity without violating the opportunity constraints, and optimize the profit while ensuring feasibility. It can be observed that when $\alpha = 0.9$, although the loss given by the CCW approximation method is significantly higher compared to the previous groups, the feasible rate can still maintain a value greater than 0.9, demonstrating the robustness of the method.
\end{enumerate}

Next, we take this set of experiments with $v = 10$ as an example to examine the specific situation of the improvement in performance brought about by the CCW approximation. In Figure~\ref{fig:Ratio}, the histogram of the ratio of the actual realized loss at the final stage to the actual realized loss of the original solution is presented. It is observable that in this circumstance, the overall average income has risen by 96\%, and in 88.4\% of instances, there is an improvement in objective. This evidently manifests that in actual scenarios, our approach can reliably effect an enhancement compared to experience-based decisions.

\begin{figure}[t]
    %\FIGURE
    \centering
    {\includegraphics[trim=0 0.5cm 0 1.5cm, clip, width=0.8\textwidth]{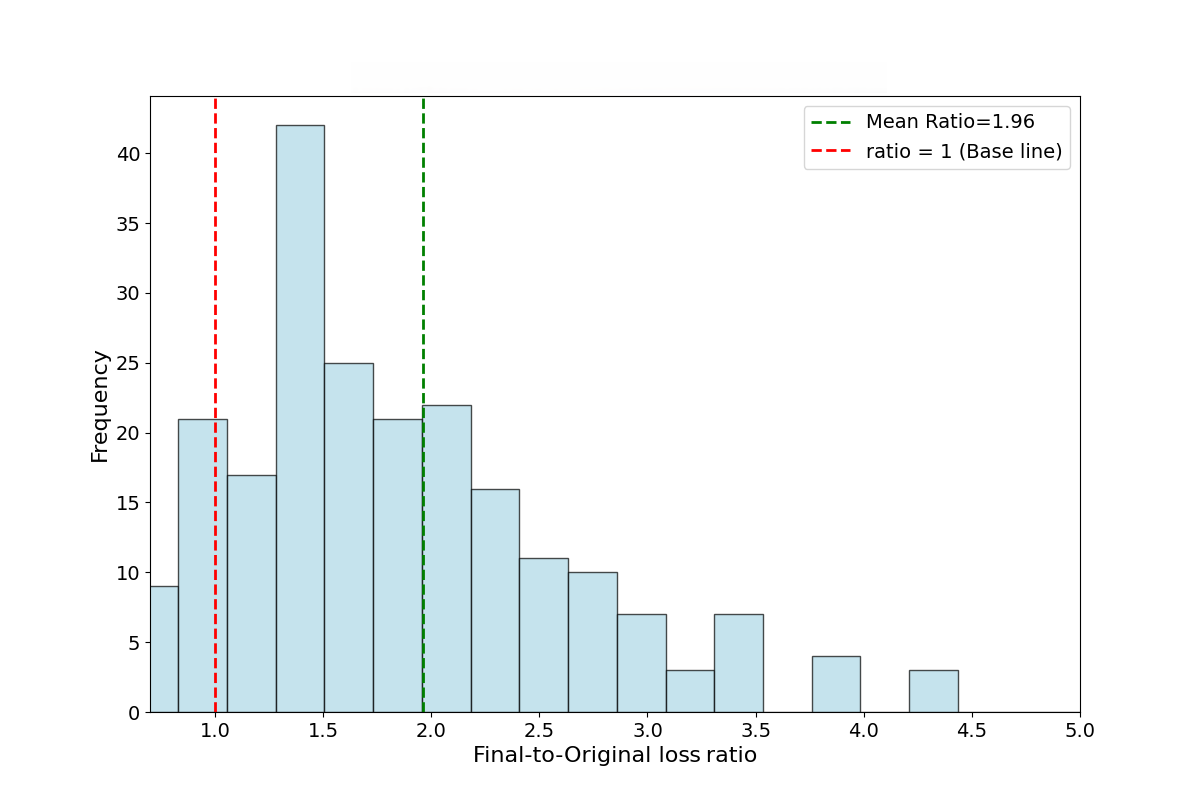}} % 图片内容
    \caption{\centering Histogram of Final-to-Original Loss Ratios. \label{fig:Ratio}} % 图题
    {} % 空的备注
\end{figure}

\section{Conclusion}\label{sec::conclusion}

In this work, we present a non-parametric framework to address CCCP-DDU, overcoming fundamental challenges in approximation and computation. By introducing the CCW approximation, we demonstrate its strong uniform consistency and derive its sample efficiency. We propose a linearization reformulation to convert the original nonlinear program into a linear program, which can be directly solved using off-the-shelf solvers. Additionally, we identify a convexity condition to develop a more efficient solution algorithm. Theoretical results, extensive simulations, and a real-world case study confirm both the validity and practical value of the proposed method.

Looking forward, several directions appear promising for extending this research. First, when handling contextual information, standardizing variables with higher impact on uncertainty to a larger scale may improve approximation accuracy. Developing a systematic quantification method to capture and leverage these influences would further enhance robustness. Second, though the asymptotic theorems work for general compact and bounded $\mathcal{Z}_1$, the current solution framework guarantees convergence only when $\mathcal{Z}_1$ is discrete. Extending the approach to continuous or hybrid domains could significantly broaden its applicability. Thirdly, beyond the three weights mentioned in this work, there may still be other weights with excellent properties to be discovered.

% Thirdly, in real life applications, it's hard to guarantee that the historical $z_1$ spread everywhere for all $x$ since all decision makers have their own decision interests (instead of randomly choosing $z_1$). As a result, the approximation might be biased. Future work may explore this situation and give an improved method.

% And lastly, in many modern applications, preserving privacy is critical. Future work may explore how to design frameworks where, given the decision, an outside observer cannot directly infer the underlying context. Incorporating privacy-aware mechanisms would expand the framework’s relevance in practice.

%Overall, this study provides a foundation for non-parametric approaches in decision-dependent chance constraint programming. By bridging theory and computation, we open pathways for more general, accurate, and privacy-conscious optimization under uncertainty.

% Acknowledgments here
\ACKNOWLEDGMENT{
% We would like to express our sincere gratitude to [acknowledge individuals, organizations, or institutions] for their invaluable contributions to this research. We are also grateful to [mention any additional acknowledgements, such as technical assistance, data providers, or colleagues] for their support and assistance throughout the course of this work.
The research is partially supported by the National Natural Science Foundation of China (Grant no. 72171129, 72188101, 92467302, 72250710683)
}

% References here (outcomment the appropriate case)

% CASE 1: BiBTeX used to constantly update the references
%   (while the paper is being written).
\bibliographystyle{informs2014}
\bibliography{reference.bib}

%\bibliographystyle{informs2014} % outcomment this and next line in Case 1
%\bibliography{sample} % if more than one, comma separated

% CASE 2: BiBTeX used to generate mypaper.bbl (to be further fine tuned)
%\input{mypaper.bbl} % outcomment this line in Case 2

% %\THEEndNotes
% \begingroup \parindent 0pt \parskip 0.0ex \def\enotesize{\normalsize} \theendnotes \endgroup

% Appendix here
% Options are (1) APPENDIX (with or without general title) or
%             (2) APPENDICES (if it has more than one unrelated sections)
% Outcomment the appropriate case if necessary
%
% \begin{APPENDIX}{<Title of the Appendix>}
% \end{APPENDIX}
%
%   or
%
\newpage

% --- 重置计数器 ---
\setcounter{proposition}{0}
\setcounter{lemma}{0}
\setcounter{figure}{0}
\setcounter{table}{0}
\setcounter{equation}{0}
\setcounter{definition}{0}

% --- 重定义编号格式 ---
% 格式变成：EC.1, EC.2 ...
% \renewcommand{\theproposition}{EC.\arabic{proposition}}
% \renewcommand{\thelemma}{EC.\arabic{lemma}}
% \renewcommand{\thefigure}{EC.\arabic{figure}}
% \renewcommand{\thetable}{EC.\arabic{table}}
% \renewcommand{\theequation}{EC.\arabic{equation}}
% \renewcommand{\thedefinition}{EC.\arabic{definition}}

\ECSwitch
\ECHead{Electronic Companion}
\renewcommand{\baselinestretch}{1} % 强制单倍行距
\small                        % 确保字号正常（防止被重置为大号字）
\selectfont                        % 立即应用字体和行距更改

\section{LSA matching procedure}\label{app::LSA}
This section discusses the reformulation and acceleration of LSA cluster identification. The core task of LSA cluster construction is to identify, for a given pair $(z_1,x)$, all historical observations whose distance to $(z_1,x)$ is no greater than a fixed bandwidth $h$. A brute-force implementation based on exhaustive pairwise distance comparisons between all candidate $(z_1,x)$ pairs and historical samples is computationally expensive. 

To accelerate this matching process, we employ a k-d tree to efficiently perform radius-based neighbor search. Specifically, a single k-d tree is constructed offline using the historical dataset
$S_N=\{(z_{1i},x_i,y_i)\}_{i=1}^N$, where the tree is built on the joint $(z_1,x)$ space and the outcome variable $y_i$ is not involved in the construction. The k-d tree is generated using a standard balanced construction procedure, which recursively partitions the data along coordinate axes according to the splitting rules of the k-d tree algorithm. Once constructed, this tree is reused for all candidate values $z_1\in\mathcal{Z}_1$ with a fixed contextual vector $x$.

For query point $(z_1,x)$, the k-d tree enables efficient identification of all historical points within Euclidean distance $h$ by pruning regions of the search space that cannot satisfy the distance condition. Compared to brute-force matching with computational complexity $O(dNN_1)$, where $d$ is the dimension of $(z_1,x)$, the k-d tree–based procedure reduces the expected cost to $O((N+N_1)\log N)$, where $k$ denotes the number of points returned by the radius query.

By constructing the k-d tree, the LSA matching procedure avoids repeated full-distance scans over the historical dataset, thereby substantially improving computational efficiency.

\section{A Solution Framework based on Benders Decomposition} \label{BDSF}

This section presents the solution framework based on Benders decomposition. We separate $z_1$ and $z_2$ in the decomposition, solving $z_1$ in the master problem and $z_2$ in the subproblem. The master problem is given as follows:
%Master problem ($MP_r$):
\begin{equation}\label{eq-master reform}
    \begin{aligned}
       \text{Master problem ($MP_r$):} \quad\quad\quad\quad\min \quad &e,\\
        s.t.\quad &\sum_{t=1}^{N_1} g_t=1, z_1 = \sum_{t=1}^{N_1} g_t z_1^t,\\
        &\mathcal{K}_i = \sum_{t=1}^{N_1} g_t \mathcal{K}_i^t,\quad i=1,\cdots,N,\\
        &e \text{ satisfies BD}_t, \quad z_1^t\in \mathbb{M}_r,\\
        &z_1 \in \mathcal{Z}_1, \mathcal{K}_i\in \{0,1 \},\quad i=1,\cdots,N,\\
        &g_t \in \{0, 1\}, t=1, \cdots, N_1,
    \end{aligned}
\end{equation}
where $e$ is a virtual variable that represents the original objective function. The Benders cut $BD_t$ is generated by solving the corresponding subproblem and its dual when $z_1=z_1^t$. $\mathbb{M}_r$, which is initially empty, is the set of $z_1$ that have entered the subproblem and are used to generate optimality cuts in iteration $r$. The optimal value of the master problem, a relaxed form of the original problem, provides a lower bound for the optimal objective of the original problem.

The master problem specifies the current optimal $z_{1r}$, the cluster $\mathcal{K}_r=(\mathcal{K}_{1r}, \cdots,\mathcal{K}_{Nr})$, and the feasible region for $z_2$, namely $F(z_{1r})$, which will be used in the corresponding subproblem formulated below.
%Subproblem ($SP_r$):
\begin{equation}
    \begin{aligned}
       \text{Subproblem ($SP_r$):} \quad\quad\quad\quad\min\quad&\frac 1 k\sum_{i=1}^N \mathcal{K}_{ir} \cdot [l((z_{1r}, z_2), y_i)],\\
        s.t.\quad &z_2 \in F(z_{1r}) \cap \mathcal{Z}_2.
    \end{aligned}
\end{equation}\label{subproblem}

The optimal value for the subproblem, a realized initialization of the original problem, gives an upper bound of the optimal objective of the original problem. Note that $\mathcal{K}_r$ in the subproblem is fixed since $z_{1r}$ is fixed. Therefore, the objective function of the subproblem is a linear combination of $l((z_{1r}, z_2), y_i)$. If $l((z_{1r}, z_2), y_i)$ is a convex function of $z_2$ regardless of $y_i$, $\mathcal{Z}_2$ is convex, and the conditions for Lemma \ref{lem::cc} holds, then the subproblem is a convex problem. Its dual problem can be derived through lagrangian multiplier method. For example, assume that $F(z_{1r})$ can be expressed as $\bar{F}(z_{1r}, z_2) \le 0, z_2 \in \mathcal{Z}_2$. Then the dual problem can be formulated as:
\begin{equation}\label{eq-dual}
\begin{aligned}
    \max_{\lambda\ge 0} &\inf_{z_2 \in \mathcal{Z}_2} 
    \Bigg( 
        \frac{1}{k} \sum_{i=1}^N \mathcal{K}_{ir}\cdot [l((z_{1r}, z_2), y_i)] 
        + \lambda\cdot \bar{F}(z_{1r}, z_2) 
    \Bigg)= \max_{\lambda\ge 0} g_d(z_{1r}, \mathcal{K}_r, \lambda).
\end{aligned}
\end{equation}

The form of the inner infimum, $g_d(z_{1r}, \mathcal{K}_r, \lambda)$, depends on the specific formulation of function $l$ and $\bar{F}$. Once the optimal $\lambda^*$ of problem (\ref{eq-dual}) is derived, $z_{1r}$ can enter $\mathbb{M}$, and the benders cut $BD_t$ can be given by:
\begin{equation}\label{eq-BDcut-general}
    \text{BD}_t: e \ge g_d(z_1,\mathcal{K}, \lambda^*).
\end{equation}

By iteratively solving the master problem, the corresponding subproblem and dual-subproblem, the upper and lower bounds of the optimal objective value will gradually converge. An overall algorithm is presented in Algorithm \ref{algo::accBD}.
\begin{algorithm}[htpb]
\footnotesize
\caption{\footnotesize BD Algorithm for DD-CCP with CCW approximation}
\label{algo::accBD}
%\vskip6pt
\begin{algorithmic}
\Procedure{AcceleratedBD}{$x, \mathcal{Z}, S_N$}
    \State \textbf{Input:} $x \in X$, $\mathcal{Z}$, $S_N = \{(z_1, x_1, y_1), \cdots, (z_N, x_N, y_N)\}$
    \State \textbf{Output:} Optimal decisions $z_1$ and $z_2$
%    \Statex
%
    \State Initialize: $UB \gets +\infty$, $LB \gets -\infty$, optimality gap $\epsilon$, iteration $r \gets 0$
    \State Pre-calculate the cluster: obtain $F(z_1^t), \mathcal{K}_i^t$ for all $z_i^t$
    \State Randomly choose a feasible $z_1 = z_{10}$
 %   \Statex
%
    \While{$\frac{UB-LB}{LB} \ge \epsilon$}
        \State Solve the replaced master problem (\ref{eq-master reform}), obtain $z_{1r}, g, \mathcal{C}, F(z_1)$, and the objective $L_r^{MP}$
        \State Update $\mathbb{M} \gets \mathbb{M} \cup \{z_{1r}\}$, $LB \gets \max \{LB, L_r^{MP}\}$
        \State Solve the subproblem $SP_r$, obtain the objective $L_r^{SP}$, update $UB \gets \min \{UB, L_r^{SP}\}$
        \State Solve the dual-subproblem, obtain dual variables, and add the optimality Benders Cut (\ref{eq-BDcut-general}) into the master problem
        \State $r \gets r+1$
    \EndWhile
\EndProcedure
%\Statex
\end{algorithmic}
\end{algorithm}

\section{PSNP with SL constraint} \label{app::PSNP SL}
We consider the PSNP with a service level constraint. The service level constraint is,
\begin{equation} \label{eq-sl}
    g_s(p,q \mid X=x) = \mathbb{P}_D[q \ge D \mid X=x] \ge 1- \alpha_s.
\end{equation}
It requires the probability of stockouts is less than $\alpha_s$. And the service level constraint can be approximated with,
\begin{equation}\label{Est-sl}
\begin{aligned}
    \hat{P}_N^s(x) &= \{p\in P,\ 0 \le q \le \bar{q}: \hat{g}_N^s(p,q\mid X=x)=\sum_{i=1}^N w_i(p,x) \mathbb{I}\{q \ge d_i\} \ge 1 - \alpha_s\}.
\end{aligned}
\end{equation}

From its geometric meaning, it can be concluded that \eqref{Est-sl} requires the order quantity to be greater than or equal to the demand of the $\lfloor k\alpha_s \rfloor$-largest demand in the cluster. 

% \section{(Original) PSNP with SL constraint} \label{app::PSNP SL}
% And the service level constraint can be written as,
% \begin{equation} \label{eq-sl}
%     g_s(p,q \mid X=x) = \mathbb{P}_D[q \ge D \mid X=x] \ge 1- \alpha_s.
% \end{equation}
% It requires the probability of stockouts is less than $\alpha_s$.

% And the service level constraint can be approximated with,
% \begin{equation}\label{Est-sl}
% \begin{aligned}
%     \hat{P}_N^s(x) &= \{p\in P,\ 0 \le q \le \bar{q}: \hat{g}_N^s(p,q\mid X=x)\\
%     &=\sum_{i=1}^N w_i(p,x) \mathbb{I}\{q \ge d_i\} \ge 1 - \alpha_s\}.
% \end{aligned}
% \end{equation}

% Regarding to the case with the service level constraint, as we mentioned in Section \ref{RoCC}, we have Proposition \ref{prop::SL_ss}.

% \begin{proposition}[\textbf{Est-SL reformulation}] \label{prop::SL_ss}
%     For all $p$, denote $D^s_i = \mathcal{K}_id_i$ for all i. Sort $D^s_i$ as $(D^s_{(1)},$ $D^s_{(2)}$, $\cdots$, $D^s_{(N)})$, where $D^s_{(1)} \ge D^s_{(2)} \ge \cdots \ge D^s_{(N)}$. When we approximate the chance constraint (\ref{eq-sl}) with CCW and $\alpha_s < 1$, the approximated chance constraint (\ref{Est-sl}) is equivalent to $q \ge D^s_{(\lfloor k\alpha_s \rfloor )}$. And when $\alpha_s=1$, constraint (\ref{Est-sl}) can be directly removed.
% \end{proposition}

% And in the case of SL constraint, $q_{min}=D^s_{(\lfloor k\alpha_s \rfloor)}$, there is no $q_{max}$. 

\section{Solving approximated PSNP with Benders Decomposition}\label{app::benders_newsvendor}

We detail the Benders Decomposition framework for the Price-Setting Newsvendor Problem (PSNP). 

\subsection{Master Problem Formulation}
After the pre-calculation process where the cluster information and the approximated chance constraint intervals are solved for all candidate prices $p \in P$, the Master Problem (MP) selects the optimal price candidate. We introduce binary variables $g_t \in \{0, 1\}$ to indicate the selection of the $t$-th candidate price $p^t$.
The MP is formulated as follows:
\begin{equation}\label{eq-master reform-PSNP}
    \begin{aligned}
        \min_{e, g, p, q_{min}, q_{max}}\quad &e, \\
        s.t.\quad & \sum_{t=1}^{N_P} g_t = 1, \\
        &p = \sum_{t=1}^{N_P} p^tg_t,\ \mathcal{K}_i = \sum_{t=1}^{N_P} \mathcal{K}_i^t g_t, \quad \forall i=1,\dots,N, \\
        &q_{min} = \sum_{t=1}^{N_P} q_{min}^t g_t,\ q_{max} = \sum_{t=1}^{N_P} q_{max}^t g_t,\\
        &q_{min} \le q_{max} \le \bar{q},\\
        &e \ge \text{Benders Cuts}_k, \quad k=1, \dots, K_{iter}, \\
        &g_t\in \{0,1 \}, \quad t=1,\cdots,N_P.
    \end{aligned}
\end{equation}
where, $e$ is a scalar variable representing the approximated loss, and the constraints ensure that exactly one price candidate is selected along with its corresponding cluster weights and feasible interval bounds.

\subsection{Primal Subproblem}
Given a fixed price $p$ (and determined by the MP solution $\hat{g}_t$), the Subproblem ($SP(p)$) minimizes the weighted expected loss by determining the optimal order quantity $q$. Based on the reformulation in Section \ref{subsec::PSNP RoCC} in the main context, the feasible region is the interval $[q_{min}, q_{max}]$.
The formulation is:
\begin{subequations}\label{PSNP-subproblem}
    \begin{align}
        \min_{q, y}\quad & \mathcal{Q}(p) = \frac{1}{K}\sum_{i=1}^N \mathcal{K}_i[-(p-c)q+(p-s)y_i],\\
        s.t.\quad &y_i \ge q - d_i, \quad i=1,\cdots,N, \label{eq-sub-y}\\
        &q_{min}\le q \le q_{max}, y_i \ge 0.\label{eq-sub-qbound}%\\
    %    &y_i \ge 0.
    \end{align}
\end{subequations}
where $y_i = \max\{q-d_i, 0\}$ captures the lost sales/excess inventory logic. As established in Proposition \ref{prop::optimal_solution}, the optimal solution $q^*$ can be obtained : $q^* = \min\{\max\{q_{min}, d_{(n^*)}\}, q_{max}\}$, where $d_{(n^*)}$ is the unconstrained fractile solution derived from sorting the demands.

\subsection{Dual Subproblem and Cut Generation}
Since the primal subproblem is a linear program, strong duality holds. We formulate the Dual Subproblem ($DSP(p)$):
\begin{subequations}\label{eq-dual-PSNP}
    \begin{align}
        \max_{\lambda}\quad &-\sum_{i=1}^N \lambda_{1i}d_i + \lambda_2 q_{min} - \lambda_3 q_{max},\\
        s.t.\quad &\sum_{i=1}^N \lambda_{1i}-\lambda_2+\lambda_3 = \frac{1}{K}\sum_{i=1}^N \mathcal{K}_i(p-c), \label{eq-dual-q}\\
        &\lambda_{1i} \le \frac{\mathcal{K}_i}{K}(p-s),\quad i=1,\cdots,N, \label{eq-dual-yi}\\
        &\lambda_{1i}, \lambda_2, \lambda_3 \ge 0,\quad i=1, \cdots, N,
    \end{align}
\end{subequations}
where $\lambda_{1i}$ corresponds to constraint \eqref{eq-sub-y}, and $\lambda_2, \lambda_3$ correspond to the lower and upper bounds in \eqref{eq-sub-qbound}, respectively.

%\paragraph{Recovering Dual Variables via Complementary Slackness.}
Instead of solving the dual LP, we recover the optimal dual variables $\lambda^*$ directly from the primal optimal $q^*$ using complementary slackness conditions:
\begin{enumerate}
    \item \textbf{For $\lambda_{1i}$:} 
    \begin{itemize}
        \item If $d_i < q^*$, then $y_i = q^*-d_i > 0$. The primal constraint is non-binding regarding non-negativity, so the dual constraint \eqref{eq-dual-yi} must be binding: $\lambda_{1i}^* = \frac{\mathcal{K}_i}{K}(p-s)$.
        \item If $d_i \ge q^*$, then $y_i = 0$. We typically set $\lambda_{1i}^* = 0$ (or determined by the balance equation).
    \end{itemize}
    \item \textbf{For $\lambda_2$ and $\lambda_3$:}
    \begin{itemize}
        \item If $q^* > q_{min}$, then $\lambda_2^* = 0$. If $q^* = q_{min}$, $\lambda_2^*$ absorbs the residual slack from equation \eqref{eq-dual-q}.
        \item If $q^* < q_{max}$, then $\lambda_3^* = 0$. If $q^* = q_{max}$, $\lambda_3^*$ absorbs the residual slack.
    \end{itemize}
\end{enumerate}

Once the optimal dual variables $\lambda^*$ are obtained, we generate the  optimality cut (\ref{eq-bd-PSNP-cut}) and add it to the Master Problem:
\begin{equation}\label{eq-bd-PSNP-cut}
    e \ge -\sum_{i=1}^N \lambda_{1i}^* d_i + \lambda_2^* \left(\sum_{t=1}^{N_P} q_{min}^t g_t\right) - \lambda_3^* \left(\sum_{t=1}^{N_P} q_{max}^t g_t\right).
\end{equation}

The decomposition approach solves the MP and SP iteratively until the gap between them is sufficiently narrowed.
%This decomposition approach iterates between the MP and the analytical solution of the SP until the gap between the upper bound (primal objective) and lower bound (MP objective $e$) closes.

\section{Performance of all candidate parametric models in the numerical experiment}\label{app::parametric}

%The performance of all candidate parametric prediction models are listed in Table \ref{tb::performance comparison}.
\begin{table}[htpb]
\TABLE
{Performance of all candidate parametric prediction models\label{tb::performance comparison}}
{%
\begin{tabular}{ll ll}
\toprule\up
Model & MSE & Model & MSE \\
\midrule
Lasso Regression & 4708.02 
& Deep Neural Network & 4760.81 \\

Lasso Least Angle Regression & 4708.02 
& Gradient Boosting Regressor & 4880.89 \\

Orthogonal Matching Pursuit & 4709.88 
& Random Forest Regressor & 4965.66 \\

Elastic Net & 4710.51 
& Dummy Regressor & 4974.70 \\

Bayesian Ridge & 4712.02 
& AdaBoost Regressor & 4999.98 \\

Ridge Regression & 4721.13 
& Extra Trees Regressor & 5002.58 \\

Least Angle Regression & 4721.15 
& Light Gradient Boosting Machine & 5090.30 \\

Linear Regression & 4721.15 
& K Neighbors Regressor & 5754.92 \\

Huber Regression & 4744.76 
& Decision Tree Regressor & 9492.33 \\

 &  
& Passive Aggressive Regressor & 14038.61 \\
\bottomrule
\end{tabular}%
}
{}
\end{table}

\section{Supplementary Materials of the Case Study}

\subsection{Data process and the establishment of true relationship}\label{app::caseData}
In case study, we employ the transaction-level data derived from JD.com for the month of March 2018, during which there were no major holidays or promotions \citepapp{2024JD}. Specifically, this study focuses on the orders table and the SKUs table. The orders table contains 486,928 unique customer orders that were placed during March 2018. For each specific order, its order date and day of week, original and final unit price, and individual order quantity are recorded. It should be noted that the original unit price is the price same for all customers, while the final unit price may vary among customers or orders owing to promotions or discounts. The SKUs table describes the characteristics of all 31,868 SKUs that belong to a single product category receiving at least one click during March 2018. Each SKU has two key attributes, whose higher values represent better performance of certain functionality. The first attribute takes integer value between 1 and 4, and the second one takes integer value between 30 and 100.

Regarding the processing and preparation of the dataset, the orders table serves as the primary sheet and the following data processing steps are conducted for each SKU. Firstly, calculate the total quantity of each SKU sold each day from the first to the last day of sale, and denote this as the daily demand for each SKU. Given substantial scale of JD.com as a major retailer, stock-outs are relatively uncommon. Secondly, compute the average original unit price and final unit price per day respectively and refer to them as the daily ones. Thirdly, recode weekdays and weekends as 0 and 1, respectively, and incorporate this variable into daily sales records. Additionally, the two key attributes from the SKUs table are merged into the daily sales records for each SKU. Finally, the dataset is filtered to include only SKUs with total sales exceeding 1,000 units, a threshold indicative of substantial sales volume, along with their relevant information, thereby forming the foundational dataset for analysis.

And then, we employ the processed dataset to conduct customized regressions for different SKUs, treating them as categorical features. The analysis examines the relationship between the dependent variable demand, and the independent variables prices and covariates including attribute 1, attribute 2, and weekday or not. The resulting random forest model achieved the best overall R² of 86\%. As an illustration, the true and predicted values of demand for the 6 SKUs with the highest R² are plotted as shown in Figure \ref{fig:Regression}.
\begin{figure}[htpb]
    \FIGURE
    {%
        \begin{minipage}{\textwidth}
        \centering
            \includegraphics[trim=0 1.5cm 0 0.2cm,clip,width=0.75\textwidth]{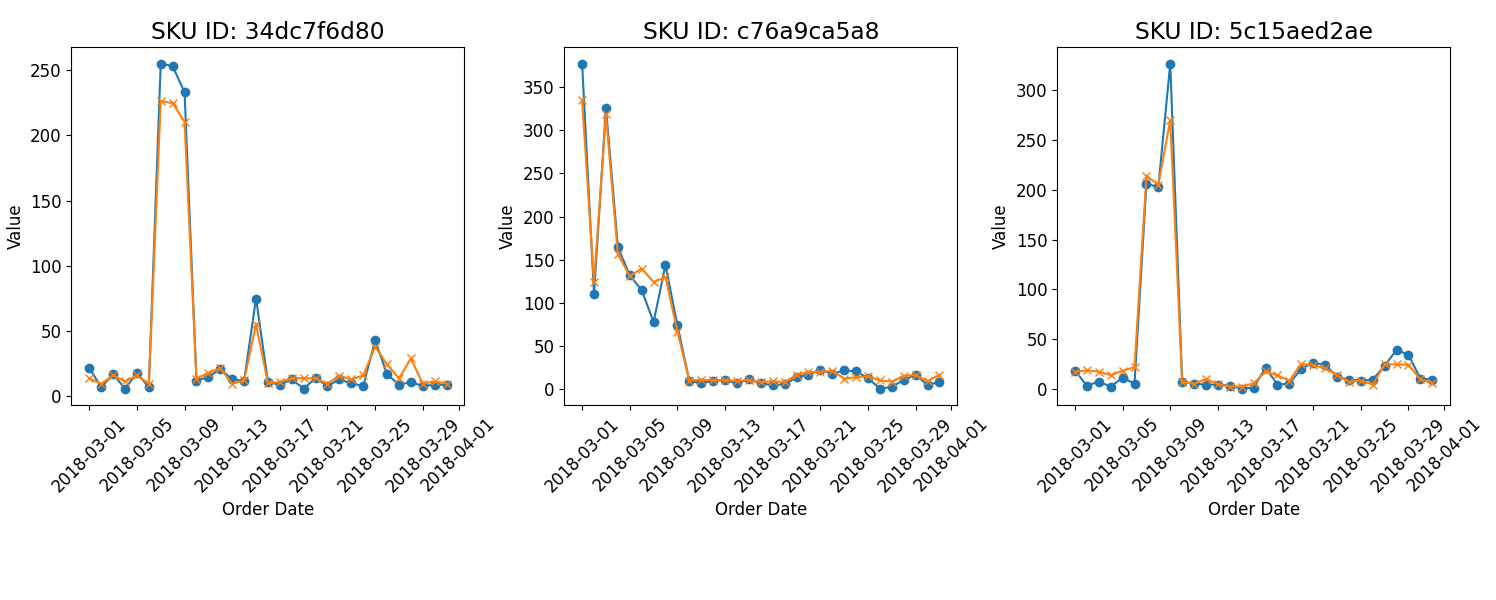} \\
            % \vspace{0.5em} % 两张图之间的间距
            \includegraphics[trim=0 0.2cm 0 0.5cm, clip, width=0.75\textwidth]{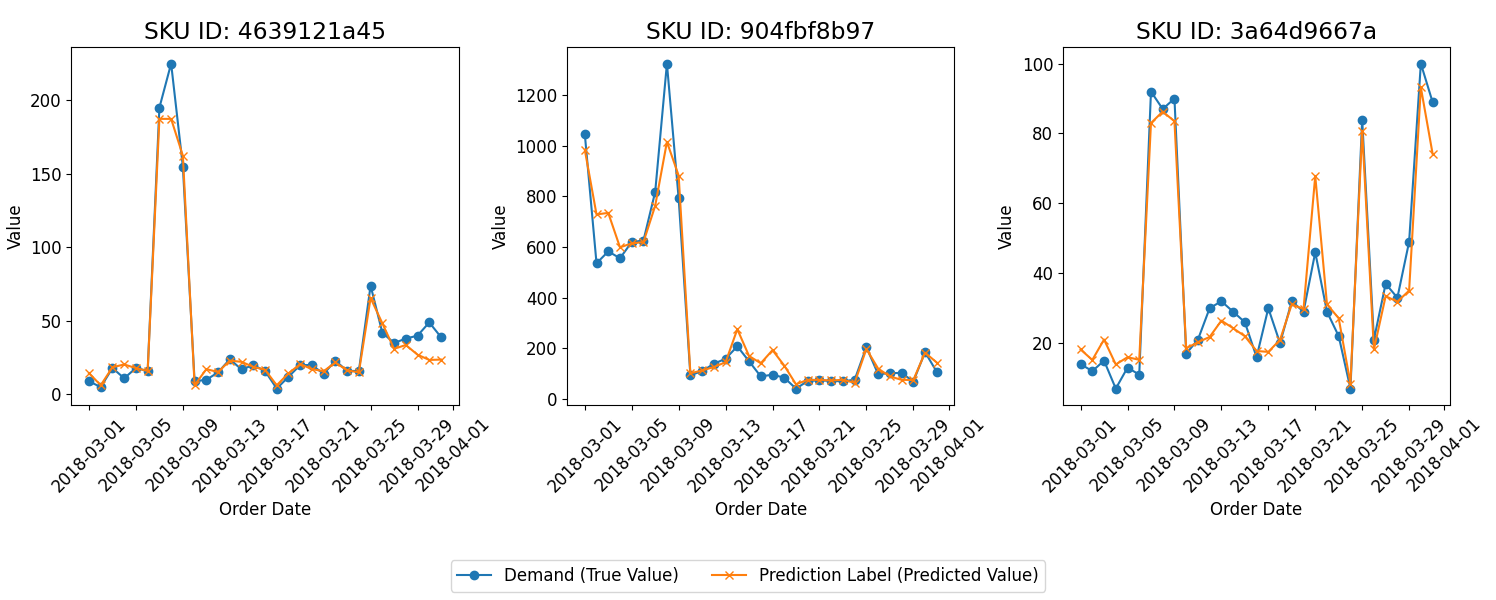}
        \end{minipage}%
    } % 图片内容
    {\centering Random Forest Model Performance. \label{fig:Regression}} % 图题
    {} % 空的备注
\end{figure}

The analysis of the prediction residuals reveals a waveform distribution centered around zero, which is a reasonable form for residuals (Figure \ref{fig:Residual}). We adopted the empirical residual distribution to characterize the error term, a methodology consistent with prior works in \citeapp{deng2022predictive, kannan2024residuals}.% and \citeapp{sen2018learning}.
\begin{figure}[htpb]
    \FIGURE
    {\includegraphics[trim=0 0 0 1.2cm, clip, width=0.6\textwidth]{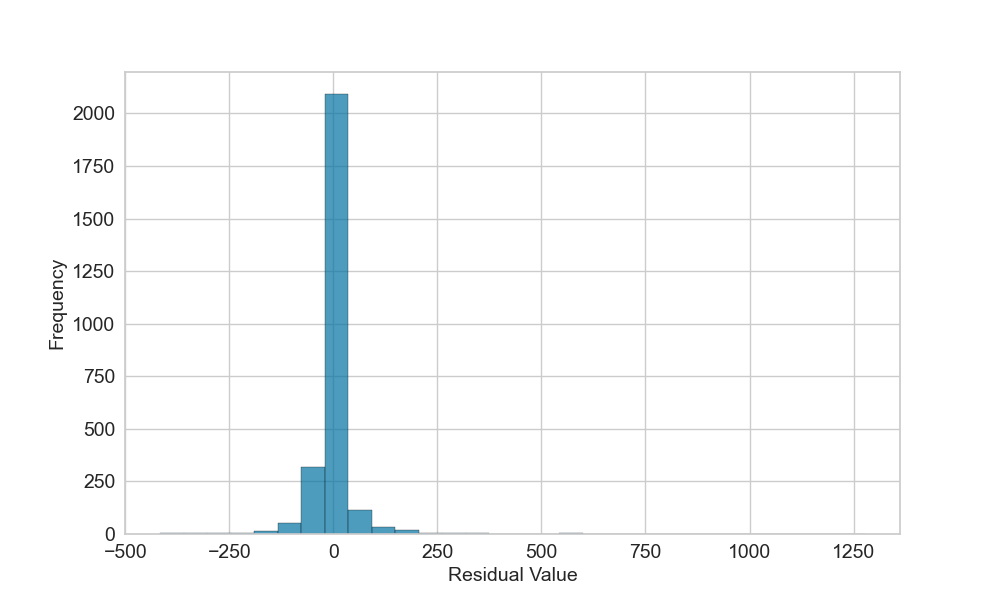}} % 图片内容
    {\centering Histogram of Residuals. \label{fig:Residual}} % 图题
    {} % 空的备注
\end{figure}

\subsection{Determination of cost, salvage value and other parameters}\label{app::caseParam}
To determine the unit cost and unit salvage value for each SKU, we referred to the financial report of JD for the first quarter of 2018 (https://ir.jd.com/news-releases/news-release-details/jdcom-announces-first-quarter-2018-results). It is stated that the average net return on goods was 2.1\% in this quarter. Therefore, we set the unit cost of the SKU to be 1/1.021 of its final average selling price.

Regarding the determination of salvage value, through our visits and investigations of merchants, we have learned that JD.com has a relatively good acquisition mechanism for its own merchants. Products that have not been fully sold within a certain period of time can be recovered at prices close to the original price. Therefore, the residual value in its operation mainly manifests in the inventory cost. For the SKUs products with the price around 100 yuan considered in this dataset, the monthly inventory cost per individual product is approximately 6 to 9 yuan, which accounts for about 9\% of their cost price. When averaged out over each day, it amounts to approximately 0.3\%. Therefore, we set its unit salvage value at 0.997 of its cost price.

In the selection of the pricing range $P$, considering the stability of the pricing, we set $P$ to be within three RMB above or below the original price, and divide it into units of 0.1 RMB.

\section{Proofs}\label{app::Proof}

We first introduce the definitions of Vapnik-Chervonenkis (VC) Dimension.

\begin{definition}[VC dimension]
    A family $\mathcal F$ of indicator functions $f:\mathcal Y\to\{0,1\}$ shatters a set $\{y_1,\cdots,y_m\}$ if every one of the $2^m$ labelings of those points can be realized by some $f\in\mathcal F$. The \textbf{VC‐dimension} of $\mathcal F$, denoted $\mathrm{VC}(\mathcal F)$, is the largest $m$ that it shatters (or $+\infty$ if arbitrarily large sets can be shattered).
\end{definition}

And we recall the Vapnik-Chervonenkis Theorem.

\begin{lemma}[VC Inequality \citepapp{vapnik1968uniform}]
\label{lem::VC Inequality}
    If an indicator class $\mathcal F$ has finite VC‐dimension of $v$, we have,
    \begin{equation}\label{eq-VC thm}
    \begin{aligned}
        &\mathbb{P}\left(\sup_{f\in F}\left|\frac 1 N \sum_{i=1}^N f(X_i) - E[f(X)]\right|>\epsilon\right)\le c_1v(2N)^v\exp(-N\epsilon^2/c_2),
    \end{aligned}
    \end{equation}
    in which $X_i$ is i.i.d sampled from the distribution of $X$, $c_1$ and $c_2$ are constants that are independent of $\mathcal{F}$, the distribution of $X$ and $N$.
\end{lemma}

Intuitively, Lemma \ref{lem::VC Inequality} states that if an indicator class $F$ has finite VC-dimension, then the sample average approximation built with this indicator has strongly uniform consistency since the right-hand-side of \eqref{eq-VC thm} converges to 0 when $N\to\infty$.

\subsection{Proof of Lemma~\ref{lem::kNN} and Lemma~\ref{lem::CART}}

In the case of kNN, for any fixed $z$:
\begin{equation}\label{eq-pointwise consistency}
    \sup_{x\in \mathcal{X}, z_1'\in\mathcal{Z}_1} \left|\sum_{i=1}^Nw_i(z_1', x) l(z; y_i) - E_{Y\sim f(\cdot \mid z_1', x)}[l(z, Y)]\right| \to 0,
\end{equation}
almost surely (Lemma 4 in \citeapp{bertsimas2019predictions}).

In the case of CART, for any fixed $z$, (\ref{eq-pointwise consistency}) holds with probability (Lemma 7 in \citeapp{bertsimas2019predictions}). (\ref{eq-pointwise consistency}) has exactly the same form as (EC.7) in \citeapp{bertsimas2020predictive}. Then, as Assumptions \ref{ass::decomp&Ign} and \ref{ass::continuous} hold, following exactly the same proving routine of Theorem EC.9 in \citeapp{bertsimas2020predictive}, we eventually get for any fixed $x$:
\begin{equation}\nonumber
    \sup_{z\in\mathcal{Z}}\left|\sum_{i=1}^Nw_i(z, x) l(z; y_i) - E_{Y\sim f(\cdot \mid z, x)}[l(z, Y)]\right|\rightarrow 0
\end{equation}
almost surely in the case of kNN or with probability in the case of CART, which is the third equation from the bottom of the proof of Theorem EC.9 in \citeapp{bertsimas2020predictive}, and also the definition of uniform consistency in Definition \ref{def::uniform consistency}. Q.E.D.

\subsection{Proof of Proposition~\ref{prop::DNConsistency}}
Since all chance constraints share similar assumptions, the proving processes are the same throughout all chance constraints. For simplicity, we drop the subscript of the chance constraints. We start with the proof of LSA approximation. As a spherical classifier, the VC dimension of $K_{z,x,h}(Z, X) = \mathbb{I}\left\{\Vert (Z, X) - (z, x)\Vert \le h\right\}$ is $v = d^{z_1} + d^x + 1$. 

Denote $\Delta_N = \sup\limits_{z\in\mathcal{Z}, x\in\mathcal{X}}\left|\hat{U}^{LSA}(z,x) -U^{LSA}(z,x)\right|$, set $\epsilon_N = M\sqrt{\frac{\log N}{Nh_N}}$, where $M$ satisfies that $M^2/c_2 > 1$. By Lemma~\ref{lem::VC Inequality},
\begin{equation}\nonumber
    \begin{aligned}
        \mathbb{P}(\Delta_N>\epsilon_N) &\le c_1(v+1)(2N)^v\exp\left(-\frac{N}{c_2} M^2 \frac{\log N}{Nh_N}\right)=c_1(v+1)(2N)^vN^{-(M^2/c_2)(1/h_N)}.
    \end{aligned}
\end{equation}

And since $h_N=CN^{-\delta_{LSA}}$, 
\begin{equation}\nonumber
    \begin{aligned}
        \mathbb{P}(\Delta_N>\epsilon_N) &\le c_1(v+1)(2N)^vN^{-(M^2/c_2)(1/CN^{-\delta_{LSA}})}\\
        &=c_1(v+1)2^v\cdot N^{v-(M^2/c_2C)N^{\delta_{LSA}}}\le c_1(v+1)2^v\cdot N^{v-(1/C)N^{\delta_{LSA}}}.
    \end{aligned}
\end{equation}

For any $\epsilon > 0$, we have,
\begin{equation}\nonumber
    \begin{aligned}
        \log(\mathbb{P}(\Delta_N>\epsilon_N))
        \le \log(c_1) + \log(v+1)+v\log2 + (v-(1/C)N^{\delta_{LSA}})\log N.
    \end{aligned}
\end{equation}

For the RHS of the constraint, $(v-(1/C)N^{\delta_{LSA}})\log N$ is decreasing when $N$ is big enough that $(v-(1/C)N^{\delta_{LSA}}) < 0$.

If $c_3=\log\epsilon - \log(c_1) - \log(v+1)-v\log2 > 0$, then when $N > \lceil \exp({\frac{\log Cv}{\delta_{LSA}}}) \rceil$, $(v-(1/C)N^{\delta_{LSA}})\log N < 0 < c_3$. And if $c_3 < 0$, set $N_0 = \lceil(\max(-2\delta_{LSA} c_3C, 1, 2vC, e))^{1/\delta_{LSA}} \rceil$. Let $t=N^{\delta_{LSA}}$, then $N = t^{1/\delta_{LSA}}$, $\log N=\log t/\delta_{LSA}$, then
\begin{equation}\nonumber
    \begin{aligned}
        (v-(1/C)N^{\delta_{LSA}})\log N &= \left(v-\frac{t}{C}\right)\frac{\log t}{\delta_{LSA}} = \frac 1 {\delta_{LSA}} \left(v\log t - \frac{t\log t}{C}\right).
    \end{aligned}
\end{equation}

When $t > t_0= N_0^{\delta_{LSA}} = \lceil(\max(-2\delta_{LSA} c_3C, 1, 2vC, e))\rceil$, since $t > e$, we have $\log t > 1.$ And since $t > 2vC$, $t/C-v > t/2C, \log t>0$,
\begin{equation}\nonumber
    \log t(t/C-v)>\frac{t}{2C}\log t.
\end{equation}

Since $t>-2\delta_{LSA} c_3C$, we have $t\log t > t > -2\delta_{LSA} c_3C$, then $\frac{t}{2C}\log t > \frac{-2\delta_{LSA} c_3C}{2C}=-\delta_{LSA} c_3$. So,
\begin{equation}\nonumber
    \begin{aligned}
        &\frac 1 {\delta_{LSA}} \left(v\log t - \frac{t\log t}{C}\right)
        <-\frac{1}{\delta_{LSA}}\log t\frac{t}{2C}
        <\frac{\delta_{LSA} c_3}{\delta_{LSA}}
        =c_3.
    \end{aligned}
\end{equation}

Namely, $\forall \epsilon > 0$, when $c_3=\log\epsilon - \log(c_1) - \log(v+1)-v\log2 < 0$, $\exists N_0 = \lceil(\max(-2\delta_{LSA} c_3C, 1, 2vC, e))^{1/\delta_{LSA}} \rceil$, when $N>N_0$, $(v-(1/C)N^{\delta_{LSA}})\log N < c_3$.

In conclusion, $\forall \epsilon > 0$, $\exists N_0$, when $N>N_0$, we have that $\mathbb{P}\left(\Delta_N>M\sqrt{\frac{\log N}{Nh_N}}\right) < \epsilon$. So, 
$$\sup\limits_{z\in\mathcal{Z}, x\in\mathcal{X}}\left|\hat{U}^{LSA}(z,x) -U^{LSA}(z,x)\right| = O_p\left(\sqrt{\frac{\log N}{Nh_N}}\right).$$

And for kNN approximation, it is evident that that its classifier is essentially a sphere centered at $(z, x)$, with the distance from $(z, x)$ to the k-th farthest historical data point as its radius. So the VC-dimension of kNN classifier class $K^{kNN}_{z,x,k,N}(Z,X) = \mathbb{I}\{(Z,X)\text{ is a kNN of }(z,x)\}$ is also $(d^{z_1} + d^x+1)$.

Take the following case (see Figure \ref{fig:kNNexample}) as an example. No matter where $(z,x)$ is placed and how $k$ is chosen, we will never classify point A, B, C as 1 while classify point O as 0. So for 2-dimension kNN classifier, the VC dimension is 3.
\begin{figure}[htpb]
    \FIGURE
    {\fbox{\includegraphics[trim=0.5cm 2.2cm 0.2cm 1cm, clip, scale=0.35]{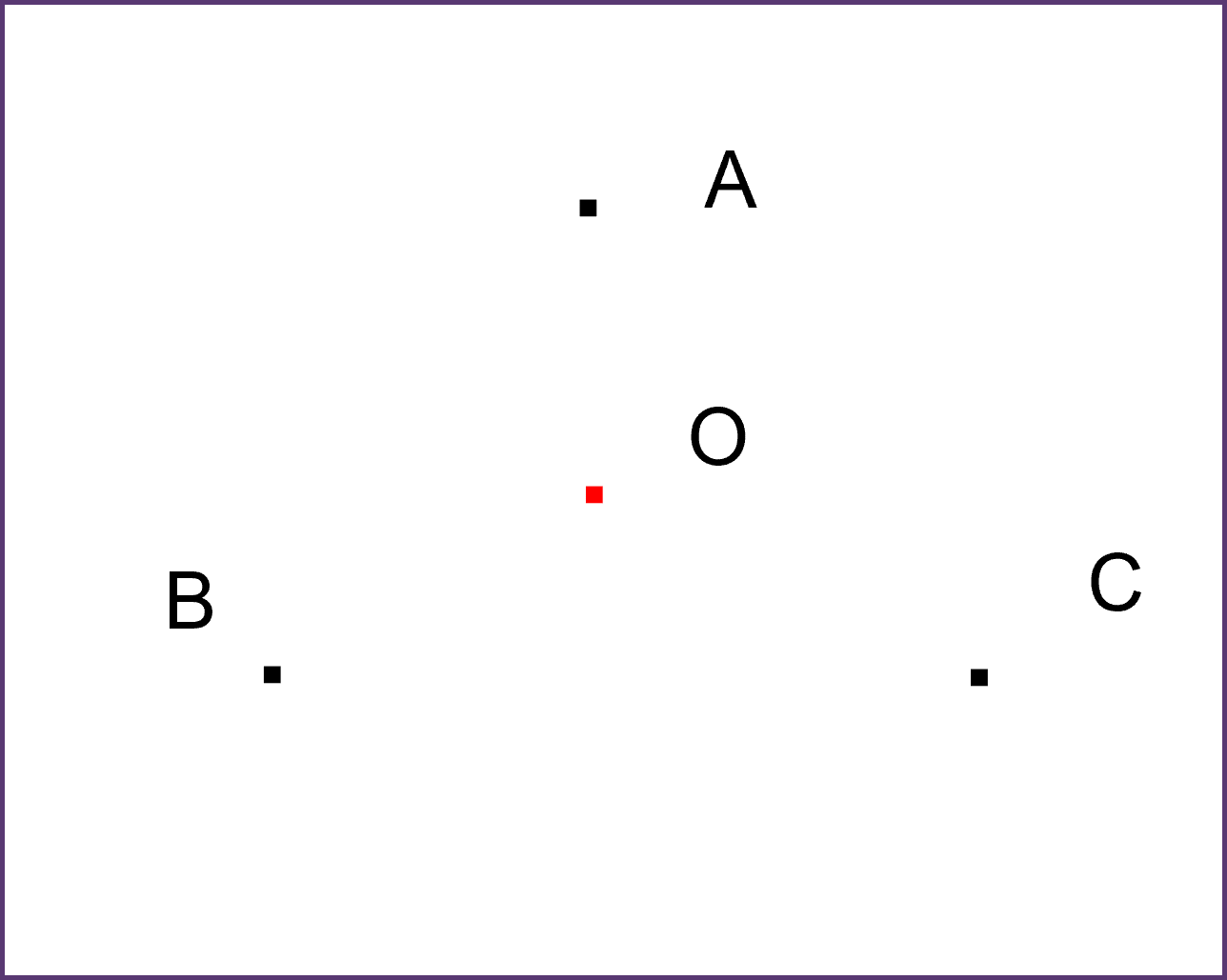}}} % 图片内容
    {\centering An example of 4 points that any kNN classifier cannot shatter. \label{fig:kNNexample}} % 图题
    {} % 空的备注
\end{figure}

Denote $\Delta_N = \sup\limits_{z\in\mathcal{Z}, x\in\mathcal{X}}\left|\hat{U}^{kNN}(z,x) -U^{kNN}(z,x)\right|$, set $\epsilon_N = M\sqrt{\frac{\log N}{N}}$, $M$ is big enough that $M^2/c_2 > v$. By Lemma~\ref{lem::VC Inequality},
\begin{equation}\nonumber
    \begin{aligned}
        \mathbb{P}(\Delta_N>\epsilon_N) &\le c_1(v+1)(2N)^v\exp\left(-\frac{N}{c_2} M^2 \frac{\log N}{N}\right)=c_1(v+1)(2N)^vN^{-M^2/c_2}=c_1(v+1)2^vN^{v-M^2/c_2}.
    \end{aligned}
\end{equation}

Since $M^2/c_2>v$, $c_1(v+1)2^vN^{v-M^2/c_2}$ is decreasing with respect to $N$. Therefore, $\forall \epsilon>0$, $\exists N_0=\lceil \exp(\frac{\log \epsilon - \log c_1 - \log (v+1) - v\log2}{v-M^2/c_2}) \rceil$, when $N>N_0$, we have, $\log c_1 + \log (v+1) + v\log2 + (v-M^2/c_2)\log N < \log \epsilon$, namely, $\mathbb{P}(\Delta_N>\epsilon_N)<c_1(v+1)2^vN^{v-M^2/c_2} < \epsilon$. And thus, 
$$\sup\limits_{z\in\mathcal{Z}, x\in\mathcal{X}}\left|\hat{U}^{kNN}(z,x) -U^{kNN}(z,x)\right| = O_p\left(\sqrt{\frac{\log N}{N}}\right).$$

Likewise, the VC dimension of $F^{LSA}_{z,x,h}(Z, X) = \mathbb{I}\left\{\Vert (Z, X) - (z, x)\Vert \le h\right\}\cdot\mathbb{I}\{\psi(z, Y)\le 0\}$ and $F^{kNN}_{z,x,k,N}(Z,X) = \mathbb{I}\{(Z,X)\text{ is a kNN of }(z,x)\}\cdot\mathbb{I}\{\psi(z, Y)\le0\}$ are less than $d^{z_1} + d^x + 1$. Former proof process can also be applied to $\left|\hat{V}(z,x) -V(z,x)\right|$. So, 
\begin{equation}\nonumber
    \begin{aligned}
    \sup_{z\in\mathcal{Z}, x\in\mathcal{X}}\left|\hat{V}^{LSA}(z,x) -V^{LSA}(z,x)\right|&= O_p\left(\sqrt{\frac{\log N}{Nh_N}}\right),\\
    \sup_{z\in\mathcal{Z}, x\in\mathcal{X}}\left|\hat{V}^{kNN}(z,x) -V^{kNN}(z,x)\right|&= O_p\left(\sqrt{\frac{\log N}{N}}\right).\qquad \text{Q.E.D.}
    \end{aligned}
\end{equation}

\subsection{Proof of Proposition~\ref{prop::Frac Consistency}}

The consistency to be proved can be decomposed into,
\begin{equation}
    \begin{aligned}
        &\left|\frac{\hat{V}(z,x)}{\hat{U}(z,x)} - \frac{V(z,x)}{U(z,x)}\right| 
        = \left| V(z, x)\left(\frac{1}{\hat{U}(z,x)} - \frac{1}{U(z,x)}\right) + \frac{\hat{V}(z, x)-V(z,x)}{\hat{U}(z, x)}\right|\\
        \le& \frac{V(z,x)}{U(z,x)\hat{U}(z,x)}\left|\hat{U}(z,x) - U(z,x)\right| + \left|\frac{\hat{V}(z, x)-V(z,x)}{\hat{U}(z, x)}\right|\\
        \le&\left|\frac{\hat{U}(z,x) - U(z,x)}{\hat{U}(z, x)}\right| + \left|\frac{\hat{V}(z, x)-V(z,x)}{\hat{U}(z, x)}\right|
        = \Delta_1 + \Delta_2,
    \end{aligned}
\end{equation}
where the first inequality comes from the triangle inequality, and the second inequality holds from the fact that $0\le V(z, x) \le  U(z,x) \le 1$.

We prove the proposition by separately stating the consistency of $\Delta_1$ and $\Delta_2$. Start with LSA approximation. As a sphere classifier, $U^{LSA}(z,x)$ is the probability that $(Z,X)$ falls into the sphere with $(z, x)$ as the center and $h_N$ as the radius. From Assumption \ref{ass::regularity}, we know that $ U^{LSA}(z,x) \ge c_4 h_N^d$, where $c_4$ is a constant related to the infimum of the distribution and the volume of the sphere, $d = d^{z_1} + d^x$.

And from Proposition~\ref{prop::DNConsistency} we know that, $\sup\limits_{z\in\mathcal{Z}, x\in\mathcal{X}}\left|\hat{U}^{LSA}(z,x) -U^{LSA}(z,x)\right| = O_p\left(\sqrt{\frac{\log N}{Nh_N}}\right)$. Let $\epsilon_N = \sqrt{\frac{\log N}{N h_N}}$. Under the condition $0 < \delta_{LSA}(2d+1) < 1$, the ratio satisfies:
$$\frac{\epsilon_N}{h_N^d} = \frac{\sqrt{\log N}}{\sqrt{N h_N^{2d+1}}} = \frac{\sqrt{\log N}}{\sqrt{C^{2d+1} N^{1-\delta_{LSA}(2d+1)}}} \rightarrow 0 \text{ as } N \rightarrow \infty \text{.}$$
This establishes that $\epsilon_N = o(h_N^d)$ uniformly over $Z \times X$. By Proposition~\ref{prop::DNConsistency}, we have $\sup_{z,x} |\hat{U}^{LSA}(z,x) - U^{LSA}(z,x)| \le \epsilon_N$ almost surely. Therefore, for sufficiently large $N$, the triangle inequality yields $\hat{U}^{LSA}(z,x) \ge U^{LSA}(z,x) - \epsilon_N \ge c_4h_N^d - o(h_N^d) \ge \frac{1}{2}c_4h_N^d$ uniformly over $Z \times X$ almost surely.

% Namely, 

% $$\mathbb{P}\left(\sup_{z\in\mathcal{Z}, x\in\mathcal{X}}\left|\hat{U}^{LSA}(z,x) -U^{LSA}(z,x)\right| > M\sqrt{\frac{\log N}{Nh_N}}\right)\xrightarrow[N\to\infty]{} 0.$$

% And since,
% \begin{equation}\nonumber
%     \begin{aligned}
%         &\inf_{z\in\mathcal{Z}, x\in\mathcal{X}}\hat{U}^{LSA}(z,x) -U^{LSA}(z,x)\le\hat{U}^{LSA}(z,x) -U^{LSA}(z,x) \le \left|\hat{U}^{LSA}(z,x) -U^{LSA}(z,x)\right|\\
%         \le &\sup_{z\in\mathcal{Z}, x\in\mathcal{X}}\left|\hat{U}^{LSA}(z,x) -U^{LSA}(z,x)\right| =O_p\left(\sqrt{\frac{\log N}{Nh_N}}\right),
%     \end{aligned}
% \end{equation}

% we have,
% \begin{equation}\nonumber
%     \mathbb{P}\left(\inf_{z\in\mathcal{Z}, x\in\mathcal{X}}\hat{U}^{LSA}(z,x) >c_4h_N^d + M\sqrt{\frac{\log N}{Nh_N}}\right)\xrightarrow[N\to\infty]{}0.
% \end{equation}

% And since in Lemma \ref{lem::LSA}, we assume that $ 0<\delta_{LSA}(2d+1)<1$ and $h_N=CN^{-\delta_{LSA}}$, we have $\sqrt{\frac{\log N}{Nh_N}}\frac{1}{h_N^d}=\sqrt{\frac{\log N}{CN^{1-\delta_{LSA}(1+2d)}}}.$ Based on the known asymptotic relationship between power functions and logarithmic functions, we know that $O\left(\sqrt{\frac{\log N}{Nh_N}}\right) = O\left(h_N^d\right)$. Hence, $\inf_{z\in\mathcal{Z}, x\in\mathcal{X}}\hat{U}^{LSA}(z,x) = O_p(h_N^d)$. 

Therefore, we have,
\begin{equation}\nonumber
    \begin{aligned}
        \sup\limits_{z\in\mathcal{Z}, x\in\mathcal{X}}\Delta_1&\le\frac{\sup\limits_{z\in\mathcal{Z}, x\in\mathcal{X}}\left|\hat{V}^{LSA}(z, x)-V^{LSA}(z,x)\right|}{\hat{U}^{LSA}(z, x)}\le\frac{O_p\left(\sqrt{\frac{\log N}{Nh_N}}\right)}{\frac{1}{2}c_4h_N^d}=O_p\left(\sqrt{\frac{\log N}{N^{1-\delta_{LSA}(2d+1)}}}\right)\xrightarrow[N\to\infty]{a.s.}0.
    \end{aligned}
\end{equation}

Similarly, since $\sup\limits_{z\in\mathcal{Z}, x\in\mathcal{X}}\left|\hat{V}^{LSA}(z,x) -V^{LSA}(z,x)\right|= O_p\left(\sqrt{\frac{\log N}{Nh_N}}\right)$, the former proof can be applied to $\Delta_2$. Namely,
\begin{equation}
    \sup_{z\in\mathcal{Z}, x\in\mathcal{X}}\Delta_2 = O_p\left(\sqrt{\frac{\log N}{N^{1-\delta_{LSA}(2d+1)}}}\right)\xrightarrow[N\to\infty]{a.s.}0
\end{equation}

Therefore, we have for LSA,
\begin{equation}\label{eq-sample efficiency-LSA}
    \sup_{z\in\mathcal{Z},x\in\mathcal{X}} |\Delta_1| + |\Delta_2| = O_p\left(\sqrt{\frac{\log N}{N^{1-\delta_{LSA}(2d+1)}}}\right)\xrightarrow[N\to\infty]{a.s.} 0.
\end{equation}

For kNN approximation, the denominator $\hat{U}^{kNN}(z,x)$ exactly equals $k/N$. And since in Lemma \ref{lem::kNN}, we assume that $0<C<1, 0.5<\delta_{kNN}<1$, and $k=\lceil CN^{\delta_{kNN}}\rceil$. We have that $\hat{U}^{kNN}(z,x)=k/N = O(N^{\delta_{kNN}-1})$. Then,
\begin{equation}
        \sup_{z\in\mathcal{Z},x\in\mathcal{X}}\Delta_1 = \frac{\sup_{z\in\mathcal{Z}, x\in\mathcal{X}}|\hat{V}^{kNN}(z, x)-V^{kNN}(z,x)|}{k/N}=\frac{O_p\left(\sqrt{\frac{\log N}{N}}\right)}{k/N}=O_p\left(\sqrt{\frac{\log N}{N^{2\delta_{kNN}-1}}}\right).
\end{equation}

Similarly,
\begin{equation}
    \sup_{z\in\mathcal{Z},x\in\mathcal{X}}\Delta_2 = O_p\left(\sqrt{\frac{\log N}{N^{2\delta_{kNN}-1}}}\right).
\end{equation}

And since $2\delta_{kNN}-1 > 0$, $\sqrt{\frac{\log N}{N^{2\delta_{kNN}-1}}} \xrightarrow[]{N\to\infty}0$. Therefore, we have for kNN,
\begin{equation}\label{eq-sample efficiency-kNN}
    \sup_{z\in\mathcal{Z},x\in\mathcal{X}} |\Delta_1| + |\Delta_2| = O_p\left(\sqrt{\frac{\log N}{N^{2\delta_{kNN}-1}}}\right)\xrightarrow[N\to\infty]{a.s.} 0.
\end{equation}
That concludes the proof for kNN approximation. In conclusion,
\begin{equation}
\sup_{z\in\mathcal{Z},x\in\mathcal{X}}\left|\frac{\hat{V}(z,x)}{\hat{U}(z,x)} - \frac{V(z,x)}{U(z,x)}\right| \xrightarrow[N\to\infty]{a.s.} 0,
\end{equation}
for both LSA and kNN approximation, which concludes the proof. Q.E.D.

To further connect Proposition~\ref{prop::Frac Consistency} and Theorem~\ref{theo::chance-consistency}, we introduce two lemmas:
\begin{lemma}
\label{lem::Probability Continuity}
    $P(\psi(z, Y)\le 0\mid Z=z', X=x')$ is continuous on $z'$ and $x'$ jointly. Furthermore, if $I: \mathcal{Z}\times\mathcal{Y}\to \mathbb{R}$ is $L_I-$Lipschitz continuous, the mapping $(z,x) \to f(z,x)$ is $L_f-$Lipschitz continuous under the 1-Wasserstein distance, and for any $(z, x)$, the probability density function of $\psi(z, Y)$ where $Y\sim f(z,x)$ has an consistent upper bound, then $P(\psi(z, Y)\le 0\mid Z=z, X=x)$ is Lipschitz continuous on $z$ and $x$ jointly.
\end{lemma}

\begin{lemma}
    \label{lem::limitation-r}
    Denote $\mathcal{B}_r(z,x)$ as the minimal sphere with $(z,x)$ as the center and a radius of $r$, and $\mathcal{U}_r(z,x)$ is an arbitrary subset of $\mathcal{B}_r(z,x)$ with non-zero Lebesgue measure, we have,
    \begin{equation}
        \begin{aligned}
            &\lim_{r\to 0}\left\lvert P[\psi(z, Y)\le 0 \mid (Z, X) \in \mathcal{U}_r(z,x)]-P[\psi(z, Y)\le 0 \mid Z=z,X=x]\right\rvert = 0.
        \end{aligned}
    \end{equation}
\end{lemma}

\subsubsection{Proof of Lemma \ref{lem::Probability Continuity}}

We have $P[\psi(z, Y)\le 0 \mid Z=z', X=x'] = \int_{\mathbb{U}(z)}f(y|z',x')dy$, where $\mathbb{U}(z)=\{y\mid \psi(z, y)\le 0\}$, and $f(y|z',x')$ is the conditional probability density of random variable $Y$ conditioned on $(Z=z',X=x')$.

Denote $f(z',x')$ as the joint probability density of random variable $(Z,X)$, then we have $f(y|z',x')=\frac{f(y,z',x')}{f(z',x')}$. Since $f(y,z',x')$ is continuous on $y$, $z'$ and $x'$ jointly and $f(z',x')$ is continuous on $z'$ and $x'$ jointly, we have that $f(y|z',x')$ is continuous on $y$, $z'$ and $x'$ jointly.

Suppose Assumption \ref{ass::regularity} holds. Since $\mathcal{X}$, $\mathcal{Y}$ and $\mathcal{Z}$ are compact and bounded, $\psi(z, y)$ is continuous on $z$ and $y$ jointly,  $\mathbb{U}(z)=\{y\mid \psi(z, y)\le 0\}$ is compact and bounded. Because $f(y|z',x')$ is continuous on $y$, $z'$ and $x'$ jointly and $\mathbb{U}(z)$, $\mathcal{Z}$ and $\mathcal{X}$ are compact and bounded sets, we can conclude that $f(y|z',x')$ is consistently continuous on $\mathbb{U}(z)\times\mathcal{Z\times\mathcal{X}}$. Denote $V_d(\mathbb{U}(z))$ as the Lebesgue measure of $\mathbb{U}(z)$. Since $0\le V_d(\mathbb{U}(z)) < \infty$, for any $z\in \mathcal{Z}$, any $\epsilon>0$ and any $(z',x')\in\mathcal{Z}\times\mathcal{X}$, there exists $\delta>0$, such that
\begin{equation*}
    \begin{aligned}
        \forall (z'',x'')\in \mathcal{Z}\times\mathcal{X}\,s.t.\,\lVert(z'',x'')-(z',x')\lVert&<\delta,\\
        \forall y\in\mathbb{U}(z),
        V_d(\mathbb{U}(z))\left|f(y\mid z'',x'')-f(y\mid z',x')\right|&<\epsilon.
    \end{aligned}
\end{equation*}

Further, we have
\begin{equation}
     \begin{aligned}
         &\lvert P[\psi(z, Y)\le 0 \mid Z=z'',X=x'']-P[\psi(z, Y)\le 0 \mid Z=z',X=x']\rvert\\
         =&\left|\int_{\mathbb{U}(z)}f(y\mid z'',x'')\,dy-\int_{\mathbb{U}(z)}f(y\mid z',x')\,dy\right|\\
         \le& \int_{\mathbb{U}(z)}\left|f(y\mid z'',x'')-f(y\mid z',x')\right|\,dy
         \le V_d(\mathbb{U}(z))\sup_{y\in \mathbb{U}(z)}\left|f(y\mid z'',x'')-f(y\mid z',x')\right|
         < \epsilon,
     \end{aligned}
\end{equation}
which means $P[\psi(z, Y)\le 0 \mid Z=z',X=x']$ is continuous on $z'$ and $x'$ jointly. Further, if $f(y|z',x')$ is consistently Lipschitz continuous on $z'$ and $x'$ jointly, for any $z\in \mathcal{Z}$, any $\epsilon>0$ and any $(z',x')\in\mathcal{Z}\times\mathcal{X}$, there exists $L>0$, such that $\forall (z'',x'')\in \mathcal{Z}\times\mathcal{X}, \forall y\in\mathbb{U}(z)$, $V_d(\mathbb{U}(z))\left|f(y\mid z'',x'')-f(y\mid z',x')\right|<L\left|(z'',x'')- (z',x')\right|$. Further, we have
\begin{equation}
     \begin{aligned}
         &\lvert P[\psi(z, Y)\le 0 \mid Z=z'',X=x'']-P[\psi(z, Y)\le 0 \mid Z=z',X=x']\rvert\\
         =&\left|\int_{\mathbb{U}(z)}f(y\mid z'',x'')\,dy-\int_{\mathbb{U}(z)}f(y\mid z',x')\,dy\right|\\
         \le& \int_{\mathbb{U}(z)}\left|f(y\mid z'',x'')-f(y\mid z',x')\right|\,dy\\
         \le& V_d(\mathbb{U}(z))\sup_{y\in \mathbb{U}(z)}\left|f(y\mid z'',x'')-f(y\mid z',x')\right|\\
         <& V_d(\mathbb{U}(z))L\left|(z'',x'')- (z',x')\right|
         \le V_d(\mathcal{Y})L\left|(z'',x'')- (z',x')\right|,
     \end{aligned}
\end{equation}
which means $P[\psi(z, Y)\le 0 \mid Z=z',X=x']$ is Lipschitz continuous on $z'$ and $x'$ jointly.

We go on to discuss the continuity about $g(z\mid X=x) = P(\psi(z, Y)\le 0\mid Z=z, X=x)$, the situation where the parameter of the inner function varies with the factors that influence the distribution, under the additional assumptions given in the lemma. $\forall (z_1, x_1), (z_2, x_2)\in\mathcal{Z\times\mathcal{X}}$. Let $d = \|(z_1, x_1) - (z_2, x_2)\|$. From the Lipschitz continuous Wasserstein distance assumption we have that, there exists a coupling $\gamma\in \Gamma(f(z_1, x_1), f(z_2, x_2))$, that
\begin{equation}\nonumber
    \mathbb{E}_{(Y_1, Y_2) \sim \gamma} \|Y_1 - Y_2\| \le W_1(f(z_1, x_1), f(z_2, x_2)) \le L_f d,
\end{equation}
which comes from the definition of Wasserstein distance. Then we have
\begin{equation}\nonumber
    \begin{aligned}
        &\left|g(z_1\mid X=x_1) - g(z_2\mid X=x_2) \right|
        =\left|E_{\gamma}\left[\mathbb{I}\{\psi(z_1, Y_1\}\le 0)  - \mathbb{I}\{\psi(z_2, Y_2)\le 0\}\right] \right|
        \le E_{\gamma}\left|\mathbb{I}(\psi(z_1, Y_1)\le 0)  - \mathbb{I}\{\psi(z_2, Y_2)\le 0\}\right|.
    \end{aligned}
\end{equation}

Note that for any $y_1, y_2\in \mathcal{Y}$, $\left|\mathbb{I}\{\psi(z_1, y_1)\le 0\}  - \mathbb{I}\{\psi(z_2, y_2)\le 0\}\right|$ takes value from $\{0, 1\}$. When it equals 1, one of the following happens:
%\begin{enumerate}
    (1) \(\psi(z_1, y_1) \le 0\) and \(\psi(z_2, y_2) > 0\);
    (2) \(\psi(z_1, y_1) > 0\) and \(\psi(z_2, y_2) \le 0\).
%\end{enumerate}

And since we assume function I has $L_I-$lipschitz continuity, we have,
\begin{equation}\label{eq-ILipCont}
        |\psi(z_1, y_1) - \psi(z_2, y_2)| 
        \le L_I \|(z_1, y_1) - (z_2, y_2)\| \le L_I (\|z_1 - z_2\| + \|y_1 - y_2\|) 
        \le L_I d + L_I \|y_1 - y_2\|.
\end{equation}

If $\left|\mathbb{I}\{\psi(z_1, y_1)\le 0\}  - \mathbb{I}\{\psi(z_2, y_2)\le 0\}\right| = 1$, then, $\min\left\{ |\psi(z_1, y_1)|, |\psi(z_2, y_2)| \right\} \le L_I (d + \|y_1 - y_2\|)$. Because if both $|\psi(z_1, y_1)|$ and $|\psi(z_2, y_2)|$ are greater than \(L_I (d + \|y_1 - y_2\|)\), we have,
$$|\psi(z_1, y_1) - \psi(z_2, y_2)| \ge \left||\psi(z_1, y_1)| - |\psi(z_2, y_2)|\right| > L_I (d + \|y_1 - y_2\|),
$$
which contradicts inequality (\ref{eq-ILipCont}). And thus,
\begin{equation}\nonumber
        \left|\mathbb{I}\{\psi(z_1, y_1)\le 0\}  - \mathbb{I}\{\psi(z_2, y_2)\le 0\}\right|
        \le\left|\mathbb{I}\{|\psi(z_1, y_1)| \le L_I(d+\|y_1-y_2\|)\}\right|  + \left|\mathbb{I}\{|\psi(z_2, y_2)|\le L_I(d+\|y_1-y_2\|)\}\right|,
\end{equation}
hence we have,
\begin{equation}\nonumber\label{eq-g-decomp}
        \left|g(z_1\mid X=x_1) - g(z_2\mid X=x_2) \right|
        \le E_{\gamma}\left[\mathbb{I}\{|\psi(z_1, Y_1)|\le L_I(d+\|Y_1-Y_2\|)\}\right] + E_{\gamma}\left[\mathbb{I}\{|\psi(z_2, Y_2)|\le L_I(d+\|Y_1-Y_2\|)\}\right].
\end{equation}

We then analyze the first term, since we assume the probability density function of $\psi(z, Y)$ where $Y\sim f(z,x)$ has an consistent upper bound, the random variable $\psi(z_1, Y_1)$ has a density of $p_{z_1, x_1}(t) \le M$. For any non-negative random variable $T$, $\mathbb{P}\left(|\psi(z_1, Y_1)| \le T \mid T=t\right) = \mathbb{P}(|\psi(z_1, Y_1)|\le t) = \int_{-t}^t p_{z_1, x_1}(s) ds \le 2Mt$, and thus $\mathbb{P}\left(|\psi(z_1, Y_1)| \le T\right) \le 2ME[T]$. Let $T = L_I(d + \|Y_1-Y_2\|)$, then we have,
\begin{equation}
    \mathbb{P}\left(\mathbb{I}\{|\psi(z_1, Y_1)|\le L_I(d+\|Y_1-Y_2\|)\}\right) \le 2ME_{\gamma}\left[L_I(d+\|Y_1-Y_2\|)\right].
\end{equation}

Similarly, the second term also satisfy this bound. Therefore,
\begin{equation}
    \left|g(z_1\mid X=x_1) - g(z_2\mid X=x_2) \right| \le 2\cdot2ML_IE_{\gamma}[d+\|Y_1-Y_2\|].
\end{equation}

From the assumption of $L_f-$Lipschitz Wasserstein distance continuity, we have that, $E_{\gamma}[\|Y_1-Y_2\|]\le L_f d$. So,
\begin{equation}\nonumber
    E_{\gamma}[d+\|Y_1-Y_2\|]\le d+L_f d=d(1+L_f).
\end{equation}

And hence, $\left|g(z_1\mid X=x_1) - g(z_2\mid X=x_2) \right| \le 4ML_I(1+L_f)d$, namely,
\begin{equation}
    {\left|g(z_1\mid X=x_1) - g(z_2\mid X=x_2) \right|}/{\|(z_1, x_1) - (z_2, x_2)\|} \le 4ML_I(1+L_f),
\end{equation}
which concludes the proof. The Lipschitz continuity constant for function $g$ is $4ML_I(1+L_f)$. Q.E.D.

\subsubsection{Proof of Lemma \ref{lem::limitation-r}}

We have
\begin{equation}
    P[\psi(z', Y)\le 0 \mid (Z, X) \in \mathcal{U}_r(z,x)] = \frac{\int_{\mathcal{U}_r(z,x)} f(z',x')P[\psi(z, Y)\le 0 \mid Z=z',X=x']\,dz'\,dx'}{\int_{\mathcal{U}_r(z,x)} f(z',x')\,dz'\,dx'},
\end{equation}
where $f(z',x')$ is the joint probability density of random variable $(Z,X)$.

Denote $u(z',x',z) = f(z',x')P[\psi(z, Y)\le 0 \mid Z=z',X=x']$. Since $f(z',x')$ and $P[\psi(z, Y)\le 0 \mid Z=z',X=x']$ are both continuous on $z'$ and $x'$ jointly, $u(z',x',z)$ is also continuous on $z'$ and $x'$ jointly. Denote
\begin{equation}
    \begin{aligned}
        \omega_u(r,z,x) &= \sup_{\lVert(z',x')-(z,x)\rVert<r}\left|u(z',x',z)-u(z,x,z)\right|,\\
        \omega_p(r,z,x) &= \sup_{\lVert(z',x')-(z,x)\rVert<r}\left|f(z',x')-f(z,x)\right|.
    \end{aligned}
\end{equation}
Because $u(z',x',z)$ and $f(z',x')$ are both continuous on $z'$ and $x'$ jointly, we have $\lim_{r \to 0}\omega_u(r,z,x)=0$ and $\lim_{r \to 0}\omega_f(r,z,x)=0$ for any $(z,x)$. For any $r>0$,
\begin{equation}
    \int_{\mathcal{U}_r(z,x)}u(z',x',z)\,dz'\,dx' = u(z,x,z)V_d(\mathcal{U}_r(z,x))+\int_{\mathcal{U}_r(z,x)}(u(z',x',z)-u(z,x,z))\,dz'\,dx',
\end{equation}
where $V_d(\mathcal{U}_r(z,x))$ is the Lebesgue measure of $\mathcal{U}_r(z,x)$.

Since $\omega_u(r,z,x) = \sup_{\lVert(z',x')-(z,x)\rVert<r}\left|u(z',x',z)-u(z,x,z)\right|$, we have
\begin{equation*}
    -\omega_u(r,z,x)V_d(\mathcal{U}_r(z,x))\le\int_{\mathcal{U}_r(z,x)}(u(z',x',z)-u(z,x,z))\,dz'\,dx'\le\omega_u(r,z,x)V_d(\mathcal{U}_r(z,x)).
\end{equation*}

Therefore, we have
\begin{equation}
    \begin{aligned}
        u(z,x,z)-\omega_u(r,z,x) &\le\frac{\int_{\mathcal{U}_r(z,x)}u(z',x',z)\,dz'\,dx'}{V_d(\mathcal{U}_r(z,x))} \le u(z,x,z)+\omega_u(r,z,x),\\
        f(z,x)-\omega_p(r,z,x) &\le\frac{\int_{\mathcal{U}_r(z,x)}f(z',x')\,dz'\,dx'}{V_d(\mathcal{U}_r(z,x))} \le f(z,x)+\omega_p(r,z,x).
    \end{aligned}
\end{equation}

Since $\lim_{r \to 0}\omega_p(r,z,x)=0$, we can take sufficiently small $r$, such that $\omega_p(r,z,x)<f(z,x)$ and for such $r$ we have
\begin{equation}\nonumber
    \frac{u(z,x,z)-\omega_u(r,z,x)}{f(z,x)+\omega_p(r,z,x)}\le P[\psi(z, Y)\le 0 \mid (Z, X) \in \mathcal{U}_r(z,x)]=\frac{\int_{\mathcal{U}_r(z,x)}u(z',x',z)\,dz'\,dx'}{\int_{\mathcal{U}_r(z,x)}f(z',x')\,dz'\,dx'}\le\frac{u(z,x,z)+\omega_u(r,z,x)}{f(z,x)-\omega_p(r,z,x)}.
\end{equation}

Since $\lim_{r \to 0}\omega_u(r,z,x)=0$ and $\lim_{r \to 0}\omega_p(r,z,x)=0$, we have $\lim_{r \to 0}\frac{u(z,x,z)-\omega_u(r,z,x)}{f(z,x)+\omega_p(r,z,x)}=\frac{u(z,x,z)}{f(z,x)}$ and $\lim_{r \to 0}\frac{u(z,x,z)+\omega_u(r,z,x)}{f(z,x)-\omega_p(r,z,x)}=\frac{u(z,x,z)}{f(z,x)}$. Further, by Squeeze Theorem, we reach conclusion
\begin{equation}\nonumber
\begin{aligned}
        &\lim_{r \to 0}P[\psi(z, Y)\le 0 \mid (Z, X) \in \mathcal{U}_r(z,x)]={u(z,x,z)}/{f(z,x)}
        =P[\psi(z, Y)\le 0 \mid Z=z,X=x],\\
    &\lim_{r\to 0}\left|P[\psi(z, Y)\le 0 \mid (Z, X) \in \mathcal{U}_r(z,x)] - P[\psi(z, Y)\le 0 \mid Z=z,X=x]\right|=0. \qquad \text{Q.E.D.}
\end{aligned}
\end{equation}

\subsection{Proof of Theorem \ref{theo::chance-consistency}}

We start with the proof of the case with kNN or LSA approximation. From Proposition~\ref{prop::Frac Consistency}, we have already known that in both kNN and LSA, 
\begin{equation}\nonumber
\sup_{z\in\mathcal{Z},x\in\mathcal{X}}\left|\frac{\hat{V}(z,x)}{\hat{U}(z,x)} - \frac{V(z,x)}{U(z,x)}\right| \xrightarrow[N\to\infty]{a.s.} 0.
\end{equation}

Note that $V(z,x)/U(z,x)=\mathbb{P}(\psi(z,Y)\le 0\mid C(z,x) = C(Z,X))$. In the case of both LSA and kNN, the cluster $(z,x)$ is in is sphere (demonstrated in the proof of Proposition~\ref{prop::DNConsistency}).
In the case of LSA, the radius $h_N = CN^{-\delta_{LSA}}$ is approaching 0 automatically with N increasing. In the case of kNN, the distance from $(z, x)$ to the k-th farthest historical data point decreases to 0 almost surely when $k/N\to 0$ and $k/\log N\to \infty$, and setting $k=CN^{\delta_{kNN}}$ where $0<\delta_{kNN}<1$ automatically satisfies the requirement. (Theorem 1.5 in \citeapp{biau2015lectures}). So, 
\begin{equation}\label{eq-theo1bd}
    \begin{aligned}
        &\sup_{z\in\mathcal{Z}, x\in \mathcal{X}} \left|\frac{\hat{V}(z,x)}{\hat{U}(z,x)} - P[\psi(z, Y)\le 0 \mid Z=z,X=x]\right|\\
        \le &\sup_{z\in\mathcal{Z}, x\in \mathcal{X}} \left\{\left|\frac{\hat{V}(z,x)}{\hat{U}(z,x)} - \frac{V(z,x)}{U(z,x)}\right| + \left|\frac{V(z,x)}{U(z,x)} - P[\psi(z, Y)\le 0 \mid Z=z,X=x] \right|\right\}\\
        =&\sup_{z\in\mathcal{Z}, x\in \mathcal{X}} \left\{\left|\frac{\hat{V}(z,x)}{\hat{U}(z,x)} - \frac{V(z,x)}{U(z,x)}\right| + \left|P[\psi(z, Y)\le 0 \mid (Z, X) \in \mathcal{U}_{r_c}(z,x) - P[\psi(z, Y)\le 0 \mid Z=z,X=x] \right|\right\},
    \end{aligned}
\end{equation}
where $r_c$ is the radius determined by the cluster. Combining Proposition~\ref{prop::Frac Consistency} and Lemma~\ref{lem::limitation-r}, we have that both terms in (\ref{eq-theo1bd}) converges to 0 when $N\to\infty$ almost surely, which concludes the proof of the case with kNN or LSA approximation. The sample efficiency for these two approximations follow the same sample efficiency of $\sup_{z\in\mathcal{Z}, x\in \mathcal{X}} \left|\frac{\hat{V}(z,x)}{\hat{U}(z,x)} - \frac{V(z,x)}{U(z,x)}\right|$, which is given in (\ref{eq-sample efficiency-LSA}) and (\ref{eq-sample efficiency-kNN}).

For CART approximation, the classifier no longer has finite VC dimension, and hence needs different proof process. From former proof and assumptions, we have the following facts:
\begin{enumerate}
    \item The trained tree T is a regular, random-split, honest tree;
    
    \item $(z_i, x_i, \mathbb{I}\{\psi(z, y_i)\le0\}$ are i.i.d with $(z_{1i}, x)$ uniform in $[0,1]^{d^{z_1}+d^x}$;

    \item $E\left[\mathbb{I}\left\{\psi(z, y_i)\le0\mid X=x,Z_1=z_1\right\}\right]$ is Lipschitz continuous;

    \item $\sup_{z ,x} E\left[exp(\lambda |\mathbb{I}\{\psi(z, Y)\le0\} - E[\mathbb{I}\{\psi(z, Y)\le0\}\mid X=x, Z_1=z_1]|) \mid X=x, Z_1=z_1\right] < \infty$ naturally holds since $0\le\mathbb{I}\{\psi(z, Y)\le0\} \le 1$;

    \item $\log N/CN^{\delta_{cart}}\xrightarrow[]{N\to\infty}0$, $N/CN^{\delta_{cart}}\xrightarrow[]{N\to\infty}0$.
\end{enumerate}

With all these, the assumptions of Lemma 7 in \citeapp{bertsimas2019predictions} are all satisfied, and thus,
\begin{equation}
    \sup_{x\in\mathcal{X},z\in\mathcal{Z}}\left|\hat{g}_N^{CART}(z\mid X=x)-g(z\mid X=x)\right|\xrightarrow[w.p.]{N\to\infty}0.
\end{equation}

Therefore, approximation $\hat{g}_N^{CART}(z\mid X=x)$ has weakly uniform consistency. Q.E.D.

\subsection{Proof to Theorem \ref{theo::asym}}

Theorem \ref{theo::asym} is the decision-dependent version of Theorem 1 to Theorem 3 in \citeapp{lin2022data}. We first demonstrate the feasibility, as in Theorem 1 and Proposition 2 in \citeapp{lin2022data}. The proof needs the following requirements:

\begin{enumerate}
    \item i.i.d and no-atom and  assumption, namely for any given $x$ and $z$, $\mathbb{P}(\psi(z, Y)=0 \mid X=x)=0$. The i.i.d assumption is given in Assumption \ref{ass::regularity}.2. And as Assumption \ref{ass::regularity}.3 requires, the joint distribution of $(z, x, y)$ is continuous and positively bounded. And as required by Assumption \ref{ass::continuous}, function $\phi$ is also equi-continuous. The no-atom assumption holds naturally;

    \item Uniform consistency. That is, $\sup_{z, x} |\hat{g}(z\mid X=x) - g(z\mid X=x)| \to 0$ almost surely (strong uniform consistency) or with probability (weak uniform consistency), which has been proven in Theorem \ref{theo::chance-consistency};

    \item The continuity of the original chance function $g(z\mid X=x)$ with respect to $z$, which has been proven in Lemma \ref{lem::Probability Continuity};

    \item Slater condition. As required in Assumption \ref{ass::slater}.
\end{enumerate}

With all these conditions on hand, following the same proof routine of \citeapp{lin2022data}, we have $\forall x\in\mathcal{X}$, for some optimal solution $z^*$, there exists a sequence $z^n=(z_1^n, z_2^n)\in \hat{P}(x)$ that $z^n\rightarrow z^*$ almost surely (or with probability).

The rest part of Theorem \ref{theo::asym} is the decision-dependent version of Theorem 3 in \citeapp{lin2022data}. The proof needs the following additional requirements: (1) The decision space is compact. As required by Assumption \ref{ass::regularity}.1; (2) Theorem 2 of \citeapp{lin2022data}: For any given $x$, $\sup_{z} |\hat{L}(z\mid X=x) - L(z\mid X=x)|\to 0, w.p.\ 1$ (in probability, resp.), which is given in Lemma \ref{lem::kNN} to Lemma \ref{lem::LSA}.

With all these conditions on hand, following the proof of \citeapp{lin2022data}, we have, the approximated objective $\hat{L}(z^n \mid X = x)$ converges to the true optimal objective $L(z^* \mid X=x)$ almost surely (or with probability), and the true objective $L(z^n \mid X = x)$ converges to the true optimal objective $L(z^* \mid X=x)$ almost surely (or with probability). Q.E.D.

\subsection{Proof to Lemma \ref{lem::cc}}

To prove this lemma, we first prove the following proposition: 

\begin{proposition}\label{prop::cc convex}
    For $N$ convex sets $A_1, A_2,\cdots,A_N$, given that $\forall i,j\in \{1,2,\cdots,N\}$, $A_i \cup A_j$ is convex. Then, $\forall k\in\{1,2,\cdots,N\}$, we have,
\begin{equation}
    \cup_{1\le i_i < i_2<\cdots<i_k\le N}\left[\cap_{i\in \{i_1, i_2,\cdots,i_k\}} A_i \right]
\end{equation}
is convex.
\end{proposition}

\textbf{Proof to Proposition \ref{prop::cc convex}:}

$\forall k\in\{1,2,\cdots,N\}$, $\forall x, y \in \cup_{1\le i_i < i_2<\cdots<i_k\le N}\left[\cap_{i\in \{i_1, i_2,\cdots,i_k\}} A_i \right]$, there exists $1 \le a_1 < a_2<\cdots<a_k\le N$, and $1 \le b_1 < b_2<\cdots<b_k\le N$, that, $x \in \cap_{i\in \{a_1, a_2,\cdots,a_k\}} A_i, y \in \cap_{i\in \{b_1, b_2,\cdots,b_k\}} A_i.$

Denote $\{a_1, a_2,\cdots,a_k\} \cap \{b_1, b_2,\cdots,b_k\}=\{c_1, c_2,\cdots,c_h\}$, where $1\le c_1 < c_2 < \cdots < c_h \le N$, $h \le k$. Then denote $\{a_1, a_2,\cdots,a_k\} / \{c_1, c_2,\cdots,c_h\}=\{d_1, d_2,\cdots,d_{k-h}\}$, where $1\le d_1 < d_2 < \cdots < d_{k-h} \le N$. Similarly denote $\{b_1, b_2,\cdots,b_k\} / \{c_1, c_2,\cdots,c_h\}=\{e_1, e_2,\cdots,e_{k-h}\}$, where $1\le e_1 < e_2 < \cdots < e_{k-h} \le N$. Since both $x$ and $y$ belong to $A_i$, which is convex, $\forall i\in\{c_1, c_2,\cdots,c_h\}$, we have $\forall \lambda\in[0, 1]$, if $z=\lambda x+(1-\lambda)y$, $z\in A_i$, $\forall i\in\{c_1, c_2,\cdots,c_h\}$. Therefore we have $z \in \cap_{i\in \{c_1, c_2,\cdots,c_h\}} A_i$.

And since $\forall j \in \{1,\cdots,k-h\}$, $x\in A_{d_j}, y\in A_{e_j}$, $x, y\in A_{d_j} \cup A_{e_j}$. By the premise that 
$\forall i,j\in \{1,2,\cdots,N\}$, $A_i \cup A_j$ is convex, we know that $z \in A_{d_j} \cup A_{e_j}$. Namely, there exists $f_j \in \{d_j, e_j\}$, $z \in A_{f_j}$.

Since $\{c_1, c_2,\cdots,c_h\} \cap \{d_1, d_2,\cdots,d_{k-h}\}=\emptyset$, $\{c_1, c_2,\cdots,c_h\} \cap \{e_1, e_2,\cdots,e_{k-h}\}=\emptyset$, and $\{d_1, d_2,\cdots,d_{k-h}\} \cap \{e_1, e_2,\cdots,e_{k-h}\}=\emptyset$, we know that $\{c_1, c_2,\cdots,c_h, d_1, d_2,\cdots,d_{k-h}, e_1, e_2,\cdots,e_{k-h}\}$ has no duplication. And therefore, $\{c_1, c_2,\cdots,c_h, f_1, f_2,\cdots,f_{k-h}\}$ has no duplication, which in total has $k$ different members. And thus,
$$z \in \cap_{i\in \{c_1, c_2,\cdots,c_h, f_1, f_2,\cdots,f_{k-h}\}} A_i.$$

So, $z \in \cup_{1\le i_i < i_2<\cdots<i_k\le N}\left[\cap_{i\in \{i_1, i_2,\cdots,i_k\}} A_i \right]$. Therefore, $\cup_{1\le i_i < i_2<\cdots<i_k\le N}\left[\cap_{i\in \{i_1, i_2,\cdots,i_k\}} A_i \right]$ is a convex set. Q.E.D.

Back to Lemma \ref{lem::cc}. Note that for any $z_1\in \mathcal{Z}_1$, each feasible solution $z_2$ should be in at least $\lceil k(1- \alpha) \rceil=K \le k$ sub-feasible regions $S(y_i)$, which are $\{z_2\in \mathcal{Z}_2 \mid \psi(z_1, z_2, y_i) \le 0\}, \forall\ i\in\{1,2,\cdots,N\}$. So the final feasible region is the union of all possible regions that are the intersection of at least $K$ sub-feasible regions, fitting the structure proposed in Proposition \ref{prop::cc convex} .  And since for any $y_1, y_2$, the sub-feasible regions $S(y_1)$ and $S(y_2)$ are nested and convex, the union of them $S(y_1) \cup S(y_2)$ must be convex. According to Proposition \ref{prop::cc convex}, the final feasible region is convex. Q.E.D.

\subsection{Proof to Proposition \ref{prop::var_ss}}

Under the CCW approximation, the approximated chance constraint implies that for those historical points in the cluster, there're at least $\lceil k(1-\alpha) \rceil$ of them that satisfy $l(p, q, d_i) \le -v$. For any $p$, $l(p, q, d_i)$ is a v-shape piecewise linear function for $q$, with a minimum value of $-(p-c)d_i$ at $q=d_i$ and a starting value of 0 at $q=0$.
\begin{equation}
    l(p, q, d_i) = \left\{
    \begin{aligned}
        &-(p-c)q, \quad &&q \le d_i,\\
        &-(p-c)d_i + (c-s)(q-d_i), \quad &&q > d_i.
    \end{aligned}
    \right.
\end{equation}

So for any $d_i$, since we assume there's no chance that $(p-c)d_i = v$, there're two possible situations:

1. $(p-c)d_i < v$: The minimum value of $l(p, q, d_i)$ is still greater than v, and $l(p, q, d_i) \le v$ will never hold for any $q$.

2. $(p-c)d_i > v$: When $v/(p-c) \le q \le d_i + ((p-c)d_i - v)/(c-s)$, $l(p, q, d_i) \le -v$ holds. Notice that for all possible $d_i > v/(p-c)$, the starting point of this interval is the same. And the ending point is increasing with respect to $d_i$.

Denote $d_i^{rhs} = d_i + ((p-c)d_i - v)/(c-s)$. Assume there're $\eta \ge \lceil K(1-\alpha) \rceil$ historical points in the cluster that satisfy $d_i > v/(p-c)$. It's easy to see that for all possible $d_i > v/(p-c)$ in the cluster, their corresponding feasible intervals for $q$ overlap in $[q_{min}, d_{(1)}^{rhs}]$ where $d_{(1)}$ is the minimum $d_i$ that's greater than $q_{min}$.

Next, we start the discussion from the simplest case, assuming that no data points have equal $d_i$, i.e., $d_{(1)} < d_{(2)} < \cdots < d_{(\eta)}$. When $d_{(1)}^{rhs} < q \le d_{(2)}^{rhs}$, the feasible regions corresponding to the remaining $\eta-1$ data points overlap, excluding the data point corresponding to $d_{(1)}$. Following this reasoning, it is not difficult to see that for $q$ to fall within the feasible regions of at least $\lceil K(1-\alpha) \rceil$ data points, $q$ should satisfy $q_{min} \le q \le d_{(\eta - \lceil K(1-\alpha) \rceil + 1)}^{rhs}$.

Considering a more complex situation where some data points have equal $d_i$ values, in the set of $\eta$ data points satisfying the conditions in the cluster, there are $\beta$ distinct $d_i$ values. Assume that after sorting, we have $d_{(1)} < d_{(2)} < \cdots < d_{(\beta)}$, where there are $n_i$ data points corresponding to each $d_{(i)}$, and $\sum_{i=1}^{\beta} n_i = \eta$. In this case, when $d_{(i-1)}^{rhs} < q \le d_{(i)}^{rhs}$, $q$ falls within the feasible regions of $\eta - \sum_{j=0}^{i-1} n_j$ data points in the cluster, where $n_0 = 0$. Therefore, for $q$ to fall within the feasible regions of at least $\lceil k(1-\alpha) \rceil$ data points, $q$ should satisfy $q_{min} \le q \le d_{(n^*)}^{rhs}$, where $n^* = \max\{i \mid \sum_{j=0}^{i-1} n_j \le \eta - \lceil k(1-\alpha) \rceil\}$. This means that there are at most $(\eta - \lceil k(1-\alpha) \rceil)$ $d_i$ values that are less than $d_{n^*}$, which aligns with the simplest case where $n_1 = n_2 = \cdots = n_\eta = 1$. And since there're $N-\eta$ historical points that are either not in the cluster or $d_i < q_{min}$, there are at most $(N - \eta + \eta - \lceil k(1-\alpha) \rceil) = (N - \lceil k(1-\alpha) \rceil)$ $D_i$ values that are less than $d_{n^*}$. In other words, when $D$ is sorted in a descending order, the $\lceil k(1-\alpha) \rceil$-th $D_i^{rhs}$ value corresponds to the required right hand side of the interval.

Based on this conclusion, $q_{min} \le q \le d_{(n^*)}^{rhs}$ is equivalent to the approximated chance constraint (\ref{Est-var}). Q.E.D.

\subsection{Proof to Proposition \ref{prop::optimal_solution}} \label{app::prop optimal solution}
Let $f(q) = -K(p-c)q + \sum_{i=1}^N \mathcal{K}_i(p-s)\max\{q-d_i, 0\}$, which is exactly the objective to minimize in (\ref{PSNP-subproblem}). We prove that $q^* \in \arg\min_q f(q)$. $f(q) = -K(p-c)q + \sum_{\mathcal{K}_i=1} (p-s)\max\{q-d_i,0\}$. Since $\mathcal{K}_i(p-s)$ is a fixed value and $\max\{q-d_i, 0\}$ is a continuous piecewise linear function with respect to $q$, it follows that $f(q)$ is also a continuous piecewise linear function of $q$. Such functions attain their maximum or minimum values only at the endpoints of the interval or at the endpoints of the pieces. Consequently, the minimum value of $f(q)$ must occur at one of the points in $\{d_{(1)}, \cdots, d_{(k)}, \hat{q}\}$. This function is exemplified in Figure \ref{fig:approximated obj demo}.

\begin{figure}[htpb]
    \FIGURE
    {\includegraphics[width=0.35\linewidth]{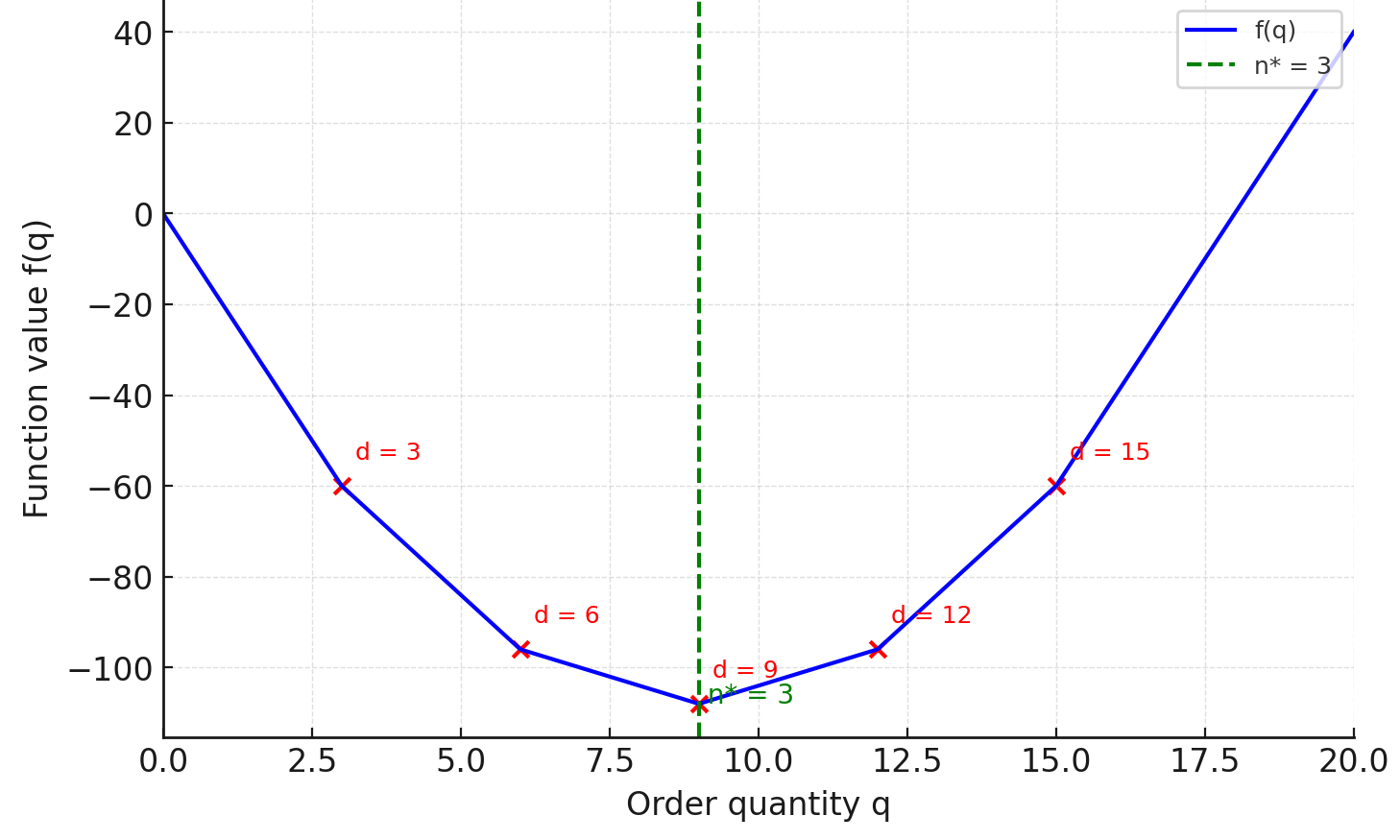}} % 图片内容
    {\centering A demonstration of the approximated objective function in the subproblem. \label{fig:approximated obj demo}} % 图题
    {} % 空的备注
\end{figure}

When $q < d_{(1)}$, all terms in the sum are zero, and thus, $f(q)$ has a minimal gradient with respect to $q$ of $-K(p-c)$. As $q$ increases, more terms in the sum become nonzero. When $d_{(j)} < q < d_{(j+1)}$, we have $\sum_{\mathcal{K}_i=1} (p-s)\max\{q-d_i,0\} = (p-s)\sum_{i \le j} (q - d_{(i)})$, resulting in a gradient of $f(q)$ with respect to $q$ being $j(p-s) - k(p-c)$. Therefore, when $j(p-s) - K(p-c) \le 0$, it follows that $f(d_{(j)}) \ge f(d_{(j+1)})$; when $j(p-s) - k(p-c) > 0$, we have $f(d_{(j)}) < f(d_{(j+1)})$.

Since $n^*=\min\{n \mid n(p-s)-K(p-c) > 0\}$, when $j < n^*$, it holds that $j(p-s) - K(p-c) \le 0$, implying $f(d_{(1)}) \ge f(d_{(2)}) \ge \cdots \ge f(d_{(n^*)})$. When $j > n^*$, we have $j(p-s) - K(p-c) > 0$, leading to $f(d_{(n^*)}) < f(d_{(n^*+1)}) < \cdots < f(d_{(K)}) < f(\hat{q})$. Consequently, among the potential points for achieving the minimum value, $f(d_{(n^*)})$ reaches the minimum value, which completes the proof for $q^* \in \arg\min_q f(q)$. Q.E.D.

\end{document}